\documentclass[11pt]{article}

\usepackage[a4paper,margin=1in]{geometry}
\usepackage{amsmath}

\usepackage{graphicx}
\usepackage{amssymb}
\usepackage{amsmath}

\usepackage{verbatim,booktabs}
\usepackage{subfigure,epsfig,url,psfrag}
\usepackage{float}
\usepackage{mathrsfs}

 \usepackage[usenames,dvipsnames]{pstricks}
 \usepackage{epsfig}
 \usepackage{pst-grad} % For gradients
 \usepackage{pst-plot} % For axes
  \usepackage{enumitem}

 \newcommand\norm[1]{\left\lVert#1\right\rVert}

\makeatletter
\newcommand{\setwindow}[5]{
\def\xmin{#1}%
\def\ymin{#2}%
\def\xmax{#3}%
\def\ymax{#4}%
\pstFPsub\viewingwidth{#3}{#1}%
\pstFPdiv\result{\strip@pt#5}{\viewingwidth}%
\psset{unit=\result pt}}
\makeatother

\usepackage{amsthm}
\newtheorem{theorem}{Theorem}

\newtheorem{proposition}{Proposition}
\newtheorem{lemma}{Lemma}
\newtheorem{example}{Example}

\newtheorem{remark}{Remark}

\def\dim{\mathop{\rm dim}}

\def\st{{\rm s.t.}}

\def\R{{\mathbb R}}

\def\ie{{i.e.,} }
\def\eg{{e.g. }}

\def\G{{\mathcal G}}
\def\S{{\mathcal S}}

\def\C{{\mathcal C}}

\def\N{{\mathcal N}}

\def\D{{\mathcal D}}
\def\U{{\mathcal U}}
\def\M{{\mathcal M}}

\def\01{\ensuremath{0\mathord{-}1}}

\def\E{{\mathcal E}}
\def\V{{\mathcal V}}
\def\Q{{\mathcal Q}}

\newcommand{\bi}{\begin{list}{$\bullet$}{\setlength{\parsep}{0pt}\setlength{\itemsep}{0pt}}}

\def\XXsum#1#2#3{{\setbox0=\hbox{$#1{#2#3}{\sum}$ }
\vcenter{\hbox{$#2#3$ }}\kern-.5\wd0}}

\newcounter{claim} %[section]
\def\XXsum#1#2#3{{\setbox0=\hbox{$#1{#2#3}{\sum}$ }
\vcenter{\hbox{$#2#3$ }}\kern-.5\wd0}}

\DeclareMathOperator{\prob}{\mathbb P}
\DeclareMathOperator{\avg}{\mathbb E}

\usepackage[ruled,vlined]{algorithm2e}

\title{Efficient Joint Object Matching via Linear Programming
}

\author{Antonio De Rosa
\thanks{Department of Mathematics, University of Maryland, 4176 Campus Dr, College Park, MD 20742, USA.
             E-mail: {\tt  anderosa@umd.edu}.
             }
\and
Aida Khajavirad
\thanks{Department of Industrial and Systems Engineering, Lehigh University, Bethlehem, PA 18015, USA.
             E-mail: {\tt aida@lehigh.edu}.
             }
}

\begin{document}

\maketitle

\date{}
\begin{abstract}
Joint object matching, also known as multi-image matching, namely, the problem of finding consistent partial maps among all pairs of objects within a collection, is
a crucial task in many areas of computer vision. This problem subsumes bipartite graph matching and graph partitioning as special cases and is NP-hard, in general.
We develop scalable linear programming (LP) relaxations with theoretical performance guarantees for joint object matching.
We start by proposing a new characterization of consistent partial maps; this in turn enables us to formulate joint object matching as an integer linear programming (ILP) problem. To construct strong LP relaxations, we study the facial structure of the convex hull of the feasible region of this ILP, which we refer to as the joint matching polytope. We present an exponential family of facet-defining inequalities that can be separated in strongly polynomial time, hence obtaining
 a partial characterization of the joint matching polytope that is both tight and cheap to compute. To analyze the theoretical performance
of the proposed LP relaxations, we focus on permutation group synchronization, an important special case of joint object matching.
We show that under the random corruption model for the input maps, a simple LP relaxation, that is, an LP containing only a very small fraction of the proposed facet-defining inequalities,  recovers the ground truth with high probability if the corruption level is below $40\%$.
%More concretely, we address the following question:
%\emph{What is the maximum level of input corruption under which the LP relaxation recovers the ground truth maps with high probability?}
%To this end we utilize tools from high dimensional probability.
Finally, via
a preliminary computational study on synthetic data, we show that the proposed LP relaxations outperform a popular SDP relaxation both in terms of recovery and tightness.
\end{abstract}


\vspace{0.1cm}


\section{Introduction}
\label{sec:intro}

Object matching techniques are widely used in many areas of computer vision such as image
analysis, object recognition, robotics, biomedical identification, and
object tracking. While there is a rich literature on finding isomorphisms between a pair of objects,
the task of~\emph{joint object matching}, also known as ~\emph{multi-image matching}, \ie finding consistent maps
among \emph{all} pairs of objects within a collection is underdeveloped.
The most classic example for joint object matching is the problem of matching feature points among many images of the same
object, a step used for instance in image recognition~\cite{Dermici06} and in structure from motion~\cite{Agar11}.
Almost all early approaches to tackle joint object matching relied on sequential matchings of pairs of objects often yielding erroneous results
when the input data is noisy.
Joint object matching is NP-hard in general; common techniques for tackling this problem are graph neural networks~\cite{Li19,Yan20}, spectral methods~\cite{PacKonSin13,Ling20,SheHuaSreSan16}, and semidefinite programming (SDP) relaxations~\cite{HuaGui13,CheGuiHua14,ZZD15,HHTG18}.
Indeed, to date, the only existing convex relaxations for joint object matching are SDP relaxations.
It is well-understood that in spite of their polynomial-time complexity, SDPs are too expensive to solve and are often impractical for large-scale problems.
In this paper, we develop scalable LP relaxations with theoretical performance guarantees for joint object matching; we start by formally defining the problem.

\subsection{Problem Statement}

We consider the problem of joint object matching in its full generality; that is, when the objects are only partially similar and the input is possibly incomplete.
To formally define this problem, we introduce some terminology which is mostly adapted from~\cite{CheGuiHua14}. Suppose that we have a collection of $n$
objects $\S_i$, $i \in [n]:= \{1, \ldots, n\}$ each of which consists of $d_i$ elements for some $n \geq 3$ and $d_i \geq 1$ such that $\max_{i \in [n]} d_i \geq 2$.
The case with $n =2$ is equivalent to bipartite graph matching (see for example chapter~8 of~\cite{schrijver86})
and the case with $d_i=1$ for all $i \in [n]$
is equivalent to graph partitioning (see for example~\cite{ChoRao93}), both of which are well-studied problems in combinatorial optimization.
Given two discrete sets $\S$ and $\S'$, a subset $\phi \subset \S \times \S'$ is called a \emph{partial map} if each
element of $\S$ (resp. $\S'$) is paired with at most one element of $\S'$ (resp. $\S$); in particular, not all
elements need to be paired. We denote by $\phi_{ij}: \S_i \rightarrow \S_j$ the partial map between the pair of objects $\S_i$ and $\S_j$.
In the special case where all elements in $\S_i$ and $\S_j$ are paired, we say that the two objects have~\emph{full similarity}; otherwise, we say that
the two objects have~\emph{partial similarity}. Note that two objects with full similarity must contain the same number of elements.

In joint object matching the input consists of \emph{noisy} pair-wise partial maps $\phi^{\rm in}_{ij}$ between some of the objects $\S_i$ and $\S_j$, $i, j \in [n]$.
These input maps are obtained using off-the-shelf pairwise graph matching algorithms and often contain some erroneous information.
An undirected graph $\G = (\V, \E)$ is called a \emph{map graph} for the objects $\S_i, i \in [n]$ with $\V := \{\S_1, \S_2, \cdots, \S_n\}$
and $(\S_i, \S_j) \in \E$ whenever an input partial map $\phi^{\rm in}_{ij}$ is available. For notational simplicity, throughout this paper instead of $(\S_i, \S_j) \in \E$,
we write $(i, j) \in \E$.
If the map graph is not a complete graph, we say that the input is \emph{incomplete}.
The objective is to find a collection of \emph{consistent} partial maps $\phi_{ij}$ for all $1 \leq i < j \leq n$
that are close to the input maps.
By consistent partial maps, we imply that for any $i,j,k \in [n]$, whenever an element $s \in \S_i$ is paired with an element $s' \in \S_j$ which in turn is paired with an element $s'' \in \S_k$, then $s$ and $s''$ are paired as well. This condition is often referred to as
\emph{cycle consistency} in the literature~\cite{HuaGui13}, as consistency can be achieved by requiring that the composition of maps between two objects is independent of the connecting path.

In the following, we formulate joint object matching as a mathematical optimization problem. To this end, we use a binary $d_i \times d_j$ matrix $X(i,j)$ to encode $\phi_{ij}$;
that is, the $s,s'$ entry of $X(i,j)$, denoted by $X_{ss'}(i,j)$, equals one if and only if $(s,s') \in \phi_{ij}$. It then follows that a binary matrix $X(i,j)$ satisfying
\begin{equation}\label{subD}
X(i,j) {\bf 1}_{d_j} \leq {\bf 1}_{d_j}, \qquad X^T(i,j) {\bf 1}_{d_i} \leq {\bf 1}_{d_i},
\end{equation}
corresponds to a partial map, where ${\bf 1}_{d_i} \in \R^{d_i}$  denotes a vector of all ones. Moreover, cycle consistency is achieved by requiring:
\begin{align}\label{consis}
&X(i,j) X(j,k) \leq X(i,k),\nonumber\\
&X(j,i) X(i,k) \leq X(j,k),\quad \forall 1 \leq i < j < k \leq n,\\
&X(i,k) X(k,j) \leq X(i,j),\nonumber
\end{align}
where $X(j,i) = X^T(i,j)$ for any $i < j$, $X(i,j) X(j,k)$ denotes the standard matrix product between $X(i,j)$ and $X(j,k)$,
and inequalities are all component-wise. In case of full similarity among all objects, $X(i,j)$ are permutation matrices for all $1 \leq i < j \leq n$; this special case
is often referred to as \emph{permutation group synchronization} in the literature (see for example~\cite{PacKonSin13}). It can be checked that cycle consistency
in this case is obtained by imposing:
\begin{equation}\label{consisPerm}
X(i,j) X(j,k) = X(i,k)\quad \forall 1 \leq i < j < k \leq n.
\end{equation}
The following example gives an illustration of partial maps and cycle consistency.
\begin{example}
Let $n = 3$, $d_1 =1$, $d_2 = d_3 = 2$. Then the following are consistent partial maps:
\begin{equation}
X(1,2) =
\begin{pmatrix}
0 & 1
\end{pmatrix}, \qquad
X(1,3) =
\begin{pmatrix}
0 & 0
\end{pmatrix}, \qquad
X(2,3) =
\begin{pmatrix}
1 & 0 \\
0 & 0
\end{pmatrix},
\end{equation}
indicating that the only element of object~1 corresponds to the second element of object~2, and the first element of object~2 corresponds to the first element of object~3.
However, the following partial maps are not consistent:
\begin{equation}
X(1,2) =
\begin{pmatrix}
1 & 0
\end{pmatrix}, \qquad
X(1,3) =
\begin{pmatrix}
0 & 1
\end{pmatrix}, \qquad
X(2,3) =
\begin{pmatrix}
0 & 0 \\
0 & 1
\end{pmatrix}.
\end{equation}
This is because we have $X_{11}(1,2) = X_{12}(1,3) = 1$, which by cycle consistency implies $X_{12}(2,3) = 1$.
Indeed by letting $X_{12}(2,3) = 1$, we obtain consistent partial maps indicating that the only element of object~1 corresponds to the first element of object~2 and to the second element of object~3, and the second element of object~2 corresponds to the second element of object~3.
\end{example}

Let $X^{\rm in}(i,j) \in \{0,1\}^{d_i\times d_j}$
denote the matrix representation of the input partial map $\phi^{\rm in}_{ij}$.
We would like to find consistent partial maps $X(i,j)$, $1\leq i < j \leq n$, so as to minimize
\begin{equation}\label{objective}
f=\sum_{(i,j) \in \E}{\norm{X^{\rm in}(i,j) - X(i,j) }^2_F},
\end{equation}
where $\norm{\cdot}_F$ denotes the Frobenius norm.
Denote by $\mathcal{N}$ the number of matched pairs in the input, \ie $\mathcal{N} = \sum_{(i,j) \in \E} {\sum_{t\in [d_i]}\sum_{q \in [d_j]}{X^{\rm in}_{tq}(i,j)}}$.
Since matrices $X^{\rm in}(i,j)$ and $X(i,j)$ are binary-valued, the objective function~\eqref{objective} can be written as:
$$
f = \mathcal{N}+ \sum_{(i,j) \in \E}{\left\langle{\bf 1}_{d_i}{\bf 1}_{d_j}^T-2 X^{\rm in}(i,j), \; X(i,j)\right\rangle},
$$
where $\langle \cdot, \cdot \rangle$ denotes the standard matrix inner product.
It then follows that joint object matching can be formulated as follows:
\begin{align}\label{jom}
\tag{JOM}
\min \quad &\sum_{(i,j) \in \E}{\left\langle{\bf 1}_{d_i}{\bf 1}_{d_j}^T-2 X^{\rm in}(i,j), \; X(i,j)\right\rangle} \nonumber\\
\st  \quad & X(i,j) {\bf 1}_{d_j} \leq  {\bf 1}_{d_j}, \quad X^T(i,j) {\bf 1}_{d_i} \leq {\bf 1}_{d_i} , \quad \forall 1 \leq i < j \leq n, \nonumber\\
&X(i,j) X(j,k) \leq X(i,k),\nonumber\\
&X(j,i) X(i,k) \leq X(j,k),\quad \forall 1 \leq i < j < k \leq n,\nonumber\\
&X(i,k) X(k,j) \leq X(i,j),\nonumber\\
           &  X(i,j) \in \{0,1\}^{d_i\times d_j}, \quad \forall 1 \leq i < j \leq n.\nonumber
\end{align}
%Problem~\eqref{jom} is NP-hard in general; common techniques for tackling this problem are graph neural networks~\cite{Li19,Yan20}, spectral methods~\cite{PacKonSin13,Ling20,SheHuaSreSan16} and SDP relaxations~\cite{HuaGui13,CheGuiHua14}.
In this paper, we are interested in the quality of convex relaxations for Problem~\eqref{jom}.
To date, the only existing convex relaxations for joint object matching are SDP relaxations.
In the following, we describe a widely-used SDP relaxation of Problem~\eqref{jom} first proposed in~\cite{HuaGui13}.

%In~\cite{HuaGui13}, for the first time, the authors consider the problem of simultaneous matching of all objects in the collection and formulate this problem as a \emph{large-scale binary semi-definite program} (SDP), a nonconvex optimization problem that is very difficult to solve even for small instances. Subsequently, they consider the corresponding SDP relaxation and study conditions, in terms of corruption-level in input maps, under which this convex optimization problem recovers the ground truth maps. It is well-understood that
%in spite of their polynomial-time complexity, SDPs are too expensive to solve and are often impractical for large-scale problems.
%Indeed, to date, the only existing convex relaxations for joint object matching are SDP relaxations.

\subsection{SDP relaxations}

Assume that there exists a universe $\U$ consisting of $m$ elements
%for some $m \in \{d, \cdots, nd\}$
such that each object $\S_i$ is a (partial) image of $\U$ and each element in $\U$ is contained in at least one object $\S_i$.
We should remark that the set $\U$ or even its size is not known a priori. We only rely on its existence to construct the SDP relaxation.
%For instance, consider two extreme cases: (i) $m = d$: all $X(i,j) $ are permutation matrices as all elements in $\S_i$ and $\S_j$ are paired,
%(ii) $m = nd$: all $X(i,j) $ are zero matrices as no elements in $\S_i$ and $\S_j$ are paired.
Let the matrix $Y(i) \in \{0,1\}^{d_i\times m}$ encode the correspondences between
$\S_i$ and $\U$, \ie for any $s \in \S_i$ and $s' \in \U$,
the $s, s'$ entry of $Y(i)$ equals one if and only if $s$ corresponds to $s'$. It then follows that
$X(i,j) = Y(i) Y^T(j)$ for all $1 \leq i < j \leq n$. Define $X(i,i) = I_{d_i}$ for all $i \in [n]$, where $I_{d_i}$ denotes a $d_i\times d_i$ identity matrix.
Let $\bar d = \sum_{i \in [n]}{d_i}$; denote by $X \in \{0,1\}^{\bar d \times \bar d}$ a matrix whose $(i,j)$-th block is given by $X(i,j)$, where we let $X(j,i) = X^T(i,j)$. Let $Y = (Y^T(1), \cdots, Y^T(n))^T$. It can be checked that
$X = Y Y^T$ which implies $X$ is a rank-$m$ positive semidefinite matrix.
Similarly denote by $X^{{\rm in}}$ a $\bar d \times \bar d$ matrix whose $(i,j)$-th
block is given by $X^{{\rm in}}(i,j)$,
where we let $X^{{\rm in}}(i,i)  = I_{d_i}$. Then an SDP relaxation of Problem~\eqref{jom} is given by:
\begin{align}\label{sdp}
%\tag{SDP}
\min \quad &\sum_{(i,j) \in \E}{\left\langle{\bf 1}_{d_i}{\bf 1}_{d_j}^T-2 X^{\rm in}(i,j), \; X(i,j)\right\rangle} \\
\st  \quad  & X\succeq 0, \quad X \geq 0, \nonumber\\
      & X(i,i) = I_{d_i}, \quad \forall i \in [n], \nonumber\\
      & X(i,j) {\bf 1}_{d_j} \leq {\bf 1}_{d_j}, \quad X^T(i,j) {\bf 1}_{d_i} \leq {\bf 1}_{d_i} , \quad \forall 1 \leq i < j \leq n, \nonumber
\end{align}
where $X\succeq 0$ and $X \geq 0$ mean that $X$ is positive semidefinite and component-wise nonnegative, respectively.
In case where $X(i,j)$, $1 \leq i < j \leq n$,
are all permutation matrices, the authors of~\cite{HuaGui13} proved that the positive semidefiniteness condition $X \succeq 0$ is equivalent to cycle consistency.  This in turn implies that in case of fully similar objects, replacing $X \geq 0$ by
$X \in \{0,1\}^{\bar d \times \bar d}$ in Problem~\eqref{sdp}, one obtains an exact reformulation of Problem~\eqref{jom} as a binary SDP.

It is important to emphasize that the size of the universe $m$ is \emph{not} known a priori; the authors of~\cite{CheGuiHua14} proposed a spectral technique to \emph{estimate} $m$ using the input data $X^{\rm in}$. They first \emph{trim} $X^{\rm in}$ to remove the bias from over represented rows or columns. Let us denote by $\tilde X^{\rm in}$ the trimmed matrix. Denote by $\lambda_k$ the $k$-th largest eigenvalue of $\tilde X^{\rm in}$. Then they let $\hat m =\arg \max_{d_{\max} \leq k < \bar d} |\lambda_k - \lambda_{k+1}|$, where $d_{\max} = \max_{i \in [n]} d_i$, and where $\hat m$ denotes the estimate of $m$.
Subsequently, they  used this estimate to further strengthen the SDP relaxation~\eqref{sdp}; that is, they replaced $X \succeq 0$ by the following constraint:
\begin{equation}\label{superSDP}
\begin{pmatrix}
\hat m & {\bf 1}^T_{\bar d} \\
{\bf 1}_{\bar d} & X
\end{pmatrix} \succeq 0.
\end{equation}
In~\cite{ZZD15}, the authors employed Burer-Monteiro factorization~\cite{BM05} by letting $X = Y Y^T$ to tackle the SDP. Recall that $Y$ is a $\bar d \times m$ matrix; hence, to benefit from the low-rank approach, the availability of a good upper bound on the size of the universe is essential.
The authors of~\cite{ZZD15} observed that the estimation technique of~\cite{CheGuiHua14} is inaccurate when the input is noisy and incomplete. They then chose to use the upper bound $2 \hat m$, where $\hat m$ is obtained by the spectral method of~\cite{CheGuiHua14} outlined above. To conclude, constraint~\eqref{superSDP} should only be used if a reliable upper bound on the size of the universe is available.

\subsection{Recovery guarantees under stochastic models}

To perform a theoretical analysis of various existing and new algorithms for data science applications,
a recent stream of research in mathematical data science is focused on obtaining sufficient conditions for recovery of the ground truth under various stochastic models for the input~(see for example~\cite{CheGuiHua14,AbbBanHal16,HajWuXu16,dPIdaTim20,AntoAida20,dPIda22}).
We say that an optimization algorithm \emph{recovers} the ground truth, whenever its unique optimal solution coincides with the ground truth.
In the context of joint object matching, the question can be formally stated as follows: given a probabilistic model for the noise in input partial maps, what is the maximum level of corruption under which the (optimization) algorithm \emph{recovers} the ground truth partial maps with high probability? Throughout this paper, by high probability, we imply the probability tending to 1 as the number of objects $n \rightarrow \infty$.

Huang and Guibas~\cite{HuaGui13} considered the SDP relaxation~\eqref{sdp} for the
permutation group synchronization problem and obtained a deterministic sufficient condition for recovery, implying under a random model recovery is possible with high probability when the corruption level remains below 50\%.
In~\cite{CheGuiHua14}, the authors considered the general joint object matching with partially similar objects,
together with the SDP relaxation~\eqref{sdp} enhanced by constraint~\eqref{superSDP}. They first significantly improved their earlier
deterministic recovery guarantee in~\cite{HuaGui13}. Subsequently, they focused on the \emph{random corruption model}, roughly defined as follows: each observed $X^{\rm in}(i,j)$
coincides with ground truth independently with probability $p_{\rm true}$, and each observed but incorrect $X^{\rm in}(i,j)$ is independently drawn from a set of partial maps satisfying
$\avg[X^{\rm in}(i,j)] = \frac{\bf{1}\bf{1}^T}{m}$, where $\avg[\cdot]$ denotes the expectation of a random variable. Note that, in this model, the authors assume that the size of the universe $m$ is known a priori, an assumption which often does not hold. They proved that the SDP relaxation recovers the ground truth with high probability if $p_{\rm true} > C \frac{\log^2(nm)}{\sqrt{n}}$, for some universal large constant $C$. In~\cite{BajGao18,Ling20}, the authors considered spectral methods for solving permutation group synchronization under the random corruption model. They proved that the spectral algorithm recovers the ground truth with high probability, if $p_{\rm true} > C \sqrt{\frac{\log(nd)}{n}}$, where $d$ denotes the number of elements in each object. This recovery guarantee is nearly optimal in terms of information
theoretical limits~\cite{ChSuGo16}. In~\cite{PaKoSaSi14}, the authors studied the recovery properties of a spectral algorithm for permutation group synchronization under the additive Gaussian noise.

To summarize, for joint object matching under the random corruption model, both SDP relaxations and spectral methods exhibit near optimal recovery thresholds.
It is well understood that solving SDPs is computationally prohibitive for large-scale problems. Spectral methods on the other hand, only require estimating the leading eigenvector of a matrix, and are quite efficient in practice.
However, works like~\cite{Wein16,Monta16} suggest that spectral methods are extremely sensitive to slight modification in the
generative model. Indeed, convex relaxations
enjoy robustness to adversarial corruptions of the inputs for statistical problems that spectral methods do
not. Therefore, it is of great interest to understand the most efficient convex relaxation algorithms for solving joint object matching.

\subsection{Our contribution}

In this paper, we propose scalable LP relaxations with theoretical performance guarantees for joint object matching.
To this end, we first present an alternative characterization of consistent partial maps. This in turn enables us to formulate joint object matching as an ILP.
Subsequently, with the objective of  constructing strong LP relaxations, we study the facial structure of the convex hull of the feasible region of the ILP.
As part of this polyhedral study, to effectively approximate cycle consistency, we introduce consistency inequalities, an exponential family of facet-defining inequalities that can be separated in strongly polynomial time.
We next study the theoretical properties of the proposed LP relaxation for the permutation group synchronization problem; we show that under the random corruption model, a simple LP relaxation containing only a very small fraction of consistency inequalities, recovers the ground truth with high probability if $p_{\rm true} > 0.585$. While in the asymptotic regime our recovery guarantee is suboptimal, in many cases of practical interest, the corruption level is below \%40. It is for such applications that our proposed LP provides a robust and efficient matching algorithm. In fact, our numerical experiments suggest that for moderate values of $n$ and $d$, the proposed LP relaxation outperforms the SDP relaxation
in recovering the ground truth.

The remainder of the paper is organized as follows. In Section~\ref{sec:lp} we present a novel LP relaxation for joint object matching.
Subsequently, in Section~\ref{sec:dualCertificate} we focus on the permutation group synchronization problem under the random corruption model and obtain a recovery guarantee for
the proposed LP relaxation. We present our numerical experiments in Section~\ref{sec:numerics}. Section~\ref{appendix} contains further results regarding the facial structure of the joint matching polytope that were omitted from Section~\ref{sec:lp}.

\section{Linear programming relaxation}
\label{sec:lp}
In this section, we propose a scalable LP relaxation with performance guarantees for joint object matching.
We start by presenting an ILP formulation for Problem~\eqref{jom}.
Let us first obtain a linear characterization of cycle consistency constraints~\eqref{consis}.
A collection of partial maps $X(i,j)$, $1 \leq i < j \leq n$, is consistent
if for any $1 \leq i < j < k \leq n$ and for any $l \in [d_i], t \in [d_j], q \in [d_k]$, whenever two out of the three elements
$X_{lt}(i,j)$, $X_{tq}(j,k)$, $X_{lq}(i,k)$ equal one, the third one equals one as well.
It is simple to check that this condition is enforced by the following system of inequalities:
\begin{eqnarray}
\left\{
\begin{array}{ll}
- X_{lt}(i,j) + X_{tq}(j,k) +X_{lq}(i,k)   \leq 1,\\
X_{lt}(i,j) - X_{tq}(j,k) +X_{lq}(i,k) \leq 1, \\
X_{lt}(i,j) + X_{tq}(j,k) -X_{lq}(i,k) \leq 1,\\
\end{array} \right. \forall l \in [d_i],t \in [d_j], q \in [d_k], \; 1 \leq i < j < k\leq n. \label{gtri}
\end{eqnarray}
Hence an ILP formulation for Problem~\eqref{jom} is given by:
\begin{align}\label{ip}
\tag{IP}
\min \quad & \sum_{(i,j) \in \E}{\left\langle{\bf 1}_{d_i}{\bf 1}_{d_j}^T-2 X^{\rm in}(i,j), \; X(i,j)\right\rangle}\nonumber\\
\st \quad  & X(i,j) {\bf 1}_{d_j} \leq {\bf 1}_{d_j}, \quad X^T(i,j) {\bf 1}_{d_i} \leq {\bf 1}_{d_i} , \quad \forall 1 \leq i < j \leq n,\nonumber \\
         &  \left\{
\begin{array}{ll}
- X_{lt}(i,j) + X_{tq}(j,k) +X_{lq}(i,k)   \leq 1,\\
X_{lt}(i,j) - X_{tq}(j,k) +X_{lq}(i,k) \leq 1, \\
X_{lt}(i,j) + X_{tq}(j,k) -X_{lq}(i,k) \leq 1,\\
\end{array} \right. \forall l \in [d_i],t \in [d_j],q \in [d_k], \; \forall 1 \leq i < j < k\leq n, \nonumber\\
           & X(i,j) \in \{0,1\}^{d_i\times d_j}, \quad \forall 1 \leq i < j \leq n. \nonumber
\end{align}
%Notice that the first set of equality constraints together with the integrality requirement enforces $X(i,j)$, $1 \leq i < j \leq n$
%to be permutation matrices while the inequality constraints enforce consistency of these permutation matrices.

\begin{remark}\label{rem1}
Inequalities~\eqref{gtri} can be considered as a generalization of triangle inequalities used for instance in graph partitioning problems~\cite{GroWak90,ChoRao93}. Given a graph with $n$ nodes, the goal in graph partitioning is to partition the nodes of the graph into at most $K$ subsets
such that some similarity measure across different partitions is minimized. In this context, for each pair of nodes $i,j \in [n]$ a
variable $y_{ij}$ is defined as follows: $y_{ij} = 1$ if nodes $i$ and $j$ belong to the same partition and $y_{ij} = 0$, otherwise.
The triangle inequalities are then defined as
$$
y_{ij} + y_{ik} - y_{jk} \leq 1, \quad \forall i \neq j \neq k \in [n].
$$
The above inequality states that if $i$ and $j$ are in the same cluster, and $i$ and $k$ are in
the same cluster, then also $j$ and $k$ must be in the same cluster. It then follows that inequalities~\eqref{gtri} are
a generalization of triangle inequalities for the case node $i$ of the graph represents object $\S_i$ consisting of $d_i$ elements where $d_i \geq 2$ for some $i \in [n]$.
\end{remark}

Now consider the simple LP relaxation of Problem~\eqref{ip}, that is, the LP obtained by replacing $X(i,j) \in \{0,1\}^{d_i\times d_j}$ by the constraint
$X(i,j) \geq 0$. Let us refer to this LP as the~\emph{basic LP}. In the next example, we show that inequalities~\eqref{gtri} are not implied by the SDP relaxation~\eqref{sdp}.

\begin{example}\label{compare}
Let $n = 3$ and $d_1 = d_2 = d_3= 2$; it can be checked that the following is feasible for Problem~\eqref{sdp}:
\begin{equation}\label{ex2}
X(1,2) = X(1,3) =
\begin{pmatrix}
\frac{3}{4} & \frac{1}{4} \\
\frac{1}{4}  & \frac{3}{4}
\end{pmatrix}, \qquad
X(2,3) =
\begin{pmatrix}
\frac{1}{4} & \frac{3}{4} \\
\frac{3}{4} & \frac{1}{4}
\end{pmatrix}.
\end{equation}
%In addition, if we assume that $m = 2$ (knowing the size of the universe is often an unreali), it can be checked that~\eqref{ex2} satisfies the positive semidefiniteness condition~\eqref{superSDP} as well.
Now consider the inequality obtained by letting $l=t=q = 1$ in the second inequality of system~\eqref{gtri}:
$$
X_{11}(1,2)-X_{11}(2,3)+X_{11}(1,3) \leq 1.
$$
Substituting~\eqref{ex2} in the above inequality yields $\frac{3}{4} - \frac{1}{4}+ \frac{3}{4} \not\leq 1$.
\end{example}

In the following, we improve the strength of the basic LP relaxation by obtaining a partial linear characterization of cycle consistency.
Subsequently, we establish the strength of the proposed inequalities by showing that they define facets of the convex hull of the feasible region of Problem~\eqref{ip}.

\subsection{Consistency inequalities}
We now present a generalization of inequalities~\eqref{gtri} for joint object matching.

\begin{proposition}\label{validity}
%Consider some nonempty
%$D_1, D_2 \subseteq [d]$;  for any
%$1 \leq i < j <k \leq n$ and for any $l \in [d]$, the following inequalities are valid for the feasible region of Problem~\eqref{ip}:
%
Let $1 \leq i < j <k \leq n$.
Then following inequalities are valid for the feasible region of Problem~\eqref{ip}:
\begin{align}\label{superg}
-&\sum_{t \in D_1}{\sum_{q \in D_2}{X_{tq}(i,j)}}+\sum_{q \in D_2}{X_{ql}(j,k)} +\sum_{t \in D_1}{X_{tl}(i,k)}\leq 1, \nonumber\\
&\qquad \qquad \qquad \qquad \qquad\forall D_1 \subseteq [d_i], \; D_2 \subseteq [d_j], \; l \in [d_k], \; D_1, D_2 \neq \emptyset \nonumber\\
&\sum_{t \in D_1}{X_{lt}(i,j)} -\sum_{t \in D_1}{\sum_{q\in D_2}{X_{tq}(j,k)}}+\sum_{q\in D_2} {X_{lq}(i,k)}\leq 1,  \nonumber\\
&\qquad \qquad  \qquad \qquad\qquad\forall D_1 \subseteq [d_j], \; D_2 \subseteq [d_k], \; l \in [d_i], \; D_1, D_2 \neq \emptyset\\
&\sum_{t \in D_1}{X_{tl}(i,j)} +\sum_{q\in D_2}{X_{lq}(j,k)}-\sum_{t \in D_1}{\sum_{q\in D_2} {X_{tq}(i,k)}}\leq 1, \nonumber \\
&\qquad \qquad \qquad \qquad \qquad\forall D_1 \subseteq [d_i], \; D_2 \subseteq [d_k], \; l \in [d_j], \; D_1, D_2 \neq \emptyset. \nonumber
\end{align}
\end{proposition}

\begin{proof}
Without loss of generality, consider the third inequality in~\eqref{superg}, for some $l \in [d_j]$.
To see the validity of this inequality, notice that by condition~\eqref{subD}, at most one term in $\sum_{t \in D_1} {X_{tl}(i,j)}$
and at most one term in $\sum_{q \in D_2}{X_{lq}(j,k)}$ equal one.
Now suppose that we have $X_{\hat t l}(i,j) = 1$
for some $\hat t \in D_1$ and $X_{l\hat q}(j,k) = 1$ for some $\hat q \in D_2$. Then by inequalities~\eqref{gtri}, we have $X_{\hat t \hat q}(i,k) = 1$ and this completes the proof of validity.
\end{proof}

Notice that by letting $|D_1| =|D_2|= 1$ in inequalities~\eqref{superg}, we obtain inequalities~\eqref{gtri}.
Henceforth, we refer to inequalities~\eqref{superg} as \emph{consistency inequalities}.
The following example demonstrates that consistency inequalities~\eqref{superg} strengthen the basic LP relaxation of Problem~\eqref{ip}.

\begin{example}\label{eg1}
Let $n =3$ and $d_1 = d_2 =d_3 =2$; it can be checked that the following satisfies inequalities~\eqref{gtri}:
\begin{equation}\label{ex1}
X(1,2) =
\begin{pmatrix}
\frac{1}{2} & \frac{1}{2} \\
0 & 0
\end{pmatrix}, \qquad
X(2,3) =
\begin{pmatrix}
0 & 0 \\
0 & 0
\end{pmatrix}, \qquad
X(1,3) =
\begin{pmatrix}
0 & \frac{1}{2} \\
0 & 0
\end{pmatrix}.
\end{equation}
Now consider the inequality obtained by letting $l=1$, $D_1 = \{1,2\}$, and $D_2 =\{2\}$ in the second inequality of system~\eqref{superg}:
$$
X_{11}(1,2)+X_{12}(1,2)-X_{12}(2,3)-X_{22}(2,3)+X_{12}(1,3) \leq 1.
$$
Substituting~\eqref{ex1} in the above inequality yields $\frac{1}{2} + \frac{1}{2}- 0  -0 + \frac{1}{2} \not\leq 1$.
Next consider
\begin{equation}\label{exx3}
X(1,2) =
\begin{pmatrix}
0 & 0 \\
0 & 0
\end{pmatrix}, \qquad
X(2,3) =
\begin{pmatrix}
0 & \frac{1}{3} \\
0 & \frac{1}{3}
\end{pmatrix}, \qquad
X(1,3) =
\begin{pmatrix}
0 & \frac{1}{3} \\
0 & \frac{1}{3}
\end{pmatrix}.
\end{equation}
It can be checked that~\eqref{exx3} satisfies all consistency inequalities with $|D_1| = 1$ or $|D_2| = 1$. Now consider the consistency inequality with
$l=2$, $|D_1| = |D_2| = 2$ given by:
$$
-X_{11}(1,2)-X_{12}(1,2)-X_{21}(1,2)-X_{22}(1,2)+X_{12}(2,3)+X_{22}(2,3)+X_{12}(1,3)+X_{22}(1,3) \leq 1.
$$
Substituting~\eqref{exx3} in the above inequality yields $-0-0-0-0+\frac{1}{3} + \frac{1}{3} + \frac{1}{3}+ \frac{1}{3} \not\leq 1$.
\end{example}

The number of consistency inequalities~\eqref{superg} is exponential in the number of elements $d_i$, $i \in [n]$.
That is, for each $1 \leq i < j < k \leq n$, we have $d_i (2^{d_j}+2^{d_k}-2)+d_j (2^{d_i}+2^{d_k}-2)+d_k (2^{d_i}+2^{d_j}-2)$ consistency inequalities. However, as we detail next, separating over these inequalities can be done in a number of operations that is polynomial in $n$ and $d_i$, $i \in [n]$.
%By polynomial equivalence of separation and optimization~\cite{gls88}, it then follows that Problem~\eqref{lp2} is solvable in polynomial time.

\paragraph{Separation of consistency inequalities.}
Let us start by defining the separation problem:

\medskip

\emph{The separation problem.} Let $\tilde X(i,j) \in [0,1]^{d_i \times d_j}$ for all $1\leq i < j \leq n$ satisfy inequalities~\eqref{subD}. Decide whether $\tilde X$
satisfies all consistency inequalities~\eqref{superg} or not, and in the latter case find a consistency inequality that is violated by $\tilde X$.

\medskip

\begin{proposition}\label{p:separate}
There exists a strongly polynomial time algorithm that solves the separation problem over all consistency inequalities.
\end{proposition}

\begin{proof}
Let $1 \leq i < j < k \leq n$ and let $l \in [d_j]$; consider the third inequality in~\eqref{superg} for all nonempty $D_1 \subseteq [d_i]$, $D_2 \subseteq [d_k]$. We claim that the inequality that is most violated by $\tilde X(i,j) \in [0,1]^{d_i \times d_j}$ for all $1\leq i < j \leq n$ (if one exists) among all such $2^{d_i}+2^{d_k}-2$ inequalities can be found by an algorithm that is strongly polynomial in $d_i,d_k$. This in turn implies that the separation problem over all consistency inequalities can be solved in strongly polynomial time with respect to $n, d_i$, $i \in [n]$.
%Let $D_1, D_2 \subseteq [d]$ be nonempty.
Define binary variables $y_t, z_q$ for all $t \in [d_i], q \in [d_k]$ as follows: for each $t \in [d_i]$, we let $y_t = 1$ if $t \in D_1$ and $y_t = 0$, otherwise; similarly, for each $q \in [d_k]$, we let $z_q = 1$ if $q \in D_2$ and $z_q = 0$, otherwise.
%Suppose that we solve a cheap LP relaxation of Problem~\eqref{ip}; for example, Problem~(LP1) without inequalities~\eqref{gtri}. Denote by $\tilde X(i,j)$, $1 \leq i < j \leq n$ the relaxation solution and suppose that this solution is not integral. Notice that $\tilde X(i,j) \geq 0$, $1 \leq i < j \leq n$.
Then to find the most violated consistency inequality, it suffices to solve the following optimization problem:
\begin{align}\label{auxSep}
\max \quad & \sum_{t \in [d_i]}{\tilde X_{tl}(i,j) y_t}+\sum_{q \in [d_k]}{\tilde X_{lq}(j,k) z_q}-\sum_{t \in [d_i]}{\sum_{q \in [d_k]}{\tilde X_{tq}(i,k) y_t z_q}}  \\
\st \quad  &y \in \{0,1\}^{d_i}, \; z \in \{0,1\}^{d_k} \nonumber.
\end{align}
If the optimal value of the above problem is greater than one, the maximizer provides us with $D_1, D_2$ corresponding to a most violated inequality; otherwise, no violated
consistency inequality exists.  Now let us examine the complexity of solving Problem~\eqref{auxSep}. Denote by $f(y,z)$ the objective function of Problem~\eqref{auxSep}.
Defining $\bar z_q  = 1-z_q$ for all $q \in [d_k]$, we get
\begin{align*}
f(y, \bar z)  &=  \sum_{t \in [d_i]}{\tilde X_{tl}(i,j) y_t}+\sum_{q \in [d_k]}{\tilde X_{lq}(j,k) (1- \bar z_q)}-\sum_{t \in [d_i]}{\sum_{q \in [d_k]}{\tilde X_{tq}(i,k) y_t (1-\bar z_q)}}\\
& =  \sum_{t \in [d_i]}{\big(\tilde X_{tl}(i,j)-\sum_{q\in [d_k]}{\tilde X_{tq}(i,k)} \big)y_t}+\sum_{q \in [d_k]}{\tilde X_{lq}(j,k) (1- \bar z_q)}+\sum_{t \in [d_i]}{\sum_{q \in [d_k]}{\tilde X_{tq}(i,k) y_t \bar z_q}}
\end{align*}
Since by assumption $\tilde X_{tq}(i,k) \geq 0$ for all $t \in [d_i], q \in [d_k]$ and for all $1 \leq i < k \leq n$, it can be checked that $f(y, \bar z)$, $y \in \{0,1\}^{d_i}, \bar z \in \{0,1\}^{d_k}$ is a super-modular function (see for example~\cite{HamHanSim84}).
It then follows that Problem~\eqref{auxSep} can be transformed into the problem of maximizing a super-modular function and hence can be solved in strongly polynomial time in $d_i+d_k$~\cite{Sch00}. Hence, solving the separation problem over all consistency inequalities amounts to (in the worst case) solving  $3 \tilde d \binom{n}{3}$ optimization problems of the form~\eqref{auxSep}, where $\tilde d = \max_{i \in [n]} d_i$.
\end{proof}

\begin{remark}
It can be shown that Problem~\eqref{auxSep}, \ie the problem of maximizing a super-modular quadratic function, can be equivalently solved by solving the following LP~\cite{BorHam02}:
\begin{align}\label{separate}
\max \quad &  \sum_{t \in [d_i]}{\tilde X_{tl}(i,j) y_t}
+\sum_{q \in [d_k]}{\tilde X_{lq}(j,k) z_q} -\sum_{t\in [d_i]}\sum_{q\in [d_k]}{\tilde X_{tq}(i,k) w_{tq}}\\
\st \quad  & w_{tq} \geq y_t + z_q -1, \; w_{tq} \geq 0, \quad \forall t \in [d_i], q \in [d_k]\nonumber\\
           & y \in [0,1]^{d_i}, \; z \in [0,1]^{d_k}, \nonumber
\end{align}
which can be readily solved using a generic LP solver. Note that the above LP has $d_i d_k+d_i+d_k$ variables; \ie $y_t, z_q, w_{tq}$ for all $t \in [d_i], q \in [d_k]$.
%We should also mention that Problem~\eqref{auxSep} can be solved efficiently using a network flow type algorithm (see Section~5.1.5 in~\cite{BorHam02}).
As we detail in Section~\ref{sec:numerics}, to solve the separation problem over consistency inequalities, we solve Problem~\eqref{separate}.
\end{remark}

Notice that in order to construct Problem~\eqref{auxSep} we do not require $D_1$ and $D_2$ to be nonempty, since if at least one of the two subsets is empty, the resulting inequality is implied by inequalities~\eqref{subD} and hence is trivially satisfied.

\paragraph{Block consistency inequalities.}
Consistency inequalities~\eqref{superg} can be significantly generalized as follows:

\begin{proposition}
Let $1 \leq i < j <k \leq n$.
Then following inequalities are valid for the feasible region of Problem~\eqref{ip}:
\begin{align}\label{duperg}
-&\sum_{t \in D_1}{\sum_{q \in D_2}{X_{tq}(i,j)}}+\sum_{q \in D_2}{\sum_{l \in D_3}{X_{ql}(j,k)}} +\sum_{t \in D_1}{\sum_{l \in D_3}{X_{tl}(i,k)}}\leq |D_3|,\nonumber\\
& \qquad \qquad \qquad \forall D_1 \subseteq [d_i], \; D_2 \subseteq [d_j], \; D_3 \subseteq [d_k], D_1, D_2, D_3 \neq \emptyset.\nonumber\\
&\sum_{l \in D_3}{\sum_{t \in D_1}{X_{lt}(i,j)}} -\sum_{t \in D_1}{\sum_{q\in D_2}{X_{tq}(j,k)}}+\sum_{l \in D_3}{\sum_{q\in D_2} {X_{lq}(i,k)}}\leq |D_3|,\nonumber\\
& \qquad \qquad \qquad\forall D_1 \subseteq [d_j], \; D_2 \subseteq [d_k], \; D_3 \subseteq [d_i], D_1, D_2, D_3 \neq \emptyset\\
&\sum_{t \in D_1}{\sum_{l \in D_3}{X_{tl}(i,j)}} +\sum_{l \in D_3}{\sum_{q\in D_2}{X_{lq}(j,k)}}-\sum_{t \in D_1}{\sum_{q\in D_2} {X_{tq}(i,k)}}\leq |D_3|,\nonumber\\
& \qquad \qquad \qquad\forall D_1 \subseteq [d_i], \; D_2 \subseteq [d_k], \; D_3 \subseteq [d_j], D_1, D_2, D_3 \neq \emptyset.\nonumber
\end{align}
\end{proposition}
\begin{proof}
Without loss of generality, let us consider the third inequality and let us rewrite this inequality as
$$
\sum_{l \in D_3}{\Big(\sum_{t \in D_1}{X_{tl}(i,j)}+\sum_{q\in D_2}{X_{lq}(j,k)}\Big)}-\sum_{t \in D_1}{\sum_{q\in D_2} {X_{tq}(i,k)}}\leq |D_3|.
$$
By inequalities~\eqref{subD} for each $l \in D_3$ we have  $\sum_{t \in D_1}{X_{tl}(i,j)} \leq 1$ and $\sum_{q\in D_2}{X_{lq}(j,k)} \leq 1$.
By proof of Proposition~\ref{validity}, for each $l \in D_3$, if $X_{\hat t l}(i,j) = X_{l \hat q}(j,k) = 1$, for some $\hat t \in D_1$
and $\hat q \in D_2$, then  $X_{\hat t \hat q}(i,k) = 1$. Now suppose that there exists $l_1 \neq l_2 \in D_3$ for which this condition is satisfied; \ie
$X_{\hat t l_1}(i,j) = X_{l_1 \hat q}(j,k) = 1$ and $X_{\tilde t l_2}(i,j) = X_{l_2 \tilde q}(j,k) = 1$. Then by inequalities~\eqref{subD} it follows that
$\hat t \neq \tilde t$ and $\hat q \neq \tilde q$ implying that by inequalities~\eqref{gtri}, two distinct terms $X_{\hat t \hat q}(i,k)$ and $X_{\tilde t \tilde q}(i,k)$ in the third summation equal one as well.
A recursive application of this argument completes the proof of validity.
\end{proof}
Notice that by letting $|D_3| = 1$ in inequalities~\eqref{duperg}, we obtain consistency inequalities~\eqref{superg}.
Henceforth, we refer to inequalities~\eqref{duperg} as \emph{block consistency inequalities}.
If $|D_1| + |D_2| \leq |D_3|$, then by~\eqref{subD}, block consistency inequalities are redundant as the value of the left-hand side of these inequalities does not exceed
$|D_1|+|D_2|$. However, in general, block consistency inequalities~\eqref{duperg} with $|D_3| > 1$ are not implied by the consistency inequalities~\eqref{superg}. The following example illustrates this fact:

\begin{example}
Let $n =3$ and $d_1 = d_2 = d_3 =2$; it can be checked that the following satisfies inequalities~\eqref{superg}:
\begin{equation}\label{ex7}
X(1,2) =
\begin{pmatrix}
0 & \frac{1}{2} \\
0 & 0
\end{pmatrix}, \qquad
X(2,3) =
\begin{pmatrix}
\frac{1}{2} & \frac{1}{2} \\
\frac{1}{2} & \frac{1}{2}
\end{pmatrix}, \qquad
X(1,3) =
\begin{pmatrix}
\frac{1}{2} & \frac{1}{2} \\
0 & 0
\end{pmatrix}.
\end{equation}
Now consider the inequality obtained by letting $D_1=\{1\}$, $D_2=D_3 = \{1,2\}$ in the first inequality of system~\eqref{duperg}:
$$
-X_{11}(1,2)-X_{12}(1,2)+X_{11}(2,3)+X_{12}(2,3)+X_{21}(2,3)+X_{22}(2,3)+X_{11}(1,3)+X_{12}(1,3) \leq 2.
$$
Substituting~\eqref{ex7} in the above inequality yields $0-\frac{1}{2} +\frac{4}{2} + \frac{2}{2} \not\leq 2$.
\end{example}

We will further detail on the theoretical strength of block consistency inequalities in the next section.
From a computational perspective however, unlike consistency inequalities~\eqref{superg}, the separation problem for this more general class cannot be reduced to the problem of
maximizing a binary supermodular quadratic function. Indeed, there does not seem to be a reduction to any known class of polynomially-solvable binary quadratic program~\cite{Bar86,Pad89,BorHam02,michini21}. While we leave the complexity of this separation problem as an open question, we suspect that it is a difficult problem.
Moreover, inequalities~\eqref{duperg} are very dense; that is they contain many variables with nonzero coefficients; a feature that is not computationally desirable.
Hence, in this paper, we do not consider block consistency inequalities for our computational experiments.

\subsection{The Joint Matching Polytope}
In this section, we establish the strength of the proposed inequalities by performing a brief polyhedral study of the convex hull of the feasible region of Problem~\eqref{ip}.
For simplicity, throughout this section we assume that all objects have the same number of elements $d$. We assume that $n \geq 3$ and $d \geq 2$.
Now consider the feasible region of Problem~\eqref{ip}, that is, the union of all consistent partial maps corresponding to $n$ objects each
of which consists of $d$ elements.
Any consistent partial map $X(i,j)$, $1 \leq i < j \leq n$, can be equivalently represented by a binary vector $\bar X $ in $\R^{\binom{n}{2} d^2}$ defined as
\begin{equation}\label{map2vec}
\bar X = ({\rm vec}(X(1,1))^T, {\rm vec}(X(1,2))^T, \cdots, {\rm vec}(X(1,n))^T, {\rm vec}(X(2,1))^T \cdots, {\rm vec}(X(n,n))^T )^T,
\end{equation}
where ${\rm vec}(X(i,j))$ denotes the vectorization of matrix $X(i,j)$. We refer to the convex hull of such binary vectors corresponding to all consistent partial maps as the \emph{joint matching polytope} and denote it by $\C_{n,d}$.
In the following, we refer to a vertex of $\C_{n,d}$ by specifying the corresponding partial maps, with the understanding that transformation~\eqref{map2vec}
is applied to construct a vertex from the partial maps. For notational simplicity, we sometimes use a set theoretic notation to specify the partial maps.
Namely, for any $q \in [d]$ and $i \in [n]$, we denote by $i_q$ the $q$-th element of object $\S_i$. The set $\M_X$ then represents the consistent partial maps
$X(i,j)$, $1 \leq i < j \leq n$, as follows:
$(\hat i_t, \hat j_q, \hat k_l) \in \M_X$ if and only if $\hat i_t$, $\hat j_q$ and $\hat k_l$ correspond to an element $u$ in the universe $\U$, \ie
$X_{tq}(\hat i, \hat j)=X_{ql}(\hat j, \hat k)=X_{tl}(\hat i,\hat k)=1$. Moreover, no other element of $\cup_{i \in [n]}\S_i$, corresponds to $u$.
We also assume that $\hat i_t \neq  \hat j_q \neq \hat k_l$, where we define $\hat i_t =  \hat j_q$ when $\hat i = \hat j$ and $t =q$.

With the objective of constructing strong LP relaxations for joint object matching, in the following, we conduct a brief polyhedral study of the joint matching polytope. Namely, we obtain several classes of facet-defining inequalities for this polytope.
%This in turn enables us to construct a strong LP relaxation for Problem~\eqref{jom}.
To this end, we first establish a fundamental property of the joint matching polytope.

\begin{proposition}\label{fullDim}
The joint matching polytope $\C_{n,d}$, $n \geq 3$, $d \geq 2$ is full dimensional, \ie $\dim(\C_{n,d}) = \binom{n}{2} d^2$.
\end{proposition}
\begin{proof}
To prove the statement, it suffices to show that $\C_{n,d}$ contains $\binom{n}{2} d^2 + 1$ affinely independent points.
Let the first point correspond to partial maps with no paired elements, \ie $X(i,j) = 0$ for all $1 \leq i < j \leq n$.
Next for each $\hat i, \hat j$ with $1\leq \hat i < \hat j \leq n$, and each $\hat t, \hat q \in [d]$ consider partial maps with $\M_X = \{(\hat i_{\hat t}, \hat j_{\hat q})\}$, \ie $X_{\hat t \hat q} (\hat i, \hat j) =1$ and $X_{tq}(i,j) = 0$, otherwise.
Clearly, we have a total number of $\binom{n}{2} d^2$ such partial maps. It then follows that together with the zero vertex, the corresponding vertices constitute a collection of $\binom{n}{2} d^2 + 1$ affinely independent points in $\C_{n,d}$.
\end{proof}

\begin{proposition}\label{facet1}
The nonnegativity constraints $X_{tq}(i,j) \geq 0$ for all $t,q \in [d]$ and for all $1 \leq i < j \leq n$, are facet-defining for $\C_{n,d}$.
\end{proposition}

\begin{proof}
Without loss of generality, we show that inequality $X_{11}(1,2) \geq 0$ defines a facet of $\C_{n,d}$.
Consider a nontrivial valid inequality $\sum_{1\leq i < j \leq n}{\sum_{t,q \in [d]}{\alpha_{tq}(i,j) X_{tq}(i,j)}} \leq \beta$ for $\C_{n,d}$ that is satisfied tightly by all consistent partial maps  $X(i,j)$, $1\leq i < j \leq n$
in which $X_{11}(1,2) = 0$. In the following, we show that the two inequalities $X_{11}(1,2) \geq 0$  and $\alpha X \leq \beta$ coincide up to a positive scaling which by full dimensionality of $\C_{n,d}$ (see Proposition~\ref{fullDim}) implies $X_{11}(1,2) \geq 0$ is facet-defining.

First consider the partial maps corresponding to no matched pairs, \ie $X(i,j) = 0$ for all $1\leq i < j \leq n$. Substituting this point in $\alpha X = \beta$ yields
$\beta = 0$. Next consider some $\hat t, \hat q \in [d]$ and $1\leq \hat i < \hat j \leq n$ such that $(\hat i, \hat j, \hat t,\hat q) \neq (1,2,1,1)$ and consider the partial maps
with $\M_X = \{(\hat i_{\hat t}, \hat j_{\hat q})\}$; substituting this point in  $\alpha X = \beta$ gives $\alpha_{\hat t \hat q}(\hat i,\hat j) = 0$. It then follows that $\alpha_{tq}(i,j) = 0$ for all $t,q \in [d]$ and, for all $1 \leq i < j \leq n$ such that $(i,j,t,q) \neq (1,2,1,1)$. Hence, inequality $\alpha X \leq \beta$ can be written as
$\alpha_{11}(1,2) X_{11} (1,2) \leq 0$. Since this inequality is valid for $\C_{n,d}$, we have $\alpha_{11}(1,2)  < 0$ and this completes the proof.
\end{proof}

\begin{proposition}\label{facet2}
Inequalities~\eqref{subD} define facets of the joint matching polytope.
\end{proposition}
\begin{proof}
Without loss of generality, we show that inequality $\sum_{q \in [d]}{X_{1q}(1,2)} \leq 1$ defines a facet of $\C_{n,d}$.
As in the proof of Proposition~\ref{facet1}, denote by $\alpha X \leq \beta$ a nontrivial valid inequality for $\C_{n,d}$ that is binding at all consistent partial maps satisfying
$\sum_{q \in [d]}{X_{1q}(1,2)} = 1$. Let $\hat q \in [d]$ and consider the partial maps with $\M_X = \{(1_1, 2_{\hat q})\}$.
Substituting this point in $\alpha X = \beta$ yields:
\begin{equation}\label{e1}
\alpha_{1q} (1,2) = \beta, \quad \forall q \in [d].
\end{equation}
Next, for each $\hat q \in [d]$, consider any partial maps of the form $\M_X = \{(1_1,2_{\hat q}), (i_t, j_l)\}$. Notice that by definition of $\M_X$, we have $i_t \neq 1_1$, $i_t \neq 2_{\hat q}$,
$j_l \neq 1_1$, and $j_l \neq 2_{\hat q}$.
% consisting of two \emph{distinct} matched pairs (\ie the four elements in the two matched pairs are distinct), where the first matched pair consist of the first element of the first object and the $\hat q$th element of the second object for some $\hat q \in [d]$.
Substituting such points in $\alpha X = \beta$ and using~\eqref{e1} give:
\begin{equation}\label{e2}
\alpha_{tq} (i,j) = 0, \quad \forall 1 \leq i < j \leq n, \; \forall t,q \in [d] \quad {\rm such \; that} \quad (i,t) \neq (1,1).
\end{equation}
Finally for any $j \in [n] \setminus \{1,2\}$ and any $\hat t, \hat q \in [d]$, consider the partial maps with $\M_X = \{(1_1, 2_{\hat t}, j_{\hat q})\}$.
% in which the first element of the first object is matched with the $\hat t$th element of the second object for some $\hat t \in [d]$
%and with the $\hat q$th element of the $j$th object for some $\hat q \in [d]$ and some $j \in [n] \setminus \{1,2\}$.
Substituting such points in $\alpha X = \beta$ and using~\eqref{e1} and~\eqref{e2} give:
\begin{equation}\label{e3}
\alpha_{1q} (1,j) = 0, \quad \forall 2 < j \leq n, \; \forall q \in [d].
\end{equation}
By~\eqref{e1}-\eqref{e3}, inequality $\alpha X \leq \beta$ can be written as $\beta \sum_{q \in [d]}{X_{1q}(1,2)} \leq \beta$. Moreover, since this inequality is valid for $\C_{n,d}$, it follows that $\beta > 0$ and this completes the proof.
\end{proof}

The following proposition establishes the strength of consistency inequalities; that is, consistency inequalities~\eqref{superg} are facet-defining for
$\C_{n,d}$.

\begin{proposition}\label{facet3}
Consistency inequalities~\eqref{superg} define facets of the joint matching polytope for all nonempty $D_1, D_2 \subseteq [d]$.
\end{proposition}
\begin{proof}
Without loss of generality, we show that for any nonempty $D_1, D_2 \subseteq [d]$,
the inequality
\begin{equation}\label{ineq1}
\sum_{t\in D_1}{X_{1t}(1,2)} - \sum_{t \in D_1} {\sum_{q \in D_2 }{X_{tq} (2,3)}}  + \sum_{q \in D_2 }{X_{1q}(1,3)} \leq 1,
\end{equation}
defines a facet of $\C_{n,d}$. Denote by $\alpha X \leq \beta$ a nontrivial valid inequality for $\C_{n,d}$ that is binding at all consistent partial maps satisfying inequality~\eqref{ineq1} at equality.
Notice that inequality~\eqref{ineq1} is binding at a point if in the corresponding partial maps there exists $e \in \M_X$ with $1_1 \in e$ such that (i) $2_t\in e$ for some $t \in D_1$
and $3_q \notin e$ for all $q \in D_2$  or (ii) $2_t \notin e$ for all $t \in D_1$ and $3_q\in e$ for some $q \in D_2$ or (iii) $2_t \in e$ for some $t \in D_1$ and
$3_q \in e$ for some $q \in D_2$.
%the first element of the first object is matched with the first element of the second object (resp. third object) but is not matched with the first element of the third object (resp. second object) or (ii) the first element of the first object is matched with the first element of the second object and the first element of the third object.
First consider the partial maps with $\M_X = \{(1_1,2_t)\}$  for some $t \in D_1$ (resp. $\M_X = \{(1_1,3_q)\}$ for some $q \in D_2$).
%consisting of one matched pair, that is the first element of the first object with the first element of the second object (resp. third object).
Substituting in $\alpha X = \beta$ yields:
\begin{equation}\label{f1}
\alpha_{1t} (1,2) = \alpha_{1q} (1,3) = \beta, \quad \forall t \in D_1, \; \forall q \in D_2.
\end{equation}
Next, consider the partial maps with $\M_X = \{(1_1,2_t,3_q)\}$ for some $t \in D_1$ and some $q \in D_2$;
%in which the only matched elements are the first elements of the first, second and third objects.
substituting this point in $\alpha X = \beta$ and using~\eqref{f1} gives:
\begin{equation}\label{f2}
\alpha_{tq} (2,3) = -\beta, \quad \forall t \in D_1, q \in D_2.
\end{equation}
Now let $t \in D_1$, consider any partial maps of the form $\M_X = \{(1_1,2_t),(i_r,j_s)\}$ (resp. let $q \in D_2$, and consider any partial map of the form
$\M_X = \{(1_1,3_q),(i_r,j_s)\}$).
% consisting of two distinct matched pairs (\ie four distinct elements)
%one of which is the first element of the first object with the first element of the second object (resp. third object).
Substituting these points in $\alpha X = \beta$ and using~\eqref{f1} yield:
\begin{align}\label{f3}
\alpha_{tq} (i,j) = 0, \quad & \forall 1 \leq i < j \leq n, \; \forall t,q \in [d], \; {\rm such \; that} \; (t,i) \neq (1,1) \; {\rm and} \\
& (t,q,i,j) \neq  (\tilde t, \tilde q,2,3), \; {\rm for \; some} \; \tilde t \in D_1, \; \tilde q \in D_2 \nonumber.
\end{align}
Let $t \in D_1$, for any $3 < j \leq n$ and any $q \in [d]$, consider the partial maps of the form $\M_X = \{(1_1, 2_t, j_q)\}$.
%the only matched elements are the first element of object one, the first element of object two (resp. first element of object three) and
%$q$th element of object $j$ for some $q \in [d]$ and some $3 < j \leq n$.
Substituting in $\alpha X = \beta$ and using~\eqref{f1} and~\eqref{f3} yield:
\begin{equation}\label{f4}
\alpha_{1q} (1,j) = 0, \quad \forall 3 < j \leq n, \; \forall q \in [d].
\end{equation}% 1_1, 3_q , q \in [d] \setminus D_2,  1_1, 2_t , p \in [d] \setminus D_1
Finally, let $t \in D_1$; for any $q \in [d] \setminus D_2$, consider the partial maps $\M_X = \{(1_1,2_t,3_q)\}$.
Substituting in $\alpha X = \beta$ and using~\eqref{f1} and~\eqref{f3} yield
\begin{equation}\label{f5}
\alpha_{1q} (1,3) = 0, \quad \forall q \in [d] \setminus D_2.
\end{equation}
Similarly, let $t \in D_2$; for any $q \in [d] \setminus D_1$, consider the partial maps $\M_X = \{(1_1,2_q, 3_t)\}$ to obtain
\begin{equation}\label{f7}
\alpha_{1q} (1,2) = 0, \quad \forall q \in [d] \setminus D_1.
\end{equation}
From~\eqref{f1}-\eqref{f7}, it follows that $\alpha X \leq \beta$ can be written as
$$\beta\Big( \sum_{t \in D_1} {X_{1t}(1,2)} - \sum_{t\in D_1}{\sum_{q \in D_2}{X_{tq} (2,3)}} + \sum_{q \in D_2}{X_{1q}(1,3)} \Big) \leq \beta.$$
By the validity of the latter we have
$\beta > 0$ and this completes the proof.
\end{proof}

Finally, we consider block consistency inequalities~\eqref{duperg}.
The next proposition indicates that all non-redundant block consistency inequalities~\eqref{duperg}; namely, those with $|D_1| + |D_2| > |D_3|$ , are facet-defining.
The proof is given in Section~\ref{appendix1}.

\begin{proposition}\label{newfacet}
Inequalities~\eqref{duperg} define facets of the joint matching polytope for all nonempty $D_1, D_2, D_3 \subseteq [d]$ such that $|D_1|+|D_2| > |D_3|$.
\end{proposition}

\begin{comment}
{\color{red} Notice that the feasible region of the permutation synchronization problem is a face of the joint matching polytope whose affine hull is defined by changing inequalities~\eqref{subD} to equalities, one can check the dimension of this face is $d^2-1$ lower the dimension of the joint matching polytope. Nonnegativity inequalities remain facet defining and some of consistency inequalities remain facet-defining as well. Not sure if it worth putting the additional proofs in.}
\end{comment}

It is interesting to note that the polytope defined by
facet-defining inequalities of Propositions~\ref{facet1}-\ref{newfacet} does \emph{not} coincide with the joint matching polytope even for $n=3$ and $d=2$.
The following example demonstrates this fact:

\begin{example}
Let $n =3$ and $d =2$; it can be checked that the following satisfies all facet-defining inequalities of Propositions~\ref{facet1}-\ref{newfacet}:
\begin{equation}\label{ex6}
X(1,2) =
\begin{pmatrix}
0 & \frac{1}{2} \\
\frac{1}{2} & 0
\end{pmatrix}, \qquad
X(2,3) =
\begin{pmatrix}
\frac{1}{2} & 0 \\
\frac{1}{2} & \frac{1}{2}
\end{pmatrix}, \qquad
X(1,3) =
\begin{pmatrix}
\frac{1}{2} & \frac{1}{2} \\
0 & \frac{1}{2}
\end{pmatrix}.
\end{equation}
Now consider the following inequality:
\begin{align*}
-&X_{11}(1,2)-X_{12}(1,2)-X_{22}(1,2)+X_{11}(2,3)+X_{21}(2,3)+X_{22}(2,3)+X_{11}(1,3)\\
 +&X_{12}(1,3) +X_{22}(1,3)\leq 2.
\end{align*}
To see the validity of the above inequality, note that by inequalities~\eqref{subD} we have $X_{11}(2,3)+X_{21}(2,3)+X_{22}(2,3) \leq 2$
and $X_{11}(1,3) +X_{12}(1,3) +X_{22}(1,3)\leq 2$. Moreover, if $X_{11}(2,3)+X_{21}(2,3)+X_{22}(2,3) =2$, then we must have $X_{11}(2,3) = X_{22}(2,3) =1$.
This in turn implies that if $X_{11}(1,3) =1$ (resp. $X_{12}(1,3) =1$ and $X_{22}(1,3) =1$), then $X_{11}(1,2) =1$ (resp. $X_{12}(1,2) = 1$ and $X_{22}(1,2) = 1$).
By symmetry, the validity of the inequality for the case with $X_{11}(1,3) +X_{12}(1,3) +X_{22}(1,3) = 2$ follows.
Substituting~\eqref{ex6} in the above inequality yields $-\frac{1}{2} + \frac{3}{2}+\frac{3}{2} \not\leq 2$, implying the proposed facet defining inequalities do not characterize $\C_{3,2}$.
\end{example}

Obtaining a linear characterization of cycle consistency for joint object matching remains an open question and is a subject of future research.

\subsection{A new LP relaxation for joint object matching}
We propose an LP relaxation of Problem~\eqref{ip} whose feasible set is defined by facet-defining inequalities of Propositions~\ref{facet1},~\ref{facet2}
and~\ref{facet3}:
\begin{align}\label{lp2}
\tag{LPF}
\min \; & \sum_{(i,j) \in \E}{\left\langle{\bf 1}_{d_i}{\bf 1}_{d_j}^T-2 X^{\rm in}(i,j), \; X(i,j)\right\rangle}\nonumber\\
\st \;  & X(i,j) {\bf 1}_{d_j} \leq {\bf 1}_{d_j}, \; X^T(i,j) {\bf 1}_{d_i} \leq {\bf 1}_{d_i} , \quad \forall 1 \leq i < j \leq n\nonumber \\
%\left\{
%\begin{array}{ll}
&-\sum_{t \in D_1}{\sum_{q \in D_2}{X_{tq}(i,j)}}+\sum_{q \in D_2}{X_{ql}(j,k)} +\sum_{t \in D_1}{X_{tl}(i,k)}\leq 1, \nonumber\\
& \qquad \qquad \qquad \forall D_1 \subseteq [d_i], \; D_2 \subseteq [d_j], \; l \in [d_k], \; D_1, D_2 \neq \emptyset, \; 1 \leq i < j  < k \leq n\nonumber\\
&\sum_{t \in D_1}{X_{lt}(i,j)} -\sum_{t \in D_1}{\sum_{q\in D_2}{X_{tq}(j,k)}}+\sum_{q\in D_2} {X_{lq}(i,k)}\leq 1,\nonumber \\
& \qquad \qquad \qquad \forall D_1 \subseteq [d_j], \; D_2 \subseteq [d_k], \; l \in [d_i], \; D_1, D_2 \neq \emptyset, \; 1 \leq i < j  < k \leq n\nonumber\\
&\sum_{t \in D_1}{X_{tl}(i,j)} +\sum_{q\in D_2}{X_{lq}(j,k)}-\sum_{t \in D_1}{\sum_{q\in D_2} {X_{tq}(i,k)}}\leq 1,\nonumber \\
&\qquad \qquad \qquad \forall D_1 \subseteq [d_i], \; D_2 \subseteq [d_k],\; l \in [d_j], \; D_1, D_2 \neq \emptyset, \; 1 \leq i < j  < k \leq n \nonumber\\
%\end{array}\right.\\
& X(i,j) \geq 0, \quad \forall 1 \leq i < j \leq n. \nonumber
\end{align}
Recall that the number of consistency inequalities~\eqref{superg} in Problem~\eqref{lp2} is exponential in the number of elements $d_i$ in object $i$.
However, by Proposition~\ref{p:separate}, separating over these inequalities can be done in time that is polynomial in $n, d_i$, $i \in [n]$. By polynomial equivalence of separation and optimization~\cite{gls88}, it then follows that Problem~\eqref{lp2} is solvable in polynomial time.

%We conclude this section by remarking that while for notational simplicity, in this paper, we assume that all objects contain the same number of elements $d$, our results can be generalized in a straightforward manner to the case where objects contain different number of elements.

As we detailed in Section~\ref{sec:intro}, in joint object matching the size of the universe is not known a priori.
However, some studies~\cite{CheGuiHua14,ZZD15} have proposed to employ~\emph{an upper bound estimate} on the size of the universe to tighten their SDP relaxation.
Such an upper bound estimate can also be used to tighten our proposed LP relaxation. Since a reliable estimation of $\hat m$ is based on a good understanding of the specific application and the input data, and hence is beyond the scope of this paper, we do not include it in our numerical experiments. However, for completeness, in Section~\ref{appendix2} we present a class of valid inequalities that can significantly tighten the LP relaxation when an estimate on the size of the universe is available.

\section{Recovery for permutation group synchronization}
\label{sec:dualCertificate}

In this section, we study the recovery properties of the proposed LP relaxation under a popular stochastic model for the input maps.
We consider the case with $d_i =d  = m$ for all $i \in [n]$; this special case is often referred to as permutation group synchronization in the literature~\cite{PacKonSin13,PaKoSaSi14,Ling20}.
For simplicity, we assume that the input consists of all possible pair-wise matchings, that is, we assume that the map graph is complete.
Define $a_{tq}(i,j) := 1 - 2 X_{tq}^{\rm in}(i,j)$ for all $t,q \in [d]$ and for all $1 \leq i < j \leq n$.
As the first step, in this paper, we focus on the basic LP relaxation, that is, the LP consisting of all consistency inequalities with $|D_1| = |D_2| = 1$:
\begin{align}\label{primal}
\tag{P}
\min \quad & \sum_{1\leq i < j \leq n}{\sum_{t,q \in [d]}{a_{tq}(i,j)X_{tq}(i,j)}}\nonumber\\
\st \quad  & \sum_{q \in [d]}{X_{tq}(i,j)} = 1, \quad \forall t \in [d], \; \forall 1 \leq i < j \leq n, \label{rsum} \\
          & \sum_{t \in [d]}{X_{tq}(i,j)} = 1, \quad \forall q \in [d], \; \forall 1 \leq i < j \leq n, \label{csum} \\
         &  \left\{
\begin{array}{ll}
X_{lt}(i,j) + X_{tq}(j,k) -X_{lq}(i,k) \leq 1,\\
X_{lt}(i,j) - X_{tq}(j,k) +X_{lq}(i,k) \leq 1, \\
- X_{lt}(i,j) + X_{tq}(j,k) +X_{lq}(i,k)   \leq 1,\\
\end{array} \right. \forall l,t,q \in [d], \; \forall 1 \leq i < j < k\leq n, \label{tri}\\
           & X_{tq}(i,j) \geq 0, \quad \forall t, q \in [d], \; \forall 1 \leq i < j \leq n. \label{nonnega}
\end{align}

\paragraph{The random corruption model.} To analyze the quality of the proposed LP relaxation,
we consider the \emph{random corruption model}~\cite{CheGuiHua14,Ling20}, defined as follows:
\begin{equation*}
X^{\rm in} (i,j)=\begin{cases}
      Y(i) Y^T(j) &  \text{with probability $p_{\rm true}$},\\
      {\bf{P}}_{ij} & \text{with probability $1-p_{\rm true}$},
       \end{cases}
\end{equation*}
where $Y(i) \in \{0,1\}^{d \times d}$, $i \in [n]$ encodes the correspondences between the $i$-th object and the universe, ${\bf P}_{ij}$ is an independent
random permutation matrix uniformly sampled from $d!$ permutation matrices, and $p_{\rm true} \in [0,1]$ denotes the probability that each $X^{\rm in} (i,j)$ coincides with the ground truth. Without loss of generality, we assume
that the elements of each object are ordered so that the $q$-th element of all objects are matched with each other for all $q \in [d]$; that is,
$X^{\rm in} (i,j) = I_d$, whenever the input matches the ground truth. For each $1 \leq i < j \leq n$, denote by $z_{ij}$ a Bernoulli random variable with parameter $p_{\rm true}$, independently drawn from ${\bf P}_{ij}$.
It then follows that:
\begin{equation}\label{unifmodel}
X^{\rm in} (i,j)=z_{ij} I_d + (1-z_{ij}) {\bf P}_{ij}, \quad \forall 1 \leq i < j \leq n.
\end{equation}
%Moreover, we assume that for each $1 \leq i < j \leq n$, the input $X^{\rm in} (i,j)$ is a permutation matrix. However, we assume that the input data is not cycle consistent; that is the matrices
%$X^{\rm in} (i,j)$, $1 \leq i < j \leq n$ do not satisfy~\eqref{consisPerm}.
We are interested in addressing the following question:

\vspace{0.1in}
\emph{What is the maximum level of corruption under which the basic LP relaxation recovers the ground truth with high probability?}
\vspace{0.1in}

Recall that we say an optimization problem (\eg the LP relaxation) recovers the ground truth if its unique optimal solution coincides with the ground truth.
Moreover, by high probability we imply that the probability tends to 1 as the number of objects $n \rightarrow \infty$.
%That is, for our asymptotic analysis, in the same vein as previous results for joint object matching, we assume that $d$ is fixed while $n \rightarrow \infty$.
The following theorem states that the basic LP relaxation recovers the ground truth with high probability provided that the corruption level is below $40\%$:

\begin{theorem}\label{main}
Suppose that input maps $X^{\rm in} (i,j)$ for all $1\leq i < j \leq n$ are generated according to the random corruption model~\eqref{unifmodel} and assume that $d \in  o(({n/\log(n)})^{\frac 14})$.
Then Problem~\eqref{primal} recovers the ground truth maps with high probability if $p_{\rm true} > 0.585$.
\end{theorem}

%{\color{red} Put the plot for intermediate $d$ here.}
In the remainder of this section, we prove Theorem~\ref{main} via a number of intermediate steps stated in six lemmata.
Recall that for our ground truth, we assume
that the elements of each object are ordered
in a way that for each $q \in [d]$,  $q$-th elements of all objects are matched with each other. That is,
for each $1 \leq i < j \leq n$, we have $\bar X_{qq} (i,j) = 1$ for all $q \in [d]$ and $\bar X_{tq}(i,j) = 0$ for all $t \neq q \in [d]$.
We first obtain a deterministic sufficient condition under which $\bar X$ is the unique optimal solution of Problem~\eqref{primal}.
Subsequently, we focus on the random corruption model
and prove that our deterministic condition implies the recovery guarantee of Theorem~\ref{main}.

\subsection{A deterministic condition for recovery}
In the following we obtain a deterministic  sufficient condition under which the ground truth is the unique optimal solution of the basic LP.
First, in Lemma~\ref{optimality}, we obtain a sufficient condition under which $\bar X$
is an optimal solution of Problem~\eqref{primal}, where we assume that each input map $X^{\rm in}(i,j)$, $1\leq i < j \leq n$ is a permutation matrix.
Subsequently, in Lemma~\ref{unique}, we examine the uniqueness of $\bar X$. In the following, we define $a(j,i) := a^T(i,j)$ for all $1 \leq i < j \leq n$.

\begin{lemma}\label{optimality}
For each $1\leq i < j \leq n$ and $t \in [d]$, define
\begin{equation}\label{kkappa}
\begin{split}
\kappa_t(i,j) &:= \frac{1}{n} \sum_{l \in [d] \setminus \{t\}}{\sum_{k \in [n] \setminus \{i,j\}}{\min\left\{\frac{a_{lt}(k,i)}{\alpha}+\frac{a_{tt}(k,j)}{\beta d},
\frac{a_{lt}(k,j)}{\alpha}+\frac{a_{tt}(k,i)}{\beta d}\right\}}}\\
&+\frac{1}{n} \Big(\frac{1}{2\alpha}-\frac{d-1}{2\beta d}\Big)\Big(\sum_{k\in [n]\setminus \{i\}} {a_{tt}(k,i)}+\sum_{k\in [n]\setminus \{j\}} {a_{tt}(k,j)}\Big)-\frac{(d-2)}{\alpha}+1\\
&+\frac{1}{2n} \sum_{l \in [d] \setminus \{t\}}{\Big(\frac{a_{tl}(i,j)+a_{tl}(j,i)}{\alpha}+2\frac{a_{tt}(i,j)}{\beta d}\Big)},
\end{split}
\end{equation}
and for each $1 \leq i < j \leq n$ and $t,q \in [d]$, define
\begin{align}\label{ldelta}
\delta_{t \rightarrow q}^{i \rightarrow j} := \frac{1}{n}\sum_{\substack{k \in [n] \setminus\{i,j\}: \\ a_{tt}(k,i)=-1}}
{\Big(\frac{a_{tq}(k,j)-a_{tq}(i,j)}{\alpha}+\frac{a_{tt}(i,j)-a_{tt}(k,j)}{\beta d}\Big)}.
\end{align}
Let $\alpha > 0$ and $\beta > 0$ be parameters independent of $n,d$.
Then $\bar X$ is an optimal solution of Problem~\eqref{primal}, if for each $t,q \in [d]$ and for each $1 \leq i < j \leq n$, the following conditions
are satisfied:
\begin{itemize}[leftmargin=*]
\item [(i)] if $a_{tt}(i,j) = a_{qq}(i,j) = a_{tq}(i,j) =1$, then
\begin{equation}\label{cond1}\tag{C1}
\delta_{t \rightarrow q}^{i \rightarrow j} \geq 0, \quad \delta_{q \rightarrow t}^{j \rightarrow i} \geq 0,
%\frac{1}{n}\sum_{\substack{k \in [n] \setminus\{i,j\}: \\ a_{tt}(k,i)=-1}}{\Big(\frac{a_{tq}(k,j)-1}{\alpha}+\frac{1-a_{tt}(k,j)}{\beta d}\Big)} \geq 0,
\end{equation}
%and
%\begin{equation}\label{cond2}\tag{C2}
%\frac{1}{n}\sum_{\substack{k \in [n] \setminus\{i,j\}: \\ a_{qq}(k,j)= -1}}{\Big(\frac{a_{tq}(i,k)-1}{\alpha}+\frac{1-a_{qq}(i,k)}{\beta d}\Big)} \geq 0.
%\end{equation}
%
\item [(ii)] if $a_{tt}(i,j) = a_{qq}(i,j)=1$,  $a_{tq}(i,j) =-1$, then
\begin{equation}\label{cond2}\tag{C2}
\delta_{t \rightarrow q}^{i \rightarrow j} \geq 1, \quad \delta_{q \rightarrow t}^{j \rightarrow i} \geq 1,
%\frac{1}{n}\sum_{\substack{k \in [n] \setminus\{i,j\}: \\ a_{tt}(k,i)=-1}}{\Big(\frac{a_{tq}(k,j)-1}{\alpha}+\frac{1-a_{tt}(k,j)}{\beta d}\Big)} \geq 0,
\end{equation}
%
%\begin{equation}\label{cond3}\tag{C3}
%\frac{1}{n}\sum_{\substack{k \in [n] \setminus\{i,j\}: \\ a_{tt}(k,i)=-1}}{\Big(\frac{a_{tq}(k,j)+1}{\alpha}+\frac{1-a_{tt}(k,j)}{\beta d}\Big)}\geq 1,
%\end{equation}
%and,
%\begin{equation}\label{cond4}\tag{C4}
%\frac{1}{n} \sum_{\substack{k \in [n] \setminus\{i,j\}: \\ a_{qq}(k,j)= -1}}{\Big(\frac{a_{tq}(i,k)+1}{\alpha}+\frac{1-a_{qq}(i,k)}{\beta d}\Big)} \geq 1.
%\end{equation}
%
\item [(iii)] if $a_{tt}(i,j) = a_{qq}(i,j) = -1$, $a_{tq}(i,j)= 1$, then
\begin{align*}\label{cond3}\tag{C3}
\kappa_t(i,j) +\frac{1}{2}\Big(\delta_{t \rightarrow q}^{i \rightarrow j} +\delta_{q \rightarrow t}^{j \rightarrow i}\Big)\geq 0, \quad
\kappa_q(i,j) +\frac{1}{2}\Big(\delta_{t \rightarrow q}^{i \rightarrow j} +\delta_{q \rightarrow t}^{j \rightarrow i}\Big)\geq 0,
\end{align*}
%
%\begin{align*}\label{cond5}\tag{C5}
%\begin{split}
%&  \kappa_t(i,j) +\frac{1}{2n}\sum_{\substack{k \in [n] \setminus\{i,j\}: \\ a_{tt}(k,i)=-1}}{\Big(\frac{a_{tq}(k,j)-1}{\alpha}+\frac{-1-a_{tt}(k,j)}{\beta d}\Big)}\\
%& +\frac{1}{2n}\sum_{\substack{k \in [n] \setminus\{i,j\}: \\ a_{qq}(k,j)= -1}}{\Big(\frac{a_{tq}(i,k)-1}{\alpha}+\frac{-1-a_{qq}(i,k)}{\beta d}\Big)}\geq 0,
%\end{split}
%\end{align*}
%and,
%\begin{align*}\label{cond6}\tag{C6}
%\begin{split}
%&  \kappa_q(i,j)+\frac{1}{2n}\sum_{\substack{k \in [n] \setminus\{i,j\}: \\ a_{tt}(k,i)=-1}}{\Big(\frac{a_{tq}(k,j)-1}{\alpha}+\frac{-1-a_{tt}(k,j)}{\beta d}\Big)}\\
%&+\frac{1}{2n}\sum_{\substack{k \in [n] \setminus\{i,j\}: \\ a_{qq}(k,j)= -1}}{\Big(\frac{a_{tq}(i,k)-1}{\alpha}+\frac{-1-a_{qq}(i,k)}{\beta d}\Big)}\geq 0.
%\end{split}
%\end{align*}
%
\item [(iv)] if $a_{tt}(i,j) = 1$, $a_{qq}(i,j) = -1$, $a_{tq}(i,j)= 1$, then
\begin{equation}\label{cond4}\tag{C4}
  \kappa_q(i,j) +\delta_{t \rightarrow q}^{i \rightarrow j} +\delta_{q \rightarrow t}^{j \rightarrow i}\geq 0,
\end{equation}

\item [(v)] if $a_{tt}(i,j) = -1$, $a_{qq}(i,j) = 1$, $a_{tq}(i,j)= 1$, then
\begin{equation}\label{cond5}\tag{C5}
  \kappa_t(i,j) +\delta_{t \rightarrow q}^{i \rightarrow j} +\delta_{q \rightarrow t}^{j \rightarrow i}\geq 0.
\end{equation}

%\begin{align*}\label{cond7}\tag{C7}
%\begin{split}
%&  \kappa_q(i,j)+\frac{1}{n}\sum_{\substack{k \in [n] \setminus\{i,j\}: \\ a_{tt}(k,i)=-1}}{\Big(\frac{a_{tq}(k,j)-1}{\alpha}+\frac{1-a_{tt}(k,j)}{\beta d}\Big)} \\
%&+\frac{1}{n}\sum_{\substack{k \in [n] \setminus\{i,j\}: \\ a_{qq}(k,j)= -1}}{\Big(\frac{a_{tq}(i,k)-1}{\alpha}+\frac{-1-a_{qq}(i,k)}{\beta d}\Big)}\geq 0.
%\end{split}
%\end{align*}

\end{itemize}
\end{lemma}

\begin{proof}
We start by constructing the dual of Problem~\eqref{primal}: define dual variables $r_t(i,j)$ for all $t \in [d]$ (resp. $c_q(i,j)$ for all $q \in [d]$ ) and for all $1 \leq i < j \leq n$ associated with equalities~\eqref{rsum} (resp. equalities~\eqref{csum}). Define dual variables $\lambda_{ltq}^1 (i,j,k)$, $\lambda_{ltq}^2 (i,j,k)$, $\lambda_{ltq}^3 (i,j,k)$ for all $l,t,q \in [d]$ and for all $1 \leq i < j < k \leq n$ for the first, the second and the third inequalities of system~\eqref{tri}, respectively.
Finally, for inequalities~\eqref{nonnega}, define dual variables
$\mu_{tq} (i,j)$ for all $t,q \in [d]$ and for all $1 \leq i < j \leq n$. It then follows that the dual of Problem~\eqref{primal} is given by:
\begin{align}\label{dual}
\tag{D}
\max \quad & -\sum_{1\leq i < j < k \leq n}{\sum_{l,t,q \in [d]}{\Big(\lambda_{ltq}^1 (i,j,k) + \lambda_{ltq}^2 (i,j,k)+ \lambda_{ltq}^3 (i,j,k)\Big)}}\nonumber\\
          &-\sum_{1\leq i < j \leq n}{\sum_{t \in [d]}{\Big(r_t(i,j)+c_t(i,j)\Big)}}\nonumber\\
\st \quad  & a_{tq}(i,j)+r_t(i,j)+c_q(i,j)-\mu_{tq} (i,j)+ \sum_{l \in [d]}\Big(\sum_{\substack{k \in [n]:\\ k <i}}{\big(\lambda_{ltq}^1 (k,i,j)-\lambda_{ltq}^2 (k,i,j)+\lambda_{ltq}^3 (k,i,j)\big)}  \nonumber\\
 +&\sum_{\substack{k\in [n]: \\ i < k < j}}{\big(-\lambda_{tlq}^1 (i,k,j)+\lambda_{tlq}^2 (i,k,j)+\lambda_{tlq}^3 (i,k,j)\big)} \nonumber\\
 +& \sum_{\substack{k\in [n]: \\ k>j}}{\big(\lambda_{tql}^1 (i,j,k)+\lambda_{tql}^2 (i,j,k)- \lambda_{tql}^3 (i,j,k)\big)}\Big)= 0, \quad \forall t,q \in [d], \forall 1 \leq i < j \leq n,\label{dineq} \\
& \lambda_{ltq}^1 (i,j, k) \geq 0, \; \lambda_{ltq}^2 (i,j,k) \geq 0, \; \lambda_{ltq}^3 (i,j,k) \geq 0, \quad  \forall l,t,q \in [d], \;  \forall 1 \leq i < j < k \leq n, \nonumber\\
& \mu_{tq} (i,j) \geq 0, \quad \forall t,q \in [d], \; \forall 1 \leq i < j \leq n.\label{munon}
\end{align}
To prove the statement,
it suffices to construct a dual feasible point $(\bar r, \bar c, \bar \lambda^1, \bar \lambda^2, \bar \lambda^3, \bar \mu)$ that satisfies complementary slackness, \ie
for all $1\leq i < j < k \leq n$ and all $l,t,q \in [d]$,
we have
\begin{itemize}
\item $\bar\lambda^1_{ltq} (i,j,k) =0$ if $l \neq t$ and $t \neq q$,
\item $\bar\lambda^2_{ltq} (i,j,k) =0$ if $l \neq t$ and $l \neq q$,
\item $\bar\lambda^3_{ltq} (i,j,k) =0$ if $l \neq q$ and $t \neq q$,
\end{itemize}
and $\bar \mu_{tt} (i,j) = 0$ for all $1 \leq i < j \leq n$ and for all $t \in [d]$.

Moreover, to construct the dual certificate, we make the following simplifications:
\begin{itemize}
\item  $\bar \lambda^1_{ttt}(i,j,k) = \bar \lambda^2_{ttt}(i,j,k) = \bar \lambda^3_{ttt}(i,j,k) = 0$, for all $t \in [d]$ an for all $1\leq i < j < k \leq n$,
\item $\bar r_t(i,j) = \bar c_t(i,j)$ for all $t \in [d]$ and for all $1 \leq i < j \leq n$.
\end{itemize}
Hence by complementary slackness and our simplifications stated above, for each $1 \leq i < j \leq n$, if $t = q$, constraint~\eqref{dineq} simplifies to
\begin{equation}\label{s1}
a_{tt}(i,j)+  2\bar r_t(i,j) + \Sigma_t (i,j)=0,
\end{equation}
where we define
\begin{align}\label{sigma}
\Sigma_t(i,j) := &\sum_{l \in [d] \setminus \{t\}}\Big(\sum_{k <i}{(\bar\lambda_{ltt}^1 (k,i,j)+\bar\lambda_{ltt}^3 (k,i,j))}
 +\sum_{i < k < j}{(\bar\lambda_{tlt}^2 (i,k,j)+\bar\lambda_{tlt}^3 (i,k,j))}\nonumber\\
& + \sum_{k>j}{(\bar\lambda_{ttl}^1 (i,j,k)+\bar\lambda_{ttl}^2 (i,j,k))}\Big),
\end{align}
and for each $1 \leq i < j \leq n$, if $t \neq q$, using inequalities~\eqref{munon} to project out $\mu_{tq}(i,j)$, constraint~\eqref{dineq} simplifies to:
\begin{equation}\label{s2}
a_{tq}(i,j)+\bar r_t(i,j) +\bar r_q(i,j)+ \Delta_{tq}(i,j) \geq 0,
\end{equation}
where we define
\begin{align}\label{delta}
\Delta_{tq}(i,j) :=&\sum_{k <i}{\big(\bar\lambda_{ttq}^1 (k,i,j)-\bar\lambda_{ttq}^2 (k,i,j)+\bar\lambda_{qtq}^3 (k,i,j)-\bar\lambda_{qtq}^2 (k,i,j)\big)}\nonumber\\
+&\sum_{i < k < j}{\big(\bar\lambda_{ttq}^2 (i,k,j)-\bar\lambda_{ttq}^1 (i,k,j) +\bar\lambda_{tqq}^3 (i,k,j)-\bar\lambda_{tqq}^1 (i,k,j)\big)}\nonumber\\
+& \sum_{k>j}{\big(\bar\lambda_{tqq}^1 (i,j,k)-\bar\lambda_{tqq}^3 (i,j,k)+\bar\lambda_{tqt}^2 (i,j,k)-\bar\lambda_{tqt}^3 (i,j,k)\big)}.
\end{align}
%From~\eqref{s1} it follows that
%\begin{equation}\label{ss2}
%\bar r_t(i,j) =
%\left\{
%\begin{array}{ll}
%-1-\Sigma_t(i,j), \quad & {\rm if} \quad a_{pp}(i,j) = 1\\
% 1-\Sigma_t(i,j), \quad & {\rm if} \quad a_{pp}(i,j) = -1\\
%\end{array} \right.
%\end{equation}
Hence, we need to determine nonnegative $\bar\lambda^1$, $\bar\lambda^2$, $\bar\lambda^3$ satisfying\eqref{s1} and~\eqref{s2}.
To this end, for each $1 \leq i < j \leq n$ and each $t \in [d]$
such that $a_{tt}(i,j) = -1$, do the following:
\begin{itemize}
\item [-] for each $1 \leq k < i$  and for each $q \in [d] \setminus \{t\}$, let
\begin{equation}\label{ff1}
\left\{\begin{array}{ll}
\bar\lambda_{qtt}^1 (k,i,j) =&\frac{1}{n}\max\left\{\frac{a_{tq}(j,k)-a_{tq}(i,k)}{\alpha}+\frac{a_{tt}(i,k)-a_{tt}(j,k)}{\beta d},\; 0\right\},\\
\bar\lambda_{qtt}^3 (k,i,j)= &\frac{1}{n}\max\left\{\frac{a_{tq}(i,k)-a_{tq}(j,k)}{\alpha}+\frac{a_{tt}(j,k)-a_{tt}(i,k)}{\beta d}, \; 0\right\},
\end{array}\right.
\end{equation}
\item [-] for each $i < k < j$ and for each $l \in [d] \setminus \{t\}$, let
\begin{equation}\label{ff2}
\left\{\begin{array}{ll}
\bar\lambda_{tqt}^2 (i,k,j) = &\frac{1}{n}\max\left\{\frac{a_{tq}(j,k)-a_{tq}(i,k)}{\alpha}+\frac{a_{tt}(i,k)-a_{tt}(j,k)}{\beta d},\; 0\right\},\\
\bar\lambda_{tqt}^3 (i,k,j)= &\frac{1}{n}\max\left\{\frac{a_{tq}(i,k)-a_{tq}(j,k)}{\alpha}+\frac{a_{tt}(j,k)-a_{tt}(i,k)}{\beta d},\; 0\right\},
\end{array}\right.
\end{equation}
\item [-] for each $j < k\leq n$ and for each $l \in [d] \setminus \{t\}$, let
\begin{equation}\label{ff3}
\left\{\begin{array}{ll}
&\bar\lambda_{ttq}^1 (i,j,k)= \frac{1}{n}\max\left\{\frac{a_{tq}(i,k)-a_{tq}(j,k)}{\alpha}+\frac{a_{tt}(j,k)-a_{tt}(i,k)}{\beta d}, \; 0\right\}, \\
&\bar\lambda_{ttq}^2 (i,j,k) =\frac{1}{n}\max\left\{\frac{a_{tq}(j,k)-a_{tq}(i,k)}{\alpha}+\frac{a_{tt}(i,k)-a_{tt}(j,k)}{\beta d},\; 0\right\}.\end{array}
\right.
\end{equation}
\end{itemize}
Moreover, for each $1 \leq i <j \leq n$ and for each $t \in [d]$ such that $a_{tt}(i,j)= 1$, do the following:
\begin{itemize}
\item [-] for each $1\leq k < i$ and for each $q \in [d] \setminus \{t\}$, let
\begin{equation}\label{ee1}
\bar \lambda_{qtt}^1(k,i,j)= \bar \lambda_{qtt}^3(k,i,j) = 0,
\end{equation}
\item [-] for each $i <k <j$ and for each $q \in [d]\setminus \{t\}$, let
\begin{equation}\label{ee2}
\bar \lambda_{tqt}^2(i,k,j) = \bar \lambda_{tqt}^3(i,k,j) = 0,
\end{equation}
\item [-] for each $j < k \leq n$ and for each $q \in [d] \setminus \{t\}$, let
\begin{equation}\label{ee3}
\bar \lambda_{ttq}^1(i,j,k) = \bar \lambda_{ttq}^2(i,j,k) = 0.
\end{equation}
\end{itemize}
Substituting~\eqref{ff1}-\eqref{ff3} in~\eqref{sigma}, gives:
%\begin{align}\label{sigma2}
%\Sigma_t(i,j) = & \frac{2(d-2)}{\alpha}+ \big(\frac{d-1}{\beta d}-\frac{1}{\alpha}\big)(\bar a_{t\rightarrow t}(i)+\bar a_{t\rightarrow t}(j))-\frac{1}{n} \sum_{l \in [d] %\setminus \{t\}}{\Big(\frac{a_{tl}(i,j)+a_{tl}(j,i)}{\alpha}+2\frac{a_{tt}(i,j)}{\beta d}\Big)}\nonumber\\
%-& \frac{2}{n} \sum_{l \in [d] \setminus \{t\}}{\sum_{k \in [n] \setminus \{i,j\}}{\min\left\{\frac{a_{lt}(k,i)}{\alpha}+\frac{a_{tt}(k,j)}{\beta d},
%\frac{a_{lt}(k,j)}{\alpha}+\frac{a_{tt}(k,i)}{\beta d}\right\}}},
%\end{align}
\begin{equation}\label{sigma2}
\Sigma_t(i,j) = 2-2\kappa_t(i,j), \qquad {\rm if} \; a_{tt}(i,j) =-1
\end{equation}
where $\kappa_t(i,j)$ is defined by~\eqref{kkappa}, and
substituting~\eqref{ee1}-\eqref{ee3} in~\eqref{sigma}, yields
\begin{equation}\label{sigma1.5}
\Sigma_t(i,j) = 0, \qquad {\rm if} \; a_{tt}(i,j) =1.
\end{equation}
%where to obtain the above expression we used the identity $|e-f| = e+f- 2\min\{e,f\}$.
Moreover, substituting~\eqref{ff1}-\eqref{ee3} in~\eqref{delta} gives
\begin{align}\label{delta2}
\Delta_{tq}(i,j) = \delta_{t \rightarrow q}^{i \rightarrow j} +\delta_{q \rightarrow t}^{j \rightarrow i},
%&\frac{1}{n}\sum_{\substack{k \in [n] \setminus\{i,j\}: \\ a_{tt}(k,i)=-1}}
%{\Big(\frac{a_{tq}(k,j)-a_{tq}(i,j)}{\alpha}+\frac{a_{tt}(i,j)-a_{tt}(k,j)}{\beta d}\Big)}\nonumber\\
%+ &\frac{1}{n}\sum_{\substack{k \in [n] \setminus\{i,j\}: \\ a_{qq}(k,j)= -1}}{\Big(\frac{a_{tq}(i,k)-a_{tq}(i,j)}{\alpha}+\frac{a_{qq}(i,j)-a_{qq}(i,k)}{\beta d}\Big)}.
\end{align}
where $\delta_{t \rightarrow q}^{i \rightarrow j}$ is defined by~\eqref{ldelta}.
Finally, substituting~\eqref{sigma2} and~\eqref{delta2} in~\eqref{s1} and~\eqref{s2}, we obtain the following five cases:
\begin{itemize}[leftmargin=*]
\item [(i)] $a_{tt}(i,j) = a_{qq}(i,j) = a_{tq}(i,j) =1$: in this case, by~\eqref{sigma1.5} we have $\Sigma_t(i,j) = \Sigma_q(i,j) = 0$ which by~\eqref{s1} gives
$\bar r_t(i,j) = \bar r_q(i,j) = -1/2$. Hence, inequality~\eqref{s2} simplifies to $\Delta_{tq}(i,j) \geq 0$. By~\eqref{delta2},
the validity of this inequality follows from condition~\eqref{cond1}.

\item [(ii)] $a_{tt}(i,j) = a_{qq}(i,j)=1$,  $a_{tq}(i,j) =-1$: in this case, by~\eqref{sigma1.5}  we have $\Sigma_t(i,j) = \Sigma_q(i,j) = 0$ and hence
by~\eqref{s1} we get $\bar r_t(i,j) = \bar r_q(i,j) = -1/2$. Hence, inequality~\eqref{s2} simplifies to $\Delta_{tq}(i,j) \geq 2$. By~\eqref{delta2},
the validity of this inequality follows from condition~\eqref{cond2}.

\item [(iii)] $a_{tt}(i,j) = a_{qq}(i,j) = -1$, $a_{tq}(i,j)= 1$: in this case by~\eqref{sigma2} we have $\Sigma_t(i,j) = 2-2\kappa_t(i,j)$ and
$\Sigma_q(i,j) = 2-2\kappa_q(i,j)$. Hence, from~\eqref{s1} it follows that $\bar r_t(i,j) = \frac{2\kappa_t(i,j)-1}{2}$
and $\bar r_q(i,j) = \frac{2\kappa_q(i,j)-1}{2}$. Then, inequality~\eqref{s2} simplifies to $\kappa_t(i,j)+\kappa_q(i,j) +\Delta_{tq}(i,j) \geq 0$.
By~\eqref{delta2} this inequality is implied by condition~\eqref{cond3}.

\item [(iv)] $a_{tt}(i,j) = 1$, $a_{qq}(i,j) = -1$, $a_{tq}(i,j)= 1$: in this case, by~\eqref{sigma1.5} we have $\Sigma_t(i,j) = 0$, and
by~\eqref{sigma2} we have $\Sigma_q(i,j) = 2-2\kappa_q(i,j)$. Hence by~\eqref{s1} we get
$\bar r_t(i,j) = -1/2$, and $\bar r_q(i,j) = \frac{2\kappa_q(i,j)-1}{2}$. Then, inequality~\eqref{s2} simplifies to $\kappa_q(i,j)+\Delta_{tq}(i,j) \geq 0$.
Finally, by~\eqref{delta2}, this inequality is equivalent to condition~\eqref{cond4}.

\item [(v)] $a_{tt}(i,j) = -1$, $a_{qq}(i,j) = 1$, $a_{tq}(i,j)= 1$: by symmetry, this case follows from case~(iv).
\end{itemize}
\end{proof}

We now provide a sufficient condition under which the ground truth $\bar X$
is the unique optimal solution of Problem~\eqref{primal}.
To this end, we use Mangasarian's characterization of uniqueness of solution in LP~\cite{Man79}:

\begin{proposition}[Part~(iv) of Theorem~2 in~\cite{Man79}]\label{mang}
Consider an LP whose feasible region is defined by $Ax =b$ and $C x \leq d$, where $x \in \R^n$ denotes the
vector of optimization variables and, $b,d, A, C$ are vectors and matrices of appropriate dimensions.
Let $\bar x$ be an optimal solution of this LP and denote by $\bar u$ the dual optimal solution corresponding to the inequality constraints.
Let $C_i$ denote
the $i$-th row of $C$. Define $K =\{i : C_i \bar x = d_i, \; \bar u_i > 0\}$, $L =\{i : C_i \bar x = d_i, \; \bar u_i = 0\}$. Let $C_K$ and $C_L$
be the
matrices whose rows are $C_i$, $i \in K$ and $C_i$, $i \in L$, respectively.
Then $\bar x$ is the unique optimal solution of the LP, if there exists no $x$ different from the zero vector satisfying
\begin{equation}\label{uc}
Ax = 0, \quad C_K x = 0, \quad C_L x \leq 0.
\end{equation}
\end{proposition}

We are now ready to establish our uniqueness result:

\begin{lemma}\label{unique}
Suppose that all inequalities in conditions~\eqref{cond1}-\eqref{cond5} are strictly satisfied.
Then $\bar X$ is the unique optimal solution of Problem~\eqref{primal}.
\end{lemma}

\begin{proof}
Consider the dual certificate $(\bar r, \bar c, \bar \lambda^1, \bar \lambda^2, \bar \lambda^3, \bar \mu)$ constructed in Lemma~\ref{optimality}; to achieve uniqueness, we perturb this point by requiring $\bar \mu_{tq}(i,j) > 0$, for all $t \neq q \in [d]$ and for all $1 \leq i < j \leq n$; this can be enforced by replacing inequalities~\eqref{s2} by strict inequalities, \ie
$$a_{tq}(i,j)+\bar r_t(i,j) +\bar r_q(i,j)+ \Delta_{tq}(i,j) > 0, \quad \forall t \neq q \in [d], \quad \forall 1 \leq i < j \leq n.$$
It can be checked that the above condition is satisfied if inequalities~\eqref{cond1}-\eqref{cond5} are strictly satisfied.
Hence it remains to show that if $\bar \mu_{tq}(i,j) > 0$, for all $t \neq q$ and for all $1 \leq i < j \leq n$, then $\bar X$ is the unique optimal solution. To this end, we make use of Proposition~\ref{mang}; suppose that $\bar X$ is not the unique optimal solution
of Problem~\eqref{primal}. Then there exists $\tilde X$ not identically zero, satisfying condition~\eqref{uc}. Since $\bar \mu_{tq}(i,j) > 0$, from condition~\eqref{uc}, it follows that
$\tilde X_{tq}(i,j) = 0$ for all $t \neq q \in [d]$ and for all $1 \leq i < j \leq n$. Moreover, by condition~\eqref{uc}, we must have $\sum_{q \in [d]}{\tilde X_{tq}(i,j)} = 0$
for all $t \in [d]$. Therefore, $\tilde X_{tt}(i,j) = 0$
for all $t \in [d]$ and for all $1 \leq i < j \leq n$, which contradicts with the assumption that $\tilde X$ must have a nonzero element. Hence $\bar X$ is the unique optimal solution
of Problem~\eqref{primal}.
\end{proof}

\subsection{Recovery under the random corruption model}
We now consider the random corruption model~\eqref{unifmodel}; we prove that our deterministic recovery condition given by conditions~\eqref{cond1}-\eqref{cond5} implies that
under the random corruption model,  Problem~\eqref{primal} recovers the ground truth with high probability, if $p_{\rm true} > 0.585$.
In this section, for an event $A$, we denote by $\prob(A)$ the probability of $A$ with respect to the random corruption model.
For a random variable $Y$, we denote by $\avg[Y]$ the expected value of $Y$ with respect to the random corruption model.
We denote by $A^1-A^5$ the events that conditions~\eqref{cond1}-\eqref{cond5} are strictly satisfied, respectively.
We denote by $A_{\rm recovery}$ the event that Problem~\eqref{primal} recovers the ground truth under the random corruption model; then by Lemma~\ref{unique} we have
$$
\prob(A_{\rm recovery}) \geq \prob\Big(A^1 \cap A^2 \cap A^3 \cap A^4 \cap A^5\Big).
$$
Since $A_{\rm recovery}$ contains the intersection of a
constant number of events $A^i$, to prove recovery with high probability, it suffices to prove that
each $A^i$, $i \in \{1,\ldots 5\}$ occurs with high probability. We establish this in the next four lemmata.

\begin{lemma}\label{MasterOfProbability1}
Suppose that $X^{\rm in}(i,j)$ for all $1 \leq i < j \leq n$ are generated according to the random corruption model \eqref{unifmodel}.
If
\begin{equation}\label{a1}\tag{A1}
p_{\rm true} \geq \frac{\beta}{\alpha+\beta},
\end{equation}
and $d \in o(\sqrt{n/\log(n)})$, then event $A^1$ occurs with high probability.
\end{lemma}

\begin{proof}
For each $1 \leq i < j \leq n$, for each $t,q \in [d]$, and for each $k \in [n] \setminus\{i,j\}$ define:
\begin{equation*}\label{y1}
Y^1_{t,q,k}(i,j):=\frac{n-2}{n}\begin{cases}
      \frac{a_{tq}(k,j)-1}{\alpha}+\frac{1-a_{tt}(k,j)}{\beta d} &  \text{if $a_{tt}(k,i)=-1$}\\
     0 & \text{otherwise,}
       \end{cases}
\end{equation*}
and
\begin{equation*}\label{y2}
Y^2_{t,q,k}(i,j):=\frac{n-2}{n}\begin{cases}
      \frac{a_{tq}(i,k)-1}{\alpha}+\frac{1-a_{qq}(i,k)}{\beta d} &  \text{if $a_{qq}(k,j)= -1$}\\
     0 & \text{otherwise.}
       \end{cases}
\end{equation*}
It then follows that event $A^1$ occurs, \ie condition~\eqref{cond1} is strictly satisfied, if for all $t,q \in [d]$ and
all $1 \leq i < j \leq n$ with $a_{tt}(i,j) = a_{qq}(i,j) = a_{tq}(i,j) = 1$, we have
\begin{align}
&\frac{1}{n-2}\sum_{k \in [n] \setminus\{i,j\}}{Y^1_{t,q,k}(i,j)} > 0, \label{set1}\\
&\frac{1}{n-2}\sum_{k \in [n] \setminus\{i,j\}}{Y^2_{t,q,k}(i,j)} > 0. \label{set2}
\end{align}
We start by showing that under assumption~\eqref{a1}, the following inequalities hold for every $t,q \in [d]$ and for every $1 \leq i < j \leq n$:
\begin{equation}\label{2222}
\epsilon_1(d):=\avg\Big[\frac{1}{n-2}\sum_{k \in [n] \setminus\{i,j\}}Y^1_{t,q,k}(i,j)\Big]>0, \quad \epsilon_2(d):=\avg\Big[\frac{1}{n-2}\sum_{k \in [n] \setminus\{i,j\}}Y^2_{t,q,k}(i,j)\Big]>0. %\epsilon_1=\epsilon_2.
\end{equation}
By direct computation we get that for every $t,q \in [d]$ and for every $1 \leq i < j \leq n$:
\begin{align*}
\epsilon_1(d) =& \avg\Big[\frac{1}{n-2}\sum_{k \in [n] \setminus\{i,j\}}Y^1_{t,q,k}(i,j)\Big]=\avg\Big[\frac{1}{n-2}\sum_{k \in [n] \setminus\{i,j\}}Y^2_{t,q,k}(i,j)\Big]= \epsilon_2(d)\\
= &\frac 2d \Big(p_{\rm true}+\frac{1-p_{\rm true}}{d} \Big) \Big(\frac{p_{\rm true}}{\beta} - (1-p_{\rm true}) (\frac{1}{\alpha}-\frac{1}{\beta d})\Big)= g_1 g_2,
\end{align*}
where $g_1 = \frac 2d (p_{\rm true}+\frac{1-p_{\rm true}}{d} )$ and $g_2 = \frac{p_{\rm true}}{\beta} - (1-p_{\rm true}) (\frac{1}{\alpha}-\frac{1}{\beta d})$.
Clearly, $g_1 > 0$ for all $d$ and $p_{\rm true}$. Moreover, $g_2$ is strictly decreasing in $d$ since its
derivative with respect
to $d$ is $-\frac{(1-p_{\rm true})}{\beta d^2}$. Therefore, to show that $g_2> 0$ for every $d$, it is enough to show that
$$\frac{p_{\rm true}}{\beta} -\frac{1-p_{\rm true}}{\alpha}=\lim_{d\to \infty}\frac{p_{\rm true}}{\beta} - (1-p_{\rm true}) (\frac{1}{\alpha}-\frac{1}{\beta d}) \geq 0,$$
which follows from assumption~\eqref{a1}.  Hence, $\epsilon_1(d)$ and $\epsilon_2(d)$
can be written as a product of two positive functions for every $d$ and this completes the proof of inequalities~\eqref{2222}.

We now show that event $A^1$ occurs with high probability. Let $A^1 = A^1_1 \cap A^1_2$, where $A^1_1$ (resp. $A^1_2$)
denotes the event that inequalities~\eqref{set1} (resp. inequalities~\eqref{set2}) are satisfied. To show $A^1$ occurs with high probability, it suffices to show that
$A^1_1$ and $A^1_2$ occur with high probability. In the following, we prove that $A^1_1$ occurs with high probability. The proof for $A^1_2$ follows from a similar line of arguments.
\begin{equation*}
\begin{split}
& \prob(A^1_1)\\
\geq & \prob \Bigg( \bigcap_{t,q,i,j}\Big\{\frac{1}{n-2}\sum_{k \in [n] \setminus \{i,j\}}Y^1_{t,q,k}(i,j) > 0\Big\}\Bigg ) \\
= &\prob \Bigg (\bigcap_{t,q,i,j}\Big\{\frac{1}{n-2}\sum_{k \in [n] \setminus \{i,j\}}Y^1_{t,q,k}(i,j)-\avg\Big[\frac{1}{n-2}\sum_{k \in [n] \setminus \{i,j\}}Y^1_{t,q,k}(i,j)\Big] >-\epsilon_1(d)\Big\}\Bigg) \\
\geq &\prob \Bigg (\bigcap_{t,q,i,j}\Big\{\Big|\frac{1}{n-2}\sum_{k \in [n] \setminus \{i,j\}}Y^1_{t,q,k}(i,j)-\avg\Big[\frac{1}{n-2}\sum_{k \in [n] \setminus \{i,j\}}Y^1_{t,q,k}(i,j)\Big]\Big| < \epsilon_1(d)\Big\}\Bigg ) \\
=& 1- \prob \Bigg(\bigcup_{t,q,i,j}\Big\{\Big|\frac{1}{n-2}\sum_{k \in [n] \setminus \{i,j\}}Y^1_{t,q,k}(i,j)-\avg\Big[\frac{1}{n-2}\sum_{k \in [n] \setminus \{i,j\}}Y^1_{t,q,k}(i,j)\Big]\Big| \geq \epsilon_1(d)\Big\}\Bigg) \\
\geq & 1 - \sum_{t,q,i,j}{\prob\Bigg(\Big|\frac{1}{n-2}\sum_{k \in [n] \setminus \{i,j\}}Y^1_{t,q,k}(i,j)-\avg\Big[\frac{1}{n-2}\sum_{k \in [n] \setminus \{i,j\}}Y^1_{t,q,k}(i,j)\Big]\Big| \geq \epsilon_1(d) \Bigg)}\\
\geq & 1- 2d^2(n-2)^2\exp\Big(-\frac{\epsilon_1^2(d) (n-2)}{8}\Big),
\end{split}
\end{equation*}
where the second and  fourth lines follow by set inclusion and the sixth line follows from the union bound. The last line follows from the application of Hoeffding's inequality~\cite{VerBookHDP} by noting that for each $(t,q,i,j)$ the random variables $Y^1_{t,q,k}(i,j)$ are independent  for all $k \in [n] \setminus \{i,j\}$ and can be bounded as $|Y^1_{t,q,k}(i,j)|\leq 2$. Since $\epsilon_1(d) \in \Theta(\frac{1}{d})$
and by assumption $d \in o(\sqrt{n/\log(n)})$, we get
$$2d^2(n-2)^2\exp\Big(-\frac{\epsilon_1^2(d) (n-2)}{8}\Big) \rightarrow 0$$
as $n \rightarrow \infty$.
\end{proof}

\begin{lemma}\label{MasterOfProbability2}
Suppose that $X^{\rm in}(i,j)$ for all $1 \leq i < j \leq n$ are generated according to the random corruption model \eqref{unifmodel}.
If
\begin{equation}\label{a2}\tag{A2}
p_{\rm true} \geq \frac{\alpha}{2},
\end{equation}
then  event $A^2$ occurs with high probability.
\end{lemma}

\begin{proof}
For each $1 \leq i < j \leq n$, for each $t,q \in [d]$ and for each $k \in [n] \setminus\{i,j\}$ define:
\begin{equation*}\label{y3}
Y^3_{t,q,k}(i,j):=\frac{n-2}{n}\begin{cases}
     \frac{a_{tq}(k,j)+1}{\alpha}+\frac{1-a_{tt}(k,j)}{\beta d}-1 &  \text{if $a_{tt}(k,i)=-1$}\\
     -1 & \text{otherwise,}
       \end{cases}
\end{equation*}
and
\begin{equation*}\label{y4}
Y^4_{t,q,k}(i,j):=\frac{n-2}{n}\begin{cases}
     \frac{a_{tq}(i,k)+1}{\alpha}+\frac{1-a_{qq}(i,k)}{\beta d}-1 &  \text{if $a_{qq}(k,j)= -1$}\\
     -1 & \text{otherwise.}
       \end{cases}
\end{equation*}
It can be checked that event $A^2$ occurs, \ie condition~\eqref{cond2} is strictly satisfied, if for all $t,q \in [d]$ and
all $1 \leq i < j \leq n$ with $a_{tt}(i,j) = a_{qq}(i,j) = 1, a_{tq}(i,j) = -1$, we have
\begin{align}
&\frac{1}{n-2}\sum_{k \in [n] \setminus\{i,j\}}{Y^3_{t,q,k}(i,j)} > 0, \label{set3}\\
&\frac{1}{n-2}\sum_{k \in [n] \setminus\{i,j\}}{Y^4_{t,q,k}(i,j)} > 0. \label{set4}
\end{align}
We start by proving that under assumption~\eqref{a2},
the following inequalities hold for every $t,q \in [d]$ and for every $1 \leq i < j \leq n$:
\begin{equation}\label{2222a}
\epsilon_3(d):=\avg\Big[\frac{1}{n-2}\sum_{k \in [n] \setminus\{i,j\}} Y^3_{t,q,k}(i,j)\Big]>0, \quad \epsilon_4(d):=\avg\Big[\frac{1}{n-2}\sum_{k \in [n] \setminus\{i,j\}} Y^4_{t,q,k}(i,j)\Big]>0.
%, \quad \epsilon_3=\epsilon_4.
\end{equation}
By direct computation we get that for every $t,q \in [d]$ and for every $1 \leq i < j \leq n$:
\begin{align*}
\epsilon_3(d) =&  \avg\Big[\frac{1}{n-2}\sum_{k \in [n] \setminus\{i,j\}} Y^3_{t,q,k}(i,j)\Big]=\avg\Big[\frac{1}{n-2}\sum_{k \in [n] \setminus\{i,j\}} Y^4_{t,q,k}(i,j)\Big]= \epsilon_4(d) \\
=& 2\Big(p_{\rm true}+\frac{1-p_{\rm true}}{d} \Big)\Big(\frac{p_{\rm true}}{\beta d}-\frac{1-p_{\rm true}}{d}(\frac{1}{\alpha}-\frac{1}{\beta d})+\frac{1}{\alpha}\Big) - 1.
\end{align*}
It can be checked that the derivative of the above expression with respect to $d$ is given by
$-\frac{2(1-p_{\rm true})^2}{\alpha d^2}(1-\frac{2}{d})- \frac{2p_{\rm true}^2}{\beta d^2}-8\frac{p_{\rm true}(1-p_{\rm true})}{\beta d^3}-6\frac{(1-p_{\rm true})^2}{\beta d^4}$,
which is negative for all $0\leq p_{\rm true} \leq 1$ and $d \geq 2$.
Hence it is enough to show that
$$\frac{2p_{\rm true}}{\alpha} -1=\lim_{d\to \infty}2\Big(p_{\rm true}+\frac{1-p_{\rm true}}{d} \Big)\Big(\frac{p_{\rm true}}{\beta d}-\frac{1-p_{\rm true}}{d}(\frac{1}{\alpha}-\frac{1}{\beta d})+\frac{1}{\alpha}\Big) - 1 \geq 0,$$
which follows from assumption~\eqref{a2}.

We now show that event $A^2$ occurs with high probability. Let $A^2 = A^2_1 \cap A^2_2$, where $A^2_1$ (resp. $A^2_2$)
denotes the event for which inequalities~\eqref{set3} (resp. inequalities~\eqref{set4}) are satisfied. To show $A^2$ occurs with high probability, it suffices to show that
$A^2_1$ and $A^2_2$ occur with high probability. Notice that for each $t, q \in [d]$ and each $1 \leq i < j \leq n$, the random variables $Y^3_{t,q,k}, Y^4_{t,q,k}$ for all $k \in [n] \setminus \{i,j\}$ are independent and can be bounded as
$ -1\leq Y^3_{t,q,k}(i,j)\leq 3$ and $-1\leq Y^4_{t,q,k}(i,j)\leq 3$.
Hence, the proof that $A^2_1$ (or $A^2_2$)  occurs with high probability is very similar to that of $A^1_1$ stated in the proof of Lemma~\ref{MasterOfProbability1}.
That is, it can be checked that
$$\prob(A^2_1) \geq  1- 2d^2(n-2)^2\exp\Big(-\frac{\epsilon_3^2(d) (n-2)}{8}\Big).$$
Since $\epsilon_3(d) \in \Theta(1)$, it follows that $\prob(A^2_1)\rightarrow 1$,
as $n \rightarrow \infty$.
\end{proof}

\begin{lemma}\label{MasterOfProbability3}
Suppose that $X^{\rm in}(i,j)$ for all $1 \leq i < j \leq n$ are generated according to the random corruption model \eqref{unifmodel}.
If
\begin{align}
&\alpha \leq \beta \leq 2,\label{a3}\tag{A3}\\
&\frac{p_{\rm true}^2}{\beta} + (\frac{1}{\alpha}-\frac{1}{\beta})p_{\rm true}+ (\frac{1}{2}-\frac{1}{\alpha}) > 0,\label{a4}\tag{A4}\\
&p_{\rm true} \geq \frac{1}{4},\label{a5}\tag{A5}\\
&n^2>12d,\label{a6}\tag{A6}
\end{align}
and $d\in o(({n/\log(n)})^{\frac 14})$, then event $A^3$ occurs with high probability.
\end{lemma}
\begin{proof}
We start by defining a number of random variables associated with inequalities of condition~\eqref{cond3}.
For each $1 \leq i < j \leq n$, for each $t \in [d]$ and for each $k \in [n] \setminus \{i,j\}$ define:
\begin{equation}\label{zz5}
Z_{t,k}(i,j):=1-\frac{(d-2)}{\alpha}+\frac{n-2}{n}\sum_{l \in [d] \setminus \{t\}}\min\left\{\frac{a_{lt}(k,i)}{\alpha}+\frac{a_{tt}(k,j)}{\beta d},
\frac{a_{lt}(k,j)}{\alpha}+\frac{a_{tt}(k,i)}{\beta d}\right\}.
\end{equation}
For each $1 \leq i < j \leq n$, for each $t \in [d]$ and for each $k \in [n]$ define:
\begin{equation}\label{qq5}
Q_{t,k}(i,j) := \big(\frac{1}{2\alpha}-\frac{d-1}{2\beta d}\big)\big({a_{tt}(k,i)}+{a_{tt}(k,j)}\big),
\end{equation}
where we define ${a_{tt}(i,i)}={a_{tt}(j,j)}=0$.
For each $1 \leq i < j \leq n$, for each $t \in [d]$ define:
\begin{equation}\label{ww5}
W_{t}(i,j):=\frac{1}{2n} \sum_{l \in [d] \setminus \{t\}}{\Big(\frac{a_{tl}(i,j)+a_{tl}(j,i)}{\alpha}+2\frac{a_{tt}(i,j)}{\beta d}\Big)}.
\end{equation}
Finally,  for each $1 \leq i < j \leq n$, for each $t,q \in [d]$ and for each $k \in [n] \setminus \{i,j\}$ define:
\begin{equation*}\label{y5}
Y^5_{t,q,k}(i,j):=\begin{cases}
     Z_{t,k}(i,j)+ \frac{1}{2}\big(\frac{a_{tq}(k,j)-1}{\alpha}+\frac{-1-a_{tt}(k,j)}{\beta d}+ \frac{a_{tq}(i,k)-1}{\alpha}+\frac{-1-a_{qq}(i,k)}{\beta d}\big)& \\
\qquad \qquad \qquad  \qquad \qquad \qquad \qquad \, \text{if $a_{tt}(k,i)=-1$, $a_{qq}(k,j)= -1$}&\\
Z_{t,k}(i,j)+ \frac{1}{2}\big(\frac{a_{tq}(i,k)-1}{\alpha}+\frac{-1-a_{qq}(i,k)}{\beta d}\big)& \\
\qquad \qquad \qquad  \qquad \qquad \qquad \qquad \, \text{if $a_{tt}(k,i)=1$, $a_{qq}(k,j)= -1$}&\\
 Z_{t,k}(i,j)+\frac{1}{2}\big(\frac{a_{tq}(k,j)-1}{\alpha}+\frac{-1-a_{tt}(k,j)}{\beta d}\big)& \\
\qquad \qquad \qquad  \qquad \qquad \qquad \qquad \, \text{if $a_{tt}(k,i)=-1$, $a_{qq}(k,j)= 1$}&\\
 Z_{t,k}(i,j)
 \qquad \qquad \qquad  \qquad \,\,\,\qquad\text{if $a_{tt}(k,i)=1$, $a_{qq}(k,j)= 1$}&
       \end{cases}
\end{equation*}
and
\begin{equation*}\label{y6}
Y^6_{t,q,k}(i,j):=\begin{cases}
      Z_{q,k}(i,j)+\frac{1}{2}\big(\frac{a_{tq}(k,j)-1}{\alpha}+\frac{-1-a_{tt}(k,j)}{\beta d}+ \frac{a_{tq}(i,k)-1}{\alpha}+\frac{-1-a_{qq}(i,k)}{\beta d}\big)& \\
\qquad \qquad \qquad  \qquad \qquad \qquad \qquad \, \text{if $a_{tt}(k,i)=-1$, $a_{qq}(k,j)= -1$}&\\
 Z_{q,k}(i,j)+\frac{1}{2}\big(\frac{a_{tq}(i,k)-1}{\alpha}+\frac{-1-a_{qq}(i,k)}{\beta d}\big)& \\
\qquad \qquad \qquad  \qquad \qquad \qquad \qquad \, \text{if $a_{tt}(k,i)=1$, $a_{qq}(k,j)= -1$}&\\
  Z_{q,k}(i,j)+\frac{1}{2}\big(\frac{a_{tq}(k,j)-1}{\alpha}+\frac{-1-a_{tt}(k,j)}{\beta d}\big)& \\
\qquad \qquad \qquad  \qquad \qquad \qquad \qquad \, \text{if $a_{tt}(k,i)=-1$, $a_{qq}(k,j)= 1$}&\\
  Z_{q,k}(i,j) \qquad \; \qquad \qquad \qquad \qquad\,  \text{if $a_{tt}(k,i)=1$, $a_{qq}(k,j)= 1$}&
       \end{cases}
 \end{equation*}
It then follows that event $A^3$ occurs, \ie condition~\eqref{cond3} is strictly satisfied, if for all $t,q \in [d]$ and
all $1 \leq i < j \leq n$ with $a_{tt}(i,j) = a_{qq}(i,j) = -1, a_{tq}(i,j) = 1$, we have
\begin{align}
&W_{t}(i,j)+\frac 1n\sum_{k \in [n]}Q_{t,k}(i,j)+\frac{1}{n-2}\sum_{k \in [n] \setminus \{i,j\}}{Y^5_{t,q,k}(i,j)} > 0, \label{set5}\\
&W_{q}(i,j)+\frac 1n \sum_{k \in [n]} Q_{q,k}(i,j)+\frac{1}{n-2}\sum_{k \in [n] \setminus \{i,j\}}{Y^6_{t,q,k}(i,j)} > 0 \label{set6}.
\end{align}
We first prove the following inequalities hold for all $t,q \in [d]$ and for all $1 \leq i < j \leq n$:
\begin{equation}
\begin{split}
&
\epsilon_5(d):= \avg\Big[\frac 1n\sum_{k \in [n]}Q_{t,k}(i,j)+\frac{1}{n-2}\sum_{k \in [n] \setminus \{i,j\}}{Y^5_{t,q,k}(i,j)}\Big]>0, \\
&\epsilon_6(d):=\avg\Big[\frac 1n\sum_{k \in [n]}Q_{q,k}(i,j)+\frac{1}{n-2}\sum_{k \in [n] \setminus \{i,j\}}{Y^6_{t,q,k}(i,j)}\Big]>0.\label{2222b}
%, \quad \epsilon_5=\epsilon_6.
\end{split}
\end{equation}
Notice that by symmetry, $\epsilon_5(d)=\epsilon_6 (d)$; hence, it suffices to prove the validity of the first inequality in \eqref{2222b}.
By direct calculation we get:
\begin{equation}\label{inter}
\begin{split}
\epsilon_5(d) = &(d-1)\Big(p_{\rm true}^2(\frac{1}{\alpha}-\frac{1}{\beta d})+2p_{\rm true}(1-p_{\rm true})(-\frac{1}{d} (\frac{1}{\alpha}+\frac{1}{\beta d})
    +(1-\frac{1}{d})(\frac{1}{\alpha}-\frac{1}{\beta d}))\\
    +&
    (1-p_{\rm true})^2 (\frac{1}{\alpha}+\frac{1}{\beta d})(1-\frac{2}{d}-\frac{2}{d-1}(\frac{1}{d}-1)^2)\Big)\\
    +&2 \Big(p_{\rm true}+\frac{1-p_{\rm true}}{d}\Big) \Big(\frac{p_{\rm true}}{\beta d}+\frac{1-p_{\rm true}}{d}(\frac{1}{\beta d}-\frac{1}{\alpha})-\frac{1}{\beta d}\Big)\\
-&\Big(\frac{d-1}{\beta d}-\frac{1}{\alpha}\Big)\Big(1-2(p_{\rm true}+\frac{1-p_{\rm true}}{d})\Big)-\frac{(d-2)}{\alpha}+1\\
%=& 1+ \frac{d-1}{\beta d^2} \Big(2(d-1)p^2_{\rm true}-2(d-2)p_{\rm true}-2\Big)-\frac{2(1-p_{\rm true})}{\alpha}\Big(1-2 (1-p_{\rm true}) \frac{d-1}{d^2}\Big)\\
=& 2\Big[\frac{p_{\rm true}^2}{\beta} + (\frac{1}{\alpha}-\frac{1}{\beta})p_{\rm true}+ (\frac{1}{2}-\frac{1}{\alpha})\Big] +(\frac 4\alpha-\frac 2\beta)\frac{(p_{\rm true}-1)^2(d-1)}{d^2} +\frac{2p_{true}}{\beta d}\\
\geq &\frac{2p_{true}}{\beta d}\geq \frac{1}{4d},
\end{split}
\end{equation}
where the inequalities in the last line follow from assumptions~\eqref{a3},~\eqref{a4}, and ~\eqref{a5}.

Next we show that event $A^3$ occurs with high probability. Denote by $A^3_1$ (resp. $A^3_2$) the event that inequalities~\eqref{set5} (resp. inequalities~\eqref{set6}) are satisfied. Then $A^3 = A^3_1 \cap A^3_2$. To show that $A^3$ occurs with high probability, it suffices to show that $A^3_1$ and $A^3_2$ occur with probability. In the following we show that $A^3_1$ occurs with high probability. The proof for $A^3_2$ follows from a similar line of arguments.
\begin{align}
&\prob(A^3_1)\nonumber\\
\geq &\prob \Bigg( \bigcap_{t,q,i,j}\Big\{W_{t}(i,j)+\frac 1n\sum_{k \in [n]}Q_{t,k}(i,j)+\frac{1}{n-2}\sum_{k \in [n] \setminus \{i,j\}}{Y^5_{t,q,k}(i,j)} > 0\Big\}\Bigg ) \nonumber\\
= &\prob \Bigg (\bigcap_{t,q,i,j}\Big\{W_{t}(i,j)  +\frac 1n\sum_{k \in [n]}Q_{t,k}(i,j) - \avg\Big[\frac 1n\sum_{k \in [n]}Q_{t,k}(i,j)\Big]+\frac{1}{n-2}\sum_{k \in [n] \setminus \{i,j\}}Y^5_{t,q,k}(i,j)\nonumber\\
-&\avg\Big[\frac{1}{n-2}\sum_{k \in [n] \setminus \{i,j\}}Y^5_{t,q,k}(i,j)\Big]>-\epsilon_5(d)\Big\}\Bigg) \nonumber\\
\geq &\prob \Bigg (\bigcap_{t,q,i,j}\Big\{\Big|W_{t}(i,j)+\frac 1n\sum_{k \in [n]}Q_{t,k}(i,j) - \avg\Big[\frac 1n\sum_{k \in [n]}Q_{t,k}(i,j)\Big]
+ \frac{1}{n-2}\sum_{k \in [n] \setminus \{i,j\}}Y^5_{t,q,k}(i,j)\nonumber\\
-&\avg\Big[\frac{1}{n-2}\sum_{k \in [n] \setminus \{i,j\}}Y^5_{t,q,k}(i,j)\Big] \Big| < \epsilon_5(d)\Big\}\Bigg ) \nonumber\\
\geq &\prob \Bigg (\bigcap_{t,i,j}\Big\{\Big|W_{t}(i,j)\Big| < \frac{\epsilon_5(d)}{3}\Big\} \cap \bigcap_{t,i,j}\Big\{\Big| \frac 1n\sum_{k \in [n]}Q_{t,k}(i,j) - \avg\Big[\frac 1n\sum_{k \in [n]}Q_{t,k}(i,j)\Big]< \frac{\epsilon_5(d)}{3}\Big\}\nonumber\\
\cap &\bigcap_{t,q,i,j}\Big\{\Big| \frac{1}{n-2}\sum_{k \in [n] \setminus \{i,j\}}Y^5_{t,q,k}(i,j)-\avg\Big[\frac{1}{n-2}\sum_{k \in [n] \setminus \{i,j\}}Y^5_{t,q,k}(i,j)\Big]\Big| < \frac{\epsilon_5(d)}{3}\Big\}\Big| \Bigg),  \label{p1}
\end{align}
where all inequalities follow by set inclusion.
It can be checked that  $|W_{t}(i,j)|\leq \frac{d}{n}$; hence by \eqref{a6} we get $|W_{t}(i,j)|\leq \frac{1}{12d}$. Moreover by~\eqref{inter} we have
$\frac{\epsilon_5(d)}{3} \geq \frac{1}{12 d}$.  So we deduce that
$|W_{t}(i,j)|\leq  \frac{\epsilon_5(d)}{3}$, which in turn implies
\begin{equation}\label{p2}
\Big\{\Big|W_{t}(i,j)\Big| \geq \frac{\epsilon_5}{3}\Big\}=\emptyset.
\end{equation}
%and using $p_{true}\geq 1/2$, and $\beta <2$, hence $\frac{1}{2d}\leq \frac{2p_{true}}{\beta d}\leq\epsilon_5(d)$, so we chose $n\geq n_0$ where
%$$
%\frac d{n_0}<\frac{1}{6d} \leq  \epsilon_5(d)/3, \quad \mbox{i.e.} n_0> 6d^2
%$$
Combining~\eqref{p1} and~\eqref{p2} we obtain:
\begin{equation*}
\begin{split}
&\prob(A^3_1)\\
\geq & 1- \prob \Bigg(\bigcup_{t,i,j}\Big\{\Big|\frac 1n\sum_{k \in [n]}Q_{t,k}(i,j) - \avg\Big[\frac 1n\sum_{k \in [n]}Q_{t,k}(i,j)\Big]\Big| \geq \frac{\epsilon_5(d)}{3}\Big\}\\ & \cup \bigcup_{t,q,i,j}\Big\{\Big|\frac{1}{n-2}\sum_{k \in [n] \setminus \{i,j\}}Y^5_{t,q,k}(i,j)-\avg\Big[\frac{1}{n-2}\sum_{k \in [n] \setminus \{i,j\}}Y^5_{t,q,k}(i,j)\Big]\Big| \geq\frac{\epsilon_5(d)}{3}\Big\}\Bigg) \\
\geq & 1- \sum_{t,i,j}\prob\Bigg(\Big|\frac 1n\sum_{k \in [n]}Q_{t,k}(i,j) - \avg\Big[\frac 1n\sum_{k \in [n]}Q_{t,k}(i,j)\Big]\Big| \geq \frac{\epsilon_5(d)}{3} \Bigg)\\
&- \sum_{t,q,i,j}\prob\Bigg(\Big|\frac{1}{n-2}\sum_{k \in [n] \setminus \{i,j\}}Y^5_{t,q,k}(i,j)-\avg\Big[\frac{1}{n-2}\sum_{k \in [n] \setminus \{i,j\}}Y^5_{t,q,k}(i,j)\Big]\Big| \geq\frac{\epsilon_5(d)}{3}  \Bigg)\\
\geq & 1 - 2dn^2\exp\Big(-\frac{\epsilon_5^2(d) n}{9}\Big)-2d^2(n-2)^2 \exp\Big(-\frac{\epsilon_5^2(d) (n-2)}{162d^2}\Big),
\end{split}
\end{equation*}
where the second inequality follows from taking the union bound. The last inequality follows from the application of Hoeffding's inequality~\cite{VerBookHDP} by noting that for each
$(t,i,j)$ the random variables $Q_{t,k}(i,j)$ are independent for all $k \in [n]$ and can be bounded as $|Q_{t,k}(i,j)|\leq 1$. Moreover,
for each $(t,q,i,j)$ the random variables $Y^5_{t,q,k}(i,j)$ are independent  for all $k \in [n] \setminus \{i,j\}$ and can be bounded as $|Y^5_{t,q,k}(i,j)|\leq 6d$. By \eqref{inter}, the result follows by letting $n \rightarrow \infty$, since $d\in o(({n/\log(n)})^{\frac 14})$.
\end{proof}

\begin{lemma}\label{MasterOfProbability4}
Suppose that $X^{\rm in}(i,j)$ for all $1 \leq i < j \leq n$ are generated according to the random corruption model \eqref{unifmodel}.
If assumptions~\eqref{a1},~\eqref{a3},~\eqref{a4},~\eqref{a5},~and \eqref{a6} are satisfied, and $d \in o(({n/\log(n)})^{\frac 14})$, then events $A^4$ and $A^5$ occur with high probability.
\end{lemma}

\begin{proof}
We start by defining some random variables associated with inequalities of conditions~\eqref{cond4} and~\eqref{cond5}.
For each $1 \leq i < j \leq n$, for each $t,q \in [d]$ and for each $k \in [n] \setminus \{i,j\}$ define:
\begin{equation*}\label{y7}
Y^7_{t,q,k}(i,j):=\begin{cases}
      Z_{q,k}(i,j)+\frac{a_{tq}(k,j)-1}{\alpha}+\frac{1-a_{tt}(k,j)}{\beta d}+ \frac{a_{tq}(i,k)-1}{\alpha}+\frac{-1-a_{qq}(i,k)}{\beta d}& \\
\qquad \qquad \qquad  \qquad \qquad \qquad \qquad \, \text{if $a_{tt}(k,i)=-1$, $a_{qq}(k,j)= -1$}&\\
 Z_{q,k}(i,j)+\frac{a_{tq}(i,k)-1}{\alpha}+\frac{-1-a_{qq}(i,k)}{\beta d}& \\
\qquad \qquad \qquad  \qquad \qquad \qquad \qquad \, \text{if $a_{tt}(k,i)=1$, $a_{qq}(k,j)= -1$}&\\
 Z_{q,k}(i,j)+\frac{a_{tq}(k,j)-1}{\alpha}+\frac{1-a_{tt}(k,j)}{\beta d}& \\
\qquad \qquad \qquad  \qquad \qquad \qquad \qquad \, \text{if $a_{tt}(k,i)=-1$, $a_{qq}(k,j)= 1$}&\\
 Z_{q,k}(i,j) \qquad \qquad \qquad \qquad \; \qquad \text{if $a_{tt}(k,i)=1$, $a_{qq}(k,j)= 1$},&
       \end{cases}
\end{equation*}
where $Z_{q,k}(i,j)$ is defined by~\eqref{zz5}.
It can be checked that condition~\eqref{cond4} is strictly satisfied, if for all $t,q \in [d]$ and
all $1 \leq i < j \leq n$ with $a_{tt}(i,j) =1, a_{qq}(i,j) = -1, a_{tq}(i,j) = 1$, we have
$$W_{q}(i,j)+\frac 1n\sum_{k \in [n]}Q_{q,k}(i,j)+\frac{1}{n-2}\sum_{k \in [n] \setminus \{i,j\}}{Y^7_{t,q,k}(i,j)} > 0,$$
where $Q_{q,k}(i,j)$ $W_{q}(i,j)$ and given by~\eqref{qq5} and~\eqref{ww5}, respectively.
For each $1 \leq i < j \leq n$, for each $t,q \in [d]$ and for each $k \in [n] \setminus \{i,j\}$ define:
\begin{equation*}\label{y7}
Y^8_{t,q,k}(i,j):=\begin{cases}
     Z_{t,k}(i,j)+\frac{a_{tq}(k,j)-1}{\alpha}+\frac{-1-a_{tt}(k,j)}{\beta d}+ \frac{a_{tq}(i,k)-1}{\alpha}+\frac{1-a_{qq}(i,k)}{\beta d}& \\
\qquad \qquad \qquad \qquad \,\,\, \text{if $a_{tt}(k,i)=-1$, $a_{qq}(k,j)= -1$}&\\
Z_{t,k}(i,j)+\frac{a_{tq}(i,k)-1}{\alpha}+\frac{1-a_{qq}(i,k)}{\beta d}& \\
\qquad \qquad \qquad \qquad \,\,\, \text{if $a_{tt}(k,i)=1$, $a_{qq}(k,j)= -1$}&\\
 Z_{t,k}(i,j)+\frac{a_{tq}(k,j)-1}{\alpha}+\frac{-1-a_{tt}(k,j)}{\beta d}& \\
\qquad \qquad \qquad \qquad \,\,\, \text{if $a_{tt}(k,i)=-1$, $a_{qq}(k,j)= 1$} &\\
 Z_{t,k}(i,j) \qquad \qquad \quad \text{if $a_{tt}(k,i)=1$, $a_{qq}(k,j)= 1$}&
       \end{cases}
\end{equation*}
It then follows that condition~\eqref{cond5} is strictly satisfied, if for all $t,q \in [d]$ and
all $1 \leq i < j \leq n$ with $a_{tt}(i,j) =-1, a_{qq}(i,j) = a_{tq}(i,j) = 1$, we have
$$W_{t}(i,j)+\frac 1n\sum_{k \in [n]}Q_{t,k}(i,j)+\frac{1}{n-2}\sum_{k \in [n] \setminus \{i,j\}}{Y^8_{t,q,k}(i,j)} > 0.$$
We first prove that the following inequalities hold for all $t,q \in [d]$ and for all $1 \leq i < j \leq n$:
\begin{equation}\label{2222c}
\begin{split}
&\epsilon_7(d):=\avg\Big[\frac 1n\sum_{k \in [n]}Q_{q,k}(i,j)+\frac{1}{n-2}\sum_{k \in [n] \setminus \{i,j\}}{Y^7_{t,q,k}(i,j)}\Big]>0, \\
&\epsilon_8(d):=\avg\Big[\frac 1n\sum_{k \in [n]}Q_{t,k}(i,j)+\frac{1}{n-2}\sum_{k \in [n] \setminus \{i,j\}}{Y^8_{t,q,k}(i,j)}\Big]>0.
\end{split}
\end{equation}
Notice that by symmetry, $\epsilon_7(d) = \epsilon_8(d)$; hence, it suffices to prove the
validity of the first inequality in~\eqref{2222c}. By direct computation we get
\begin{align*}
\epsilon_7(d) = &(d-1)\Big(p_{\rm true}^2(\frac{1}{\alpha}-\frac{1}{\beta d})+2p_{\rm true}(1-p_{\rm true})(-\frac{1}{d} (\frac{1}{\alpha}+\frac{1}{\beta d})
    +(1-\frac{1}{d})(\frac{1}{\alpha}-\frac{1}{\beta d}))\\
    +&(1-p_{\rm true})^2 (\frac{1}{\alpha}+\frac{1}{\beta d})(1-\frac{2}{d}-\frac{2}{d-1}(\frac{1}{d}-1)^2)\Big)\\
    +&4 \Big(p_{\rm true}+\frac{1-p_{\rm true}}{d}\Big)\Big(\frac{p_{\rm true}}{\beta d}+\frac{1-p_{\rm true}}{d}(\frac{1}{\beta d}-\frac{1}{\alpha})-\frac{1}{2\beta d}\Big)\\
-&\Big(\frac{d-1}{\beta d}-\frac{1}{\alpha}\Big)\Big(1-2(p_{\rm true}+\frac{1-p_{\rm true}}{d})\Big)-\frac{d-2}{\alpha}+1\\
= & \epsilon_5(d) + \frac{p_{\rm true}}{\beta d} + \frac{1-p_{\rm true}}{d} (\frac{1}{\beta d}-\frac{1}{\alpha})\\
\geq& \epsilon_5(d),
\end{align*}
where $\epsilon_5(d)$ is defined by \eqref{2222b}. The last inequality follows from the proof of Lemma~\ref{MasterOfProbability1} which in turn follows from assumption~\eqref{a1}.
Hence, $\epsilon_7(d) > 0$ is implied by $\epsilon_5(d) > 0$ which by proof of Lemma~\ref{MasterOfProbability3} is valid if assumptions~\eqref{a3}, ~\eqref{a4}, and~\eqref{a5} hold. It remains to show that events $A^4, A^5$ occur with high probability.
Note that for each $(t,q,i,j)$, random variables $Y^7_{t,q,k}(i,j)$ (resp. $Y^8_{t,q,k}(i,j)$) are independent for all $k \in [n] \setminus \{i,j\}$ and can be bounded as $|Y^7_{t,q,k}(i,j)| \leq 6d$
(resp. $|Y^8_{t,q,k}(i,j)| \leq 6d$). Hence, by assumption~\eqref{a6},
the associated proof for event $A^3_1$ of Lemma~\ref{MasterOfProbability3} can be repeated verbatim to show $A^4$ and $A^5$ occur with high probability.
\end{proof}

\begin{comment}
Moreover, it is simple to verify that the above random variables can be bounded as follows:
\begin{align}\label{bounds}
&|Y^1_{t,q,k}(i,j)|\leq 2, \quad |Y^2_{t,q,k}(i,j)|\leq 2, \quad -1\leq Y^3_{t,q,k}(i,j)\leq 3, \quad -1\leq Y^4_{t,q,k}(i,j)\leq 3,\nonumber\\
&|Y^5_{t,q,k}(i,j)|\leq 6d, \quad |Y^6_{t,q,k}(i,j)|\leq 6d, \quad |Y^7_{t,q,k}(i,j)|\leq 6d, \quad |Y^8_{t,q,k}(i,j)|\leq 6d, \quad |Q_{t,k}(i,j)|\leq 1.
\end{align}
\end{comment}
%Denote by $A_{\rm recovery}$ the event that Problem~\eqref{primal} recovers the ground truth under the random corruption model; then:
%$$
%\prob(A_{\rm recovery}) \geq \prob \Big (\bigcap_{m=1}^8\bigcap_{1 \leq i < j \leq n} \bigcap_{t,q \in [d]}\Big\{\frac{1}{n}\sum_{k \in [n]}Y^m_{t,q,k}(i,j) > 0\Big\}\Big).
%$$
%Let us refer to the event on the right hand side of the above inequality as $A_{\rm subR}$.
%It then follows that to prove Theorem~\ref{main}, it suffices to show that the value of the right-hand side of the above inequality tends to $1$ as $n \rightarrow \infty$.
%$\prob(A_{\rm subR}) \rightarrow 1$ as $n \rightarrow \infty$.
%In the next four lemmata, we characterize the conditions, in terms of $p_{\rm true}, \alpha, \beta$, under which inequalities
%$\frac{1}{n}\sum_{k \in [n]}Y^m_{t,q,k}(i,j) > 0$, $m \in \{1,\ldots,7\}$, $t,q \in [d]$, $1 \leq i < j \leq n$, hold in expectation.
%Subsequently, via the application of concentration inequalities, we prove Theorem~\ref{main} by showing that as $n \rightarrow \infty$, the value of the left-hand sides of these inequalities concentrates around their expected value.

We are now ready to prove the main result of this section.

\noindent\emph{Proof of Theorem~\ref{main}.}
Let $\alpha = 1.172$ and $\beta = 1.657$; then it is simple to verify that for $p_{\rm true} > 0.585$, assumptions~\eqref{a1}-\eqref{a5} of Lemmata~\ref{MasterOfProbability1}-\ref{MasterOfProbability4} are satisfied. Moreover, assumption $d\in o(({n/\log(n)})^{\frac 14})$ implies that \eqref{a6} is satisfied for $n$ large enough and this completes the proof. \qed

%\begin{remark}
In the proofs of Lemmata~\ref{MasterOfProbability1}-\ref{MasterOfProbability4}, to prove the positivity of expectations of certain random variables, we show that
these expectations are non-increasing functions of $d$ and hence we show their positivity for the limiting case $d \rightarrow \infty$. Indeed assumptions~\eqref{a1}-\eqref{a5} can be slightly relaxed if such asymptotic analysis is avoided. That is, for each $d \geq 2$, we consider the following optimization problem:
\begin{align}\label{exact}
 \min_{\alpha, \beta, p_{\rm true}} \; & p_{\rm true} \\
 \st \; &\frac{p_{\rm true}}{\beta} + (1-p_{\rm true}) (\frac{1}{\beta d}-\frac{1}{\alpha}) \geq 0, \nonumber\\
&  \Big(p_{\rm true}+\frac{1-p_{\rm true}}{d} \Big)\Big(\frac{p_{\rm true}}{\beta d}+(\frac{1}{\beta d}-\frac{1}{\alpha})\frac{1-p_{\rm true}}{d}+\frac{1}{\alpha}\Big) \geq \frac{1}{2},\nonumber\\
& 2\Big(\frac{p_{\rm true}^2}{\beta} + (\frac{1}{\alpha}-\frac{1}{\beta})p_{\rm true}+ (\frac{1}{2}-\frac{1}{\alpha})\Big) +(\frac 4\alpha-\frac 2\beta)\frac{(p_{\rm true}-1)^2(d-1)}{d^2} +\frac{2p_{{\rm true}}}{\beta d} \geq 0, \nonumber\\
& \frac{1}{4} \leq p_{\rm true} \leq 1,\nonumber\\
& 0 \leq \alpha \leq \beta \leq 2, \nonumber
\end{align}
where the first constraint is stated in the proof of Lemma~\ref{MasterOfProbability1}, the second constraint is stated in the proof of Lemma~\ref{MasterOfProbability2}, and
the third  constraint is stated in the proof of Lemma~\ref{MasterOfProbability3}. Then for each $d \in [2, +\infty)$, the optimal value $p_{\rm true}^*$ of the above problem serves as a recovery threshold of Problem~\eqref{primal}.
While we are unable to obtain an analytical solution for Problem~\eqref{exact}, we solve this problem numerically with high accuracy using the global solver {\tt BARON}~\cite{IdaNick18}
for $d \in [2,50]$. Results are depicted in Figure~\ref{figure1}; for example, this figure indicate that for $d=2$,  the LP recovers the ground truth with high probability if $p_{\rm true} > \frac{1}{3}$. While by performing a worst-case analysis, we have obtained a more conservative bound of $p_{\rm true} > 0.585$ for the recovery threshold of the LP, Figure~\ref{figure1} suggests that this bound is a good approximation for $d \gtrsim 40$.
%\end{remark}

\begin{figure}[htbp]
 \centering
 \epsfig{figure=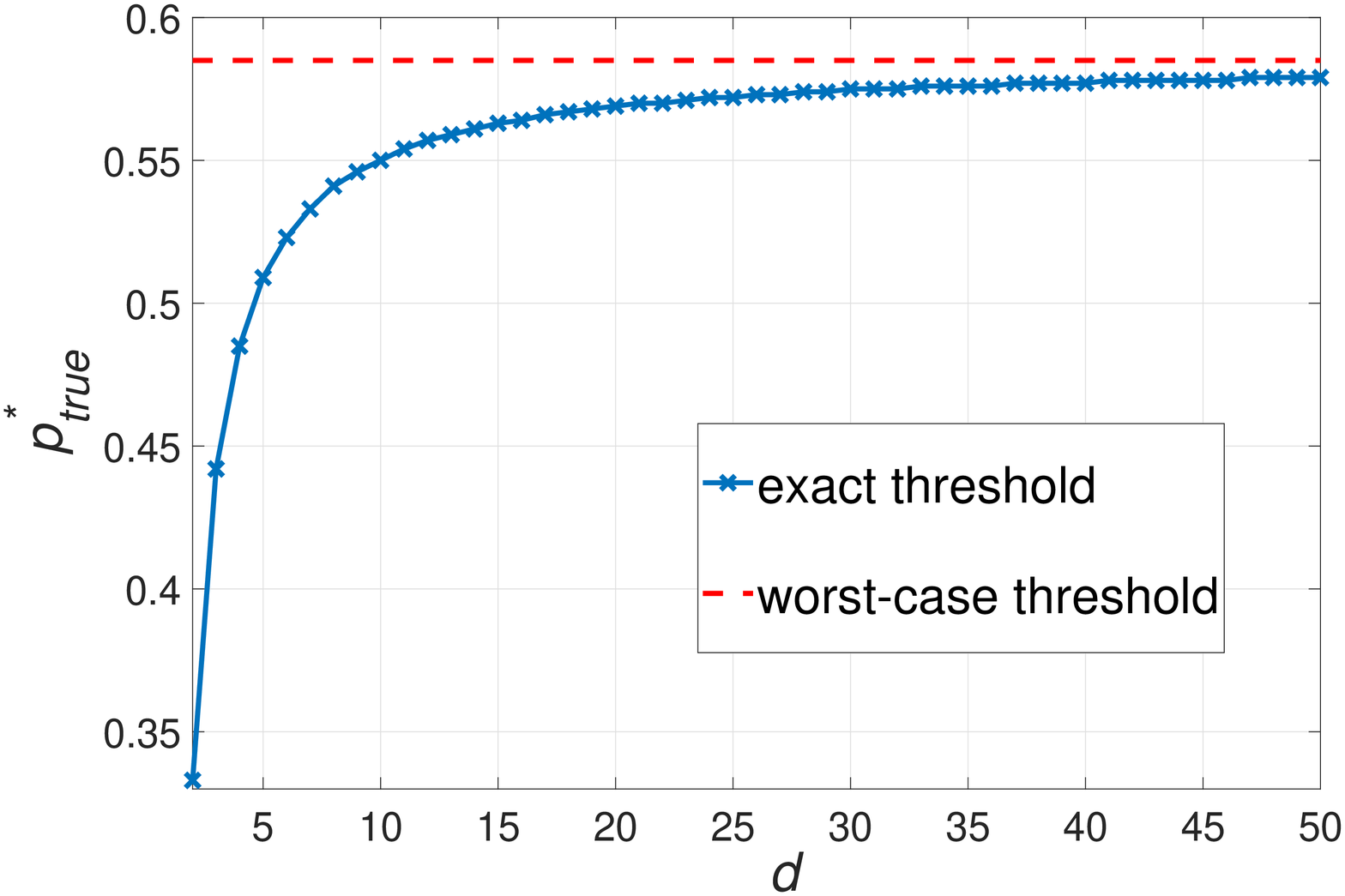, scale=0.23, trim=0mm 0mm 80mm 0mm, clip}
 \caption{Recovery thresholds for Problem~\eqref{primal}: the worst-case recovery threshold $p_{\rm true}^* = 0.585$ is given by Theorem~\ref{main} and the exact recovery threshold is
 obtained by solving Problem~\eqref{exact} numerically for each $d \in [2,50]$.}
\label{figure1}
\end{figure}

We conclude this section by acknowledging that this theoretical study serves as a starting point for understanding the recovery properties of LP relaxations for joint object matching. Obtaining recovery guarantees for the general problem with partially similar objects and incomplete map graphs together with investigating the impact of consistency inequalities~\eqref{superg} on the power of LP relaxations are topics of future research.

\section{Numerical Experiments}
\label{sec:numerics}

In this section, we conduct a preliminary numerical study to demonstrate the desirable
numerical properties of proposed LP relaxations for joint object matching.  A comprehensive
computational study that includes various real data sets from the literature is a topic of future research.
Throughout this section, we focus on the special case of permutation group synchronization problem and generate problem instances according to the random corruption model~\eqref{unifmodel}. Moreover, our numerical experiments are all performed on the {\tt NEOS} server~\cite{neos98}.

\subsection{Exact solution of the ILP}

We start by demonstrating that even for small instances,
the exact solution of joint object matching using the state-of-the-art MIP solvers is beyond reach. To this end, we solve Problem~\eqref{ip} using {\tt GAMS/Gurobi}~\cite{gurobi}.
We set a time limit of $25,000$ seconds; all other parameters are set to their default values. We set $n=20$, $d \in \{3,4\}$, and $p_{\rm true} \in \{0.30, 0.40, 0.50\}$.
For each combination of $(n,d,p_{\rm true})$, we run 5 random instances. We find that the MIP solver is unable to solve {\rm any} of these instances
to optimality within the time limit. More detailed results are shown in Table~\ref{table1}; for each $(n,d,p_{\rm true})$, we provide the average relative gap of the MIP upon termination. The relative gap is defined as $g_{\rm rel} =\frac{|BP-BF|}{|BF|}$, where $BF$ is the objective function value of the current best integer solution, while $BP$ is the best possible integer solution. For each case, we also report the average solution time of the corresponding basic LP given by Problem~\eqref{primal}. As practical instances of joint object matching have larger values for $(n,d)$, this experiment conveys the need for developing efficient convex relaxations for this problem.

\begin{table}
\centering
\caption{Exact solution of permutation group synchronization with $n =20$ using the MIP solver {\tt Gurobi}. All CPU times are reported in seconds and are averaged over
5 random instances. $g_{\rm rel}$ denotes the MIP relative gap upon termination and is averaged over 5 random instances.\\}
\label{table1}
{\begin{tabular}{l|l|l|l}
\toprule														
$(d, p_{\rm true})$	&	MIP time  & $g_{\rm rel}$ (\%)& LP time \\
\midrule	
$(3, \; 0.3)$  &    $> 25000$   &   $11.5\%$ &    $3.8$\\
$(3, \; 0.4)$  &    $> 25000$   &   $7.5 \%$ &    $2.2$\\
$(3,\; 0.5)$  &    $>25000$     &   $10.5\%$ &    $2.8$\\
$(4, \; 0.3)$  &    $> 25000$   &   $15.5\%$ &     $5.3$\\
$(4, \; 0.4)$  &    $> 25000$   &   $24.9\%$ &     $10.2$\\
$(4, \; 0.5)$  &    $> 25000$   &   $21.1\%$ &     $11.4$ \\
\bottomrule															
\end{tabular}}															
\end{table}

\subsection{The basic LP versus the SDP}

Next we compare recovery properties of the basic LP relaxation~\eqref{primal} with those of the popular SDP relaxation~\eqref{sdp} strengthened with constraint~\eqref{superSDP}. In the following, we consider a more general random model than the random corruption model given by~\eqref{unifmodel}; namely, we do not assume that the map graph is complete. Let $p_{\rm obs} \in (0, 1]$; for each $1 \leq i < j \leq n$, with probability $p_{\rm obs}$
the input map $X^{\rm in}(i,j)$ is generated according to~\eqref{unifmodel}; otherwise, the input map between objects $i$ and $j$ is not observed and we set $a_{tq}(i,j) = 0$ for all $t,q \in [d]$.
For our numerical experiments, we set $n = 20$, $d \in \{3,4,5\}$, and $p_{\rm obs} \in \{\frac{1}{2}, 1\}$. We set $p_{\rm true} \in [\underline p: 0.02: \bar p]$,
where $[\underline p: 0.02: \bar p]$
denotes a regularly-spaced vector between $\underline p $ and $\bar p$ using 0.02 as the increment between elements, and where $\underline p$ (resp. $\bar p$) is chosen small enough (resp. large enough) so that the recovery rate is zero (resp. one).
For each combination of $(n,d,p_{\rm true}, p_{\rm obs})$ we conduct 50 random trials.
We count the number of times the optimization algorithm returns the
ground truth $\bar X$ as the optimal solution; dividing this number by total number of trials, we
obtain the empirical recovery rate.
In addition to the empirical recovery rate, we compute the empirical \emph{tightness rate}. That is, we compute the fraction of times the optimization algorithm returns a binary solution. Recall that if the solution of the LP/SDP relaxation is binary, then it is feasible for Problem~\eqref{jom} and hence is optimal for the original nonconvex problem.

\begin{figure}[htbp]
 \centering
% \psfrag{$p$}{$p_{\rm true}$}
 \subfigure [$p_{\rm obs}= 1.0, \; d=3$]{\label{fig2a}\epsfig{figure=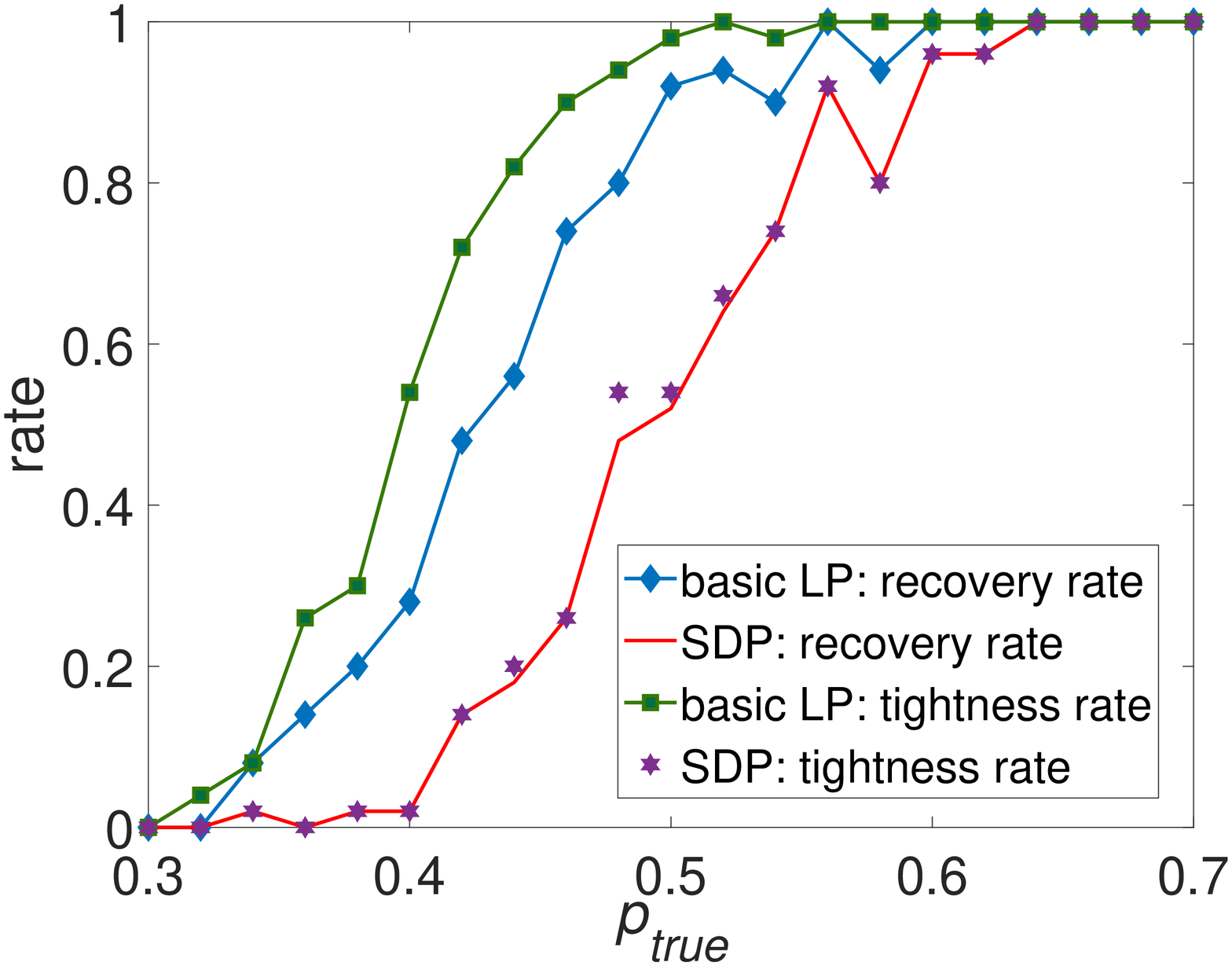, scale=0.225, trim=25mm 0mm 130mm 0mm, clip}}
   \subfigure[$p_{\rm obs}= 0.5, \; d=3$]{\label{fig2d}\epsfig{figure=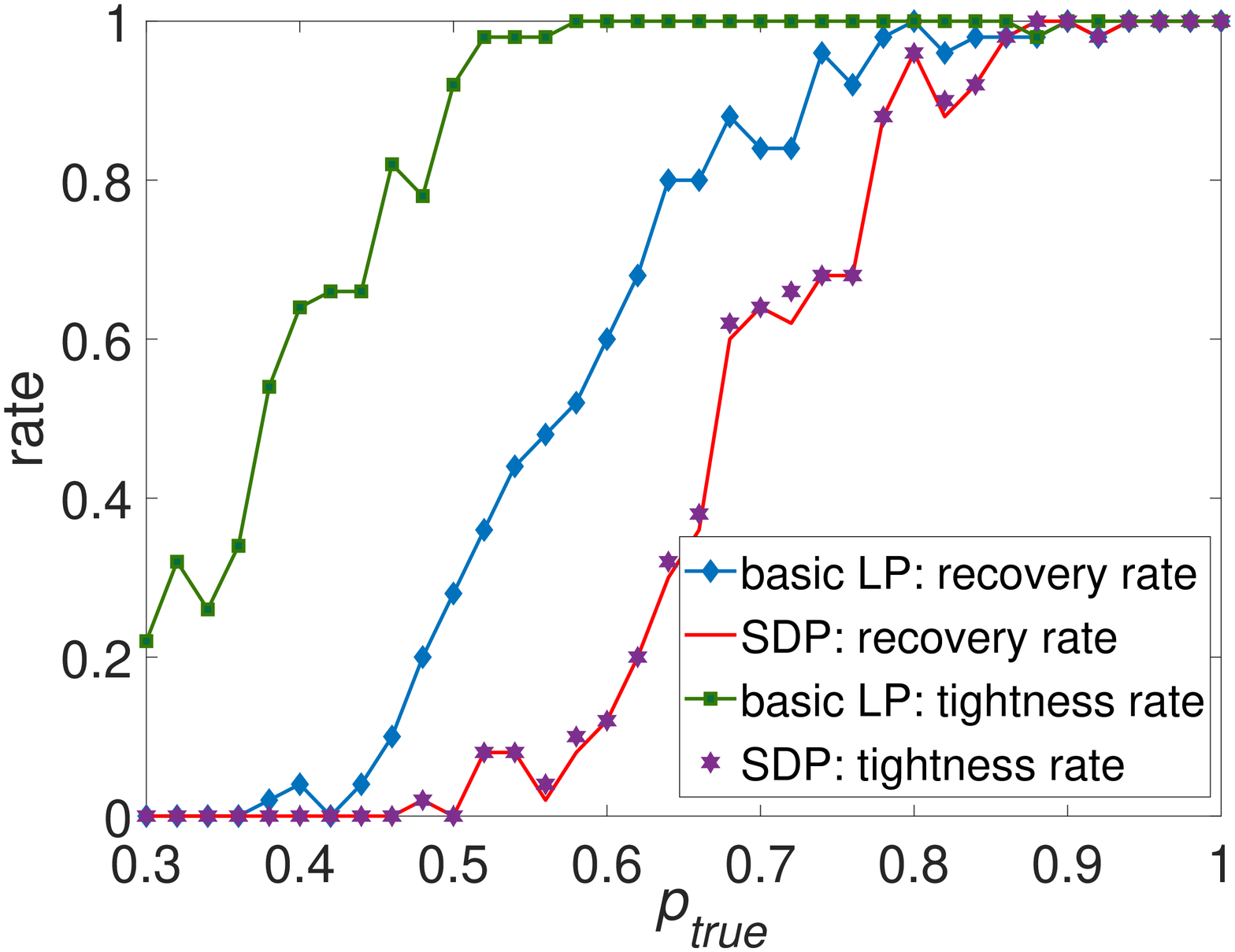, scale=0.225, trim=25mm 0mm 130mm 0mm, clip}}
 \subfigure [$p_{\rm obs}= 1.0, \; d=4$]{\label{fig2b}\epsfig{figure=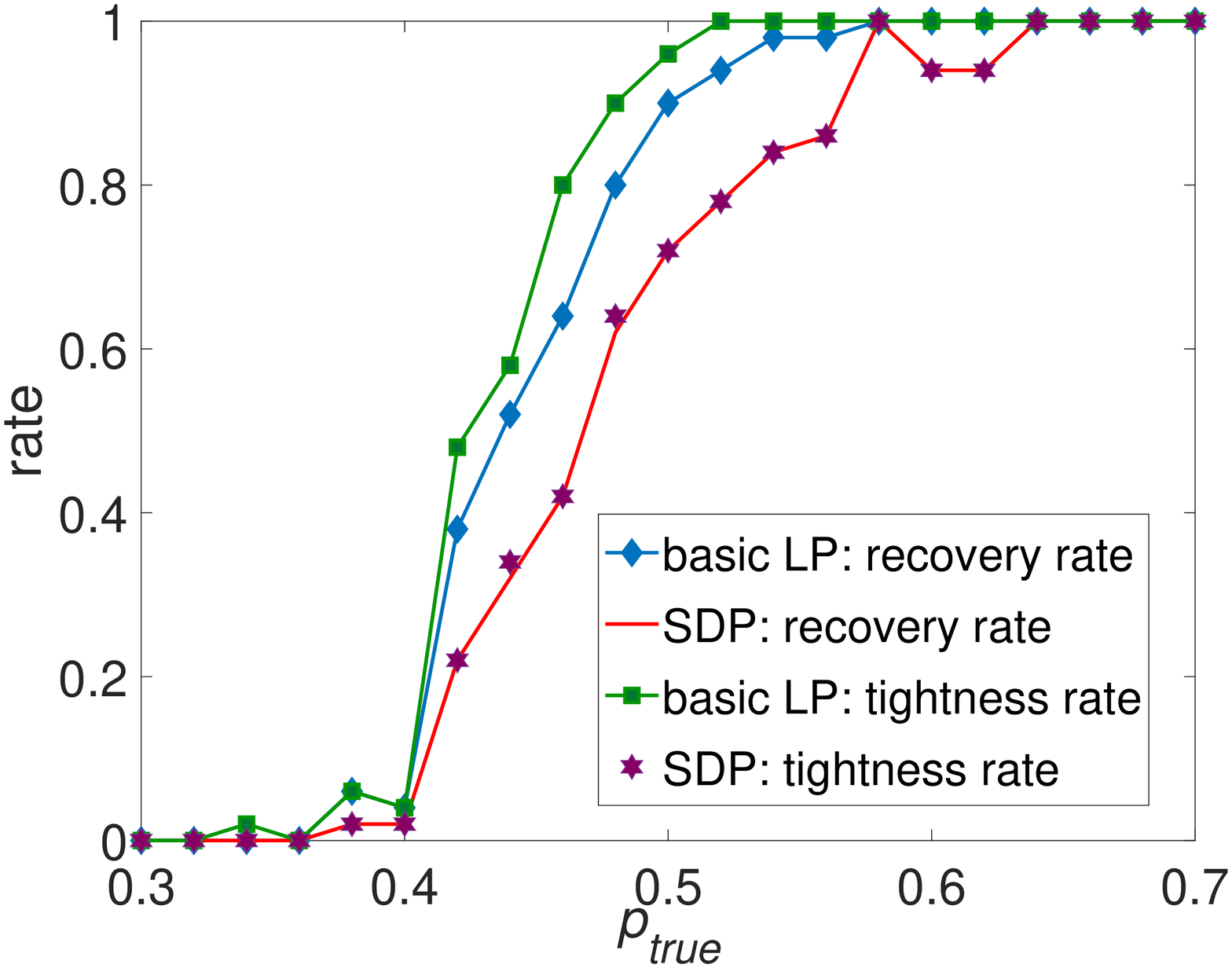, scale=0.225, trim=25mm 0mm 130mm 0mm, clip}}
  \subfigure[$p_{\rm obs}= 0.5, \; d=4$]{\label{fig2e}\epsfig{figure=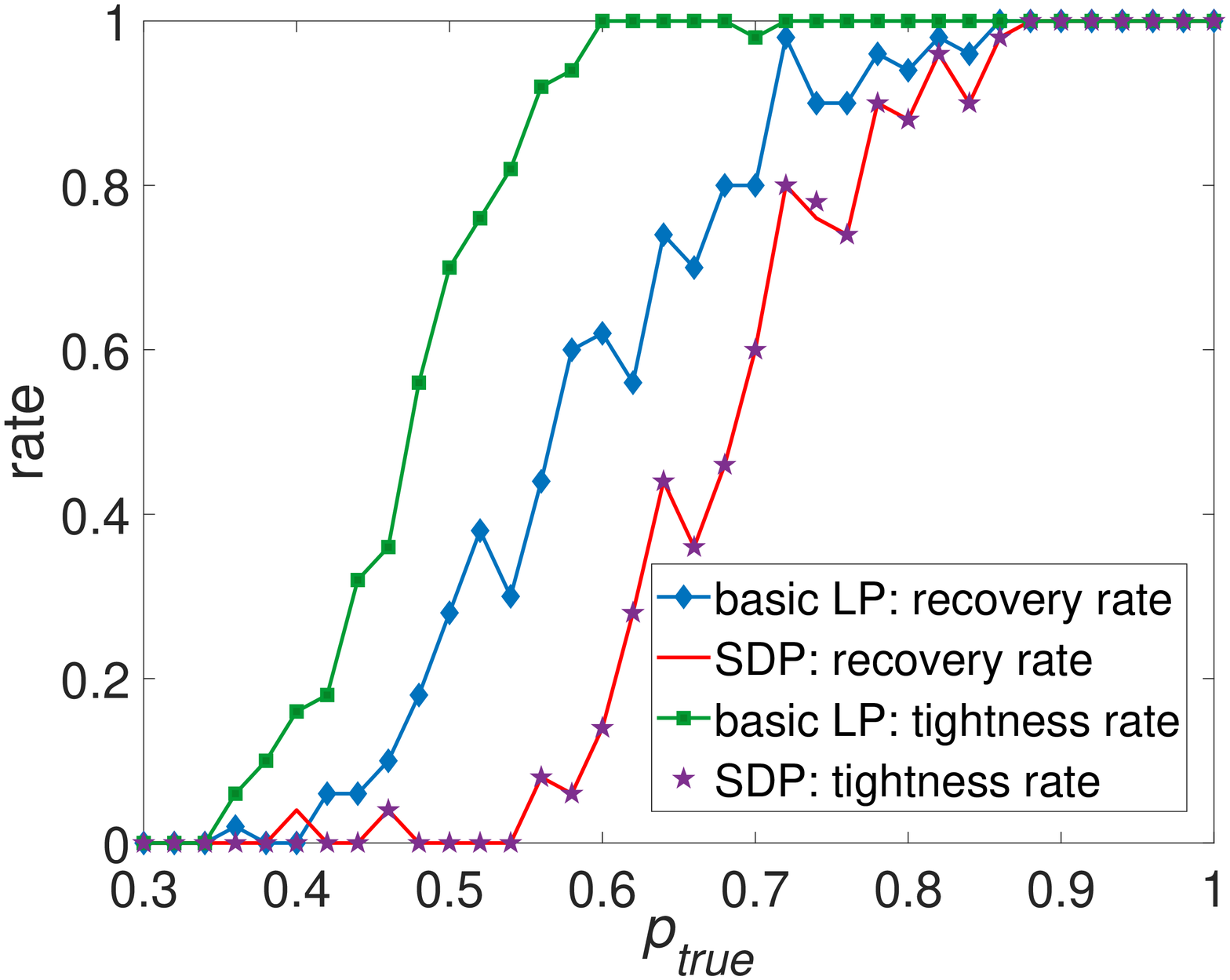, scale=0.225, trim=25mm 0mm 130mm 0mm, clip}}
 \subfigure[$p_{\rm obs}= 1.0, \; d=5$]{\label{fig2c}\epsfig{figure=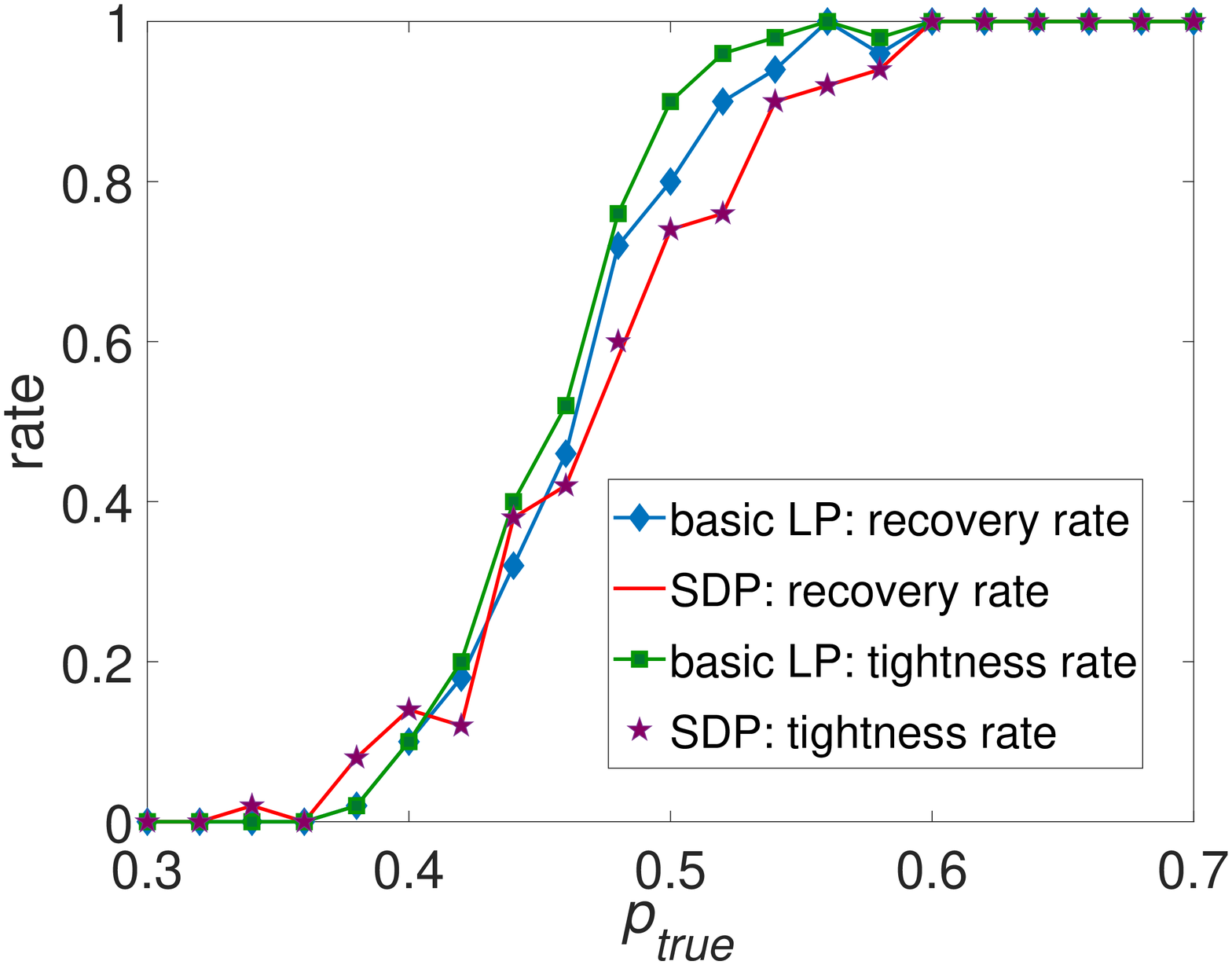, scale=0.225, trim=25mm 0mm 130mm 0mm, clip}}
    \subfigure[$p_{\rm obs}= 0.5, \; d=5$]{\label{fig2f}\epsfig{figure=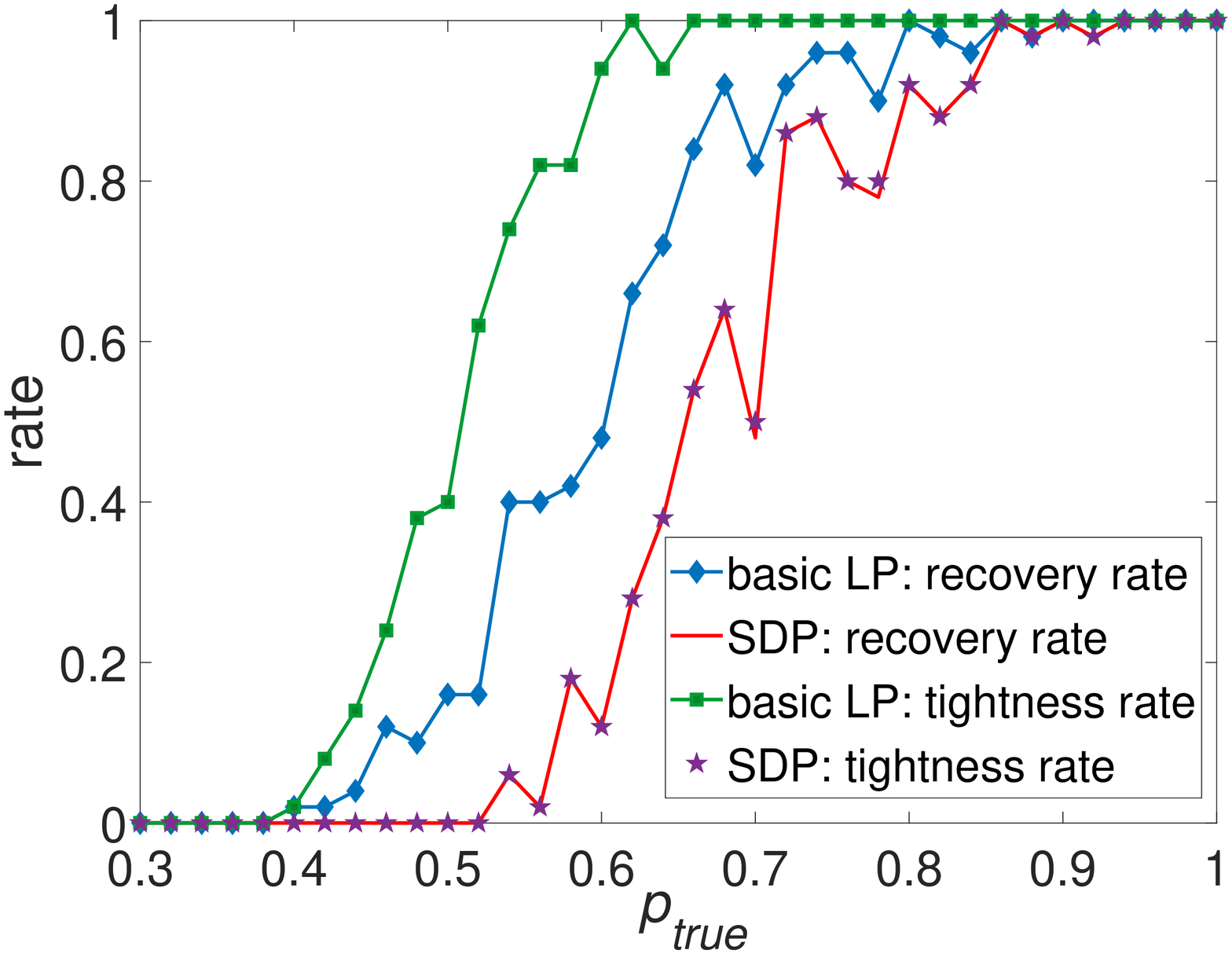, scale=0.225, trim=25mm 0mm 130mm 0mm, clip}}
 \caption{Empirical rates of recovery and tightness of the
basic LP relaxation versus the SDP relaxation under the random corruption model.}
\label{figure2}
\end{figure}

\begin{figure}[htbp]
 \centering
% \psfrag{$p$}{$p_{\rm true}$}
 \epsfig{figure=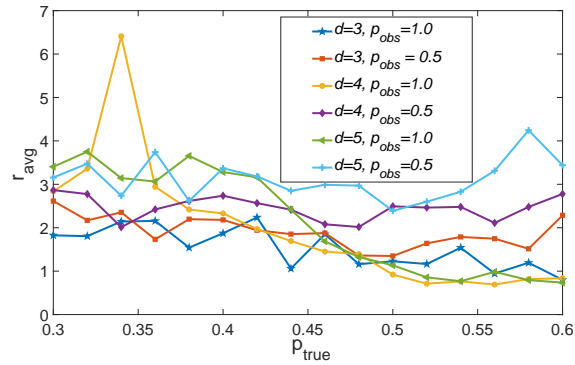, scale=0.25, trim=10mm 0mm 50mm 0mm, clip}
 \caption{Comparing the CPU times of solving the basic LP relaxation ($T_{\rm LP}$) vs solving the SDP relaxation ($T_{\rm SDP}$), where we define $r_{\rm avg} = \avg[\frac{T_{\rm SDP}}{T_{\rm LP}}]$.}
\label{figure4}
\end{figure}

All LPs and SDPs are solved with {\tt GAMS/MOSEK}~\cite{mosek}.
Results are shown in Figure~\ref{figure2}: in all cases, the basic LP outperforms the SDP in terms of recovery and tightness.
Moreover, the CPU times are compared in Figure~\ref{figure4}. As can be seen from this figure, in most of the experiments the LP solver is 2-4 times faster than the SDP solver. We should remark that we ran these experiments setting all options to their default values. That is, we did not use any technique to expedite either the LP or the SDP solver.
In particular, for solving the LP relaxation employing a cutting plane type algorithm together with dual simplex often leads to significant speedups.
%Moreover, these results suggest that our sufficient condition for recovery of the basic LP with $p_{\rm obs} = 1$, stated in Theorem~\ref{main}, is not too conservative.
From this experiment we make three important observations:
\begin{itemize}[leftmargin=*]
\item [(i)] Unlike the SDP relaxation, the quality of the basic LP relaxation degrades by increasing $d$; this is due to the fact that
the basic LP includes only $\theta(d)$ of consistency inequalities which are $\theta(d 2^{2d})$ in total; as we detail next, we address this shortcoming by considering a
stronger LP relaxation, \ie Problem~\eqref{lp2}.

\item [(ii)] Unlike the SDP relaxation, the quality of the basic LP relaxation does \emph{not}
quickly degrade by decreasing $p_{\rm obs}$. For the random model described above and for the SDP relaxation, the authors of~\cite{CheGuiHua14} obtain a recovery guarantee of the form  $p_{\rm true} > \theta(\frac{1}{\sqrt{p_{\rm obs}}})$. Recall that our theoretical analysis in Section~\ref{sec:dualCertificate} is under the assumption that the map graph is complete, \ie $p_{\rm obs} = 1.0$. Our numerical experiments suggest that in case of the basic LP, as a function of  $p_{\rm obs}$, the recovery threshold grows slower than
$\frac{1}{\sqrt{p_{\rm obs}}}$; we plan to explore the exact form of this relationship as a next step.

\item [(iii)] Unlike the SDP relaxation, the basic LP relaxation returns a binary solution in many cases for which it fails in recovering the ground truth. Recall that if the LP solution is binary, then it is optimal for Problem~\eqref{jom}. Indeed, in most practical settings, asking for the exact recovery of the ground truth is not realistic. Hence obtaining sufficient conditions for the tightness of the LP relaxation under proper stochastic models is of key importance and is a topic of future research.

\end{itemize}

\begin{figure}[htbp]
 \centering
% \psfrag{$p$}{$p_{\rm true}$}
 \subfigure [$p_{\rm obs}= 1.0$]{\label{fig3a}\epsfig{figure=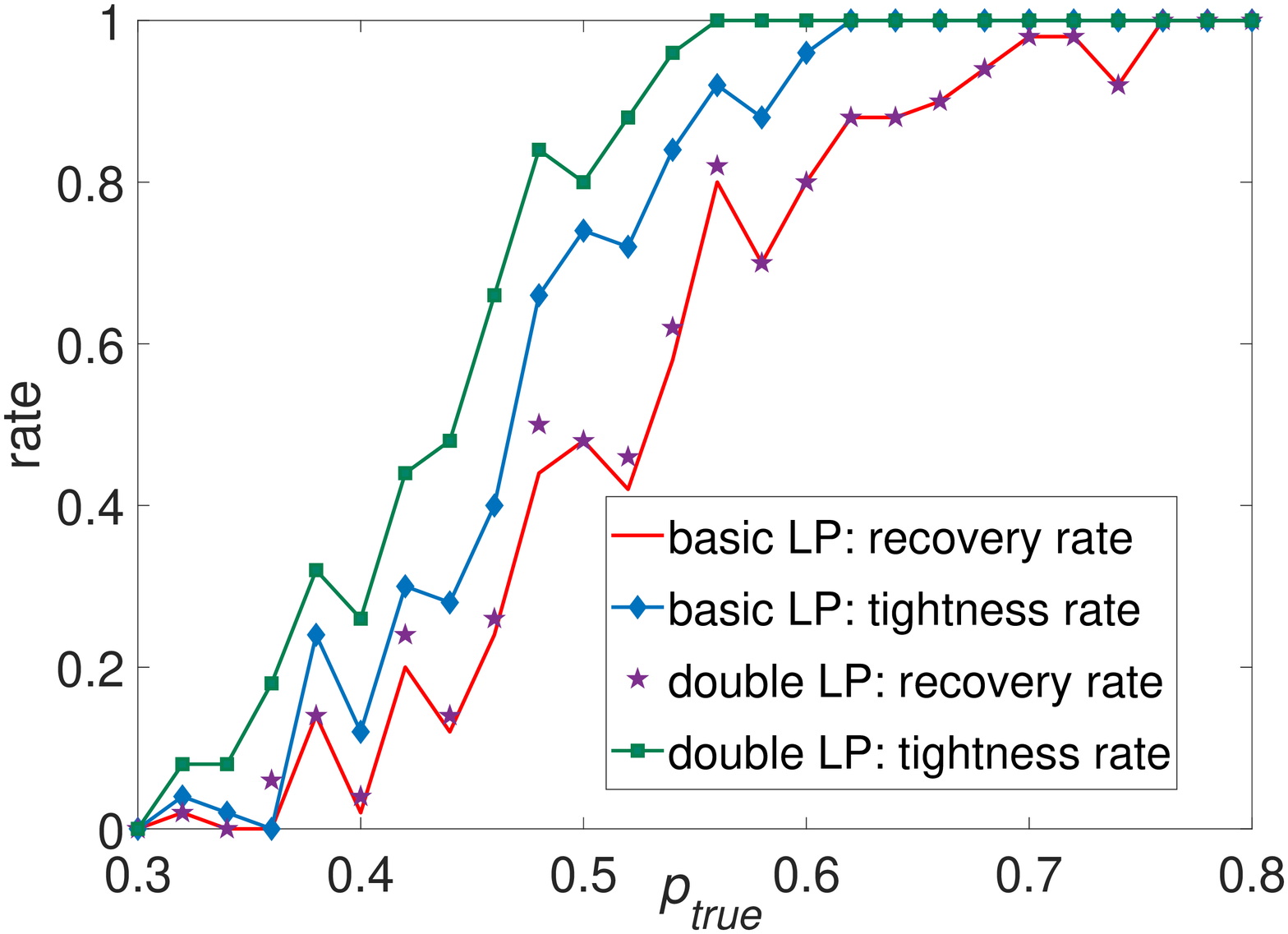, scale=0.23, trim=25mm 0mm 120mm 0mm, clip}}
   \subfigure[$p_{\rm obs}= 0.75$]{\label{fig3b}\epsfig{figure=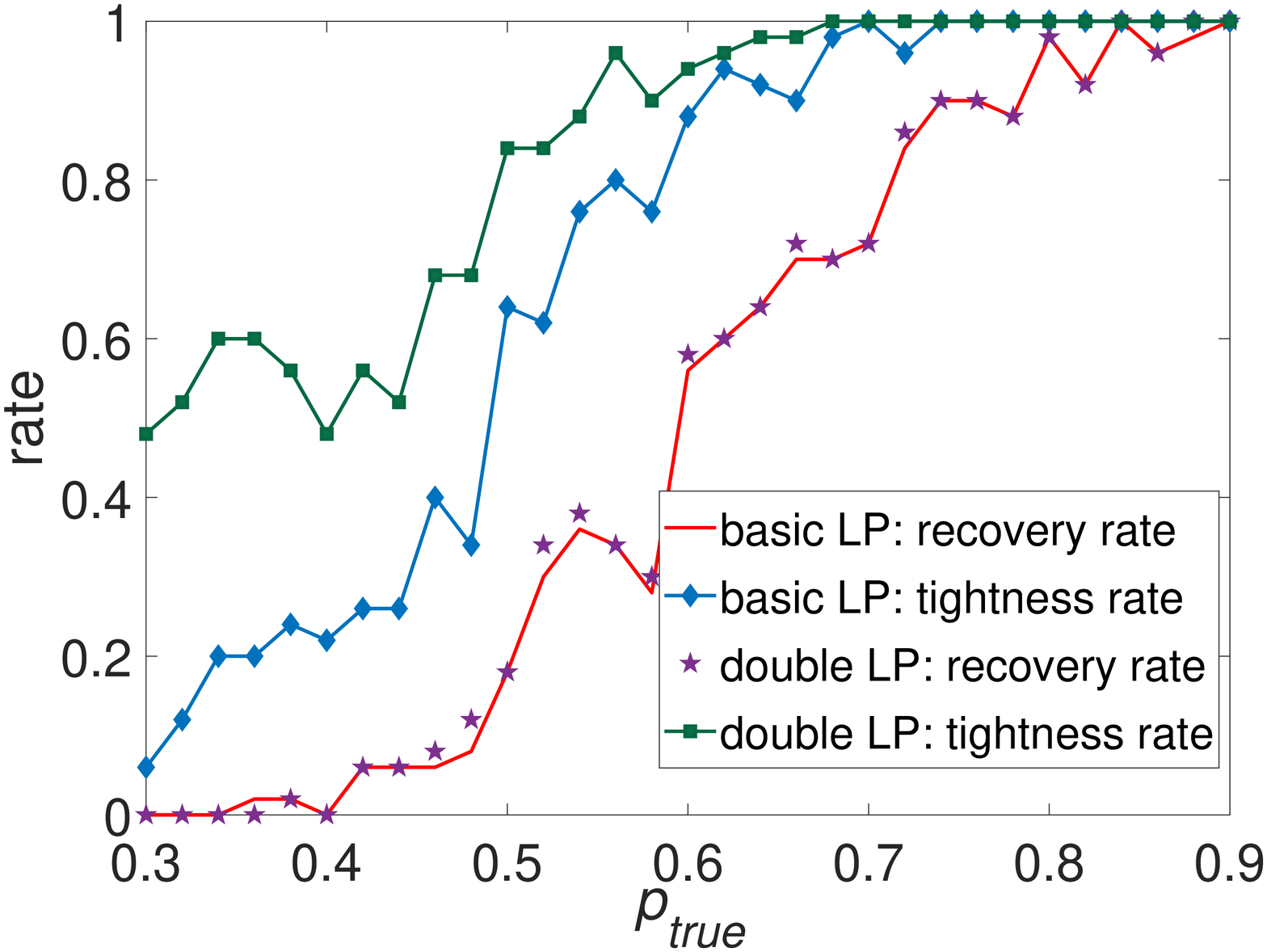, scale=0.23, trim=25mm 0mm 120mm 0mm, clip}}
 \subfigure [$p_{\rm obs}= 0.5$]{\label{fig3c}\epsfig{figure=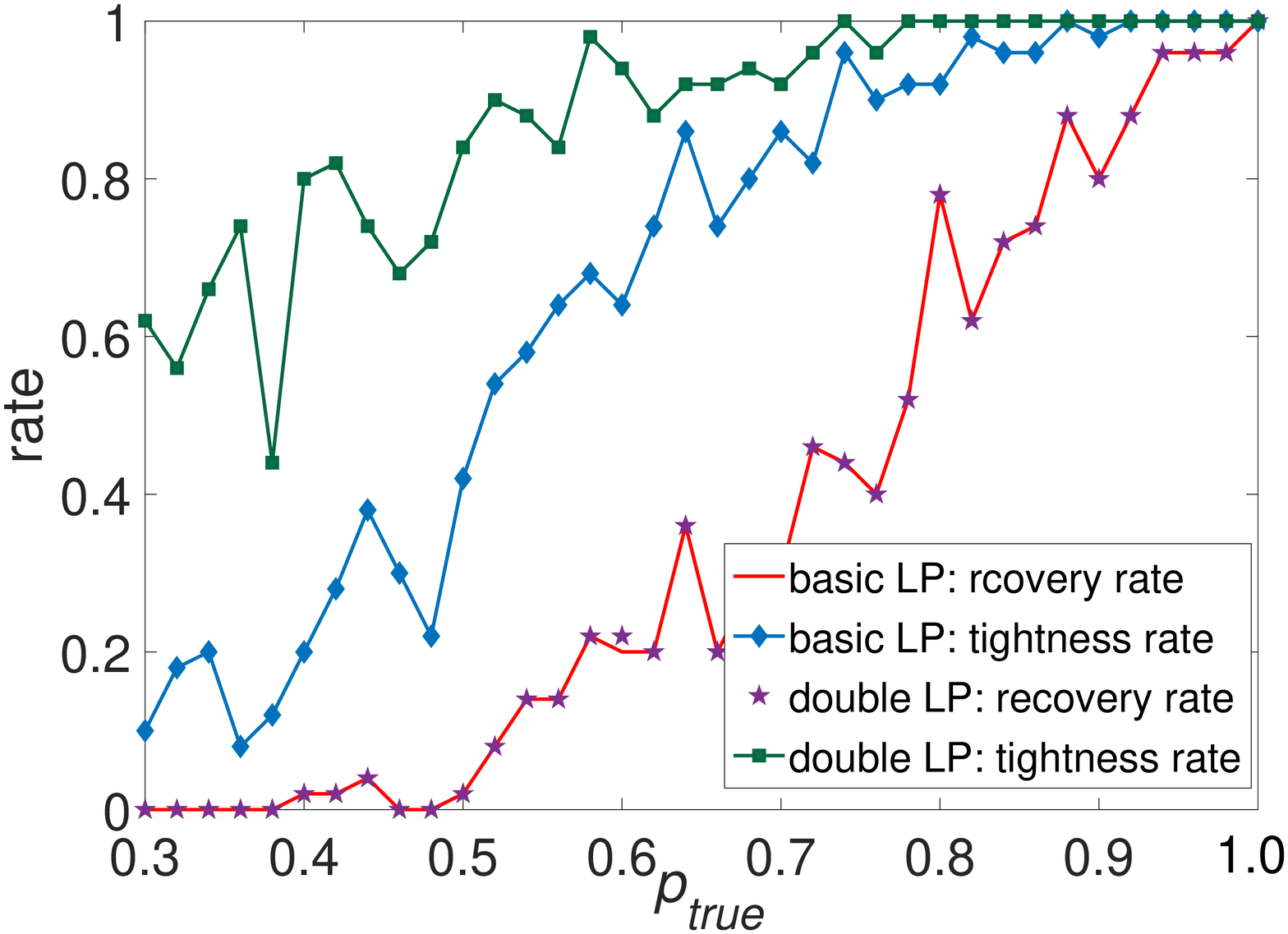, scale=0.23, trim=25mm 0mm 120mm 0mm, clip}}
 \caption{Empirical rates of recovery and tightness of the
basic LP relaxation versus the double LP relaxation under the random corruption model.}
\label{figure3}
\end{figure}

\subsection{The basic LP versus the double LP}

We now illustrate the impact of consistency inequalities~\eqref{superg} in strengthening the basic LP relaxation~\eqref{primal}. To this end, we consider the following simple two-step algorithm:
\begin{enumerate}
\item The basic LP relaxation, \ie Problem~\eqref{primal} is solved; if the optimal solution $\tilde X$ is binary-valued, then the algorithm terminates and returns
$\tilde X$ as the optimal solution.
\item If $\tilde X$ is not binary-valued, then at most $1000$
consistency inequalities that violate  $\tilde X$ are generated by solving the separation problem~\eqref{separate}. The cutting planes are added to the basic LP and the augmented LP is solved using the dual simplex algorithm with $\tilde X$ as the starting point. The solution of this augmented LP is reported as the optimal solution.
\end{enumerate}
We refer to this two-step algorithm as the~\emph{double LP}.
To understand the impact of consistency inequalities~\eqref{superg}, we compare the performance of basic LP and double LP with respect to recovery and tightness rates.

We set $n = 10$, $d = 5$, $p_{\rm obs} \in \{0.5, 0.75, 1\}$, and $p_{\rm true} \in [0.3:0.02:1.0]$. As before, for each combination of $(n,d, p_{\rm true},p_{\rm obs})$, we conduct 50 random trials. As the cutting plane algorithm described above cannot be efficiently implemented in the {\tt GAMS} modeling language, the two-step algorithm is implemented in {\tt JuMP}~\cite{DunningHuchetteLubin2017} and all corresponding LPs are solved with {\tt Gurobi}~\cite{gurobi}. Results are depicted in Figure~\ref{figure3}; as can be seen from this figure, while there is no visible difference between the recovery rates of the basic LP and the double LP, the tightness rate of the double LP is significantly better than that of the basic LP. This improvement is more significant for {\rm sparser} problems; that is, problems with smaller $p_{\rm obs}$.
We believe that this phenomenon is due to the cutting plane strategy employed in Step~2 of double LP and is not an inherent property of consistency inequalities.
Namely, in order to keep the overall computational cost low, in Step~2 of double LP, we add the first 1000 violated consistency inequalities obtained by solving the separation problem~\eqref{separate}.
For denser problems, that is, for problems with larger $p_{\rm obs}$, in almost all cases,  many more consistency inequalities are violated by $\tilde X$.
Hence, an upper bound of 1000 for the number of violated inequalities is often too small. However, increasing this number will increase the cost of solving the augmented LP. Hence, for dense problems, a more elaborate strategy such as adding top 1000 most violated cuts or conducting multiple rounds of cut generation seems more appropriate.

We conclude by emphasizing that our goal in this experiment was to convey the effectiveness of consistency inequalities~\eqref{superg}. Indeed, solving double LP is not equivalent to solving Problem~\eqref{lp2} as double LP contains only a small subset of consistency inequalities~\eqref{superg}. An efficient solution of Problem~\eqref{lp2} via devising effective cut generation and cut management algorithms for consistency inequalities is a subject of future research. Finally, an effective incorporation of consistency inequalities in a branch-and-cut framework to solve Problem~\eqref{jom} to global optimality is an interesting future direction as well.

\section{Appendix}
\label{appendix}

\subsection{Facetness for block consistency inequalities}
\label{appendix1}

\paragraph{Proof of Proposition~\ref{newfacet}.} Without loss of generality, we prove for any nonempty $D_1, D_2, D_3 \subseteq [d]$
with $|D_1|+|D_2| > |D_3|$, the following defines a facet of the joint matching polytope $\C_{n,d}$:
\begin{equation}\label{nf1}
\sum_{l \in D_3}{\Big(\sum_{t\in D_1}{X_{lt}(1,2)}+\sum_{q\in D_2}{X_{lq}(1,3)}\Big)}-\sum_{t \in D_1}{\sum_{q \in D_2}{X_{tq}(2,3)}} \leq |D_3|.
\end{equation}
We start by identifying the set of consistent partial maps in $\C_{n,d}$ that satisfy inequality~\eqref{nf1} tightly. Subsequently, we show that any nontrivial valid
inequality $\alpha X \leq \beta$ for $\C_{n,d}$ that is satisfied tightly at all such maps coincides with~\eqref{nf1} up to a positive
scaling. Since $\C_{n,d}$ is full dimensional, this in turn implies that inequality~\eqref{nf1} defines a facet of $\C_{n,d}$.

A consistent partial map is binding for inequality~\eqref{nf1}, if for every $l \in D_3$, there exists $e_l \in \M_X$ with $1_l \in e_l$ satisfying one of the following conditions:
\begin{itemize}
\item [(i)]  $2_t \in e_l$ for some $t \in D_1$ and $3_q \notin e_l$ for all $q \in D_2$,
\item [(ii)] $2_t \notin e_l$ for all $t \in D_1$ and $3_q \in e_l$ for some $q \in D_2$,
\item [(iii)] $2_t \in e_l$ for some $t \in D_1$ and $3_q \in e_l$ for some $q \in D_2$.
\end{itemize}
%Moreover, $e_l$ should not contain any element other than those stated above.
%Notice that a tight point satisfying condition~(i) or (ii) above for every $l \in D_3$ corresponds to a $|D_3|$
%
Now consider a consistent partial map $\M^1_X$ satisfying conditions~(i) or~(ii) above for all $l \in D_3$ such that for some $t' \in D_1$ we have $2_{t'} \notin e_l$ for all $l \in D_3$. Notice that such a consistent partial map exists since by assumption $|D_1| +|D_2| > |D_3|$.
Moreover, suppose that $\M^1_X$ contains no matched pairs other than those required by conditions~(i)-(ii), \ie $\M^1_X = \{e_l,\; l \in D_3\}$ with $|e_l|= 2$ for all $l \in D_3$. Next consider another consistent partial map obtained from $\M^1_X$ by replacing a matched pair of the form $(1_{\hat l},2_{\hat t})$ for some $\hat t \in D_1$
and  $\hat l \in D_3$ by $(1_{\hat l},2_{t'})$, where $t' \in D_1$ is the index defined above.
Notice that this \emph{flipping} operation results in a consistent partial map that is also binding for~\eqref{nf1}. Substituting these two partial maps in $\alpha X = \beta$ yields
$\alpha_{\hat l \hat t}(1,2) =\alpha_{\hat l t'}(1,2)$. Using a similar line of arguments for all possible consistent partial maps satisfying conditions~(i) or~(ii) for all $l \in D_3$ together with $2_{t'} \notin e_l$ for some $t' \in D_1$ ($3_{t'} \notin e_l$ for some $t' \in D_2$) for all $l \in D_3$, we conclude that for each $l \in D_3$, we have
\begin{equation}\label{nf2}
\alpha_{lt}(1,2) = \alpha_{lq}(1,3), \quad \forall t \in D_1, \; \forall q \in D_2.
\end{equation}
Next consider a consistent partial map $\M^2_X$ satisfying conditions~(i) or~(ii) above for all $l \in \D_3 \setminus \{\hat l\}$ and suppose that for $\hat l$ condition~(iii) is satisfied, \ie
$e_{\hat l} = (1_{\hat l}, 2_{\hat t}, 3_{\hat q})$ for some $\hat t \in D_1$ and $\hat q \in D_2$. Consider a second consistent partial map satisfying
conditions~(i) or~(ii) above for all $l \in \D_3 \setminus \{\tilde l\}$ where $\tilde l \neq \hat l$ and suppose that for $\tilde l$ condition~(iii) is satisfied with
$e_{\tilde l} = (1_{\tilde l}, 2_{\hat t}, 3_{\hat q})$. In addition, suppose that in both these maps, no additional matched pairs other than those
required by conditions~(i)-(iii) exist. Substituting these two maps in $\alpha X = \beta$ and using~\eqref{nf2} yield $\alpha_{\hat l \hat t}(1,2) = \alpha_{\tilde l \hat t}(1,2)$ and $\alpha_{\hat l \hat q}(1,3) = \alpha_{\tilde l \hat q}(1,3)$. Using a similar line of arguments for all possible pairs of partial maps satisfying assumptions above, we obtain:
\begin{equation}\label{nf3}
\alpha_{lt}(1,2) = \alpha_{lq}(1,3) = \frac{\beta}{|D_3|}, \quad \forall t \in D_1, \; \forall q \in D_2, \; \forall l \in D_3.
\end{equation}
Moreover, substituting the partial map corresponding to $\M^2_X$ in $\alpha X = \beta$ and using~\eqref{nf3} we obtain $\alpha_{\hat t \hat q}(2,3) = -\beta/|D_3|$.
It then follows that
\begin{equation}\label{nf4}
\alpha_{tq}(2,3) = -\frac{\beta}{|D_3|}, \quad \forall t \in D_1, \; \forall q \in D_2.
\end{equation}
Next consider a consistent partial map $\M^3_X$ satisfying conditions~(i) or~(ii) above for all $l \in \D_3$ of the form $\M^3_X = \{e_l: l\in D_3\}$
with $|e_l| = 2$ for all $l \in D_3$. Construct another consistent partial map of the form $\M^3_X \cup \{(i_t, j_q)\}$ for some
$(i,j,t,q) \in \Q$, where
\begin{align*}
Q :=& \Big\{(i,j,t,q): 4 \leq i < j \leq n, t \in [d], q \in [d]\Big\} \cup \Big\{(1,j,t,q): j \geq 4, t \in [d] \setminus D_3, q \in [d]\Big\}\\
\cup & \Big\{(2,j,t,q): 4 \leq j \leq n, t \in [d] \setminus D_1, q \in [d]\Big\}
 \cup   \Big\{(3,j,t,q): 4 \leq j \geq n, t \in [d] \setminus D_2, q \in [d]\Big\}\\
 \cup & \Big\{(1,2,t,q): t \in [d] \setminus D_3, q \in [d] \setminus D_1\Big\} \cup \Big\{(1,3,t,q): t \in [d] \setminus D_3, q \in [d] \setminus D_2\Big\}\\
 \cup & \Big\{(2,3,t,q): t \in [d] \setminus D_1, q \in [d] \setminus D_2\Big\}.
\end{align*}
Substituting these two partial maps in $\alpha X = \beta$ yields
\begin{equation}\label{nf5}
\alpha_{tq}(i,j) = 0, \quad \forall (t,q,i,j) \in \Q.
\end{equation}
Consider the consistent partial map $\M^3_X$ defined above with the additional assumption that for some $\hat t \in D_1$ (resp. $\hat q \in D_2$), we have
$(1_l, 2_{\hat {t}}) \notin \M^3_X$ (resp. $(1_l, 3_{\hat {q}}) \notin \M^3_X$)
for all $l \in D_3$. Notice that such a partial map always exists since by assumption $|D_1|+|D_2| > |D_3|$.  Construct another consistent partial map of the form:
\begin{itemize}
\item $\bar \M^3_X = \M^3_X \cup \{(1_{\hat l}, 2_{\hat t})\}$ (resp. $\bar \M^3_X =\M^3_X \cup \{(1_{\hat l}, 3_{\hat q})\}$) for some $\hat l \notin D_3$. Substituting $\M^3_X, \bar \M^3_X$
in $\alpha X = \beta$ gives
$\alpha_{\hat l \hat t}(1,2) =0$ (resp. $\alpha_{\hat l \hat t}(1,3) =0$). Using a similar line of arguments for all possible partial maps satisfying the assumptions above, we obtain
\begin{equation}\label{nf6}
\alpha_{lt}(1,2) = \alpha_{lq}(1,3) = 0, \quad \forall l \in [d] \setminus D_3, \forall t \in D_1, \forall q \in D_2.
\end{equation}

\item $\bar \M^3_X = \M^3_X \cup \{(2_{\hat t}, j_s)\}$ (resp. $\bar \M^3_X =\M^3_X \cup \{(3_{\hat q}, j_s)\}$ ) for some $4 \leq j \leq n$ and some $s \in [d]$. Substituting $\M^3_X, \bar \M^3_X$ in $\alpha X = \beta$ gives
$\alpha_{\hat t s }(2,j) =0$ (resp. $\alpha_{\hat q s}(3, j) =0$). Using a similar line of arguments for all possible partial maps satisfying the assumptions above, we obtain
\begin{equation}\label{nf7}
\alpha_{ts}(2,j) = \alpha_{qs}(3,j) = 0, \quad \forall 4 \leq j \leq n, \forall t \in D_1, \forall q \in D_2, \forall s \in [d].
\end{equation}
\item $\bar \M^3_X = \M^3_X \cup \{(2_{\hat t}, 3_q)\}$ for some $q \in [d] \setminus D_2$ (resp. $\bar \M^3_X = \M^3_X \cup \{(2_t, 3_{\hat q})\}$
for some $t \in [d] \setminus D_1$). Substituting $\M^3_X, \bar \M^3_X$ in $\alpha X = \beta$ gives
$\alpha_{\hat t s }(2,j) =0$ (resp. $\alpha_{\hat q s}(3, j) =0$). Using a similar line of arguments for all possible partial maps satisfying the assumptions above, we obtain
\begin{equation}\label{nf9}
\alpha_{tq}(2,3) = 0, \quad \forall t \in [d]\setminus D_1, q \in D_2 \; {\rm or} \; \forall t \in D_1, q \in [d] \setminus D_2.
\end{equation}
\end{itemize}
Consider a consistent partial map $\M^4_X$ satisfying conditions~(i) or~(ii) above for all $l \in \D_3$ of the form $\M^4_X = \{e_l: l\in D_3\}$
with $|e_l| = 2$ for all $l \in D_3$. Consider some $\hat l \in D_3$ for which $e_{\hat l} = (1_{\hat l}, 2_{\hat t})$ for some $\hat t \in D_1$
(resp. $e_{\hat l} = (1_{\hat l}, 3_{\hat q})$ for some $\hat q \in D_2$).
Construct another consistent partial map of the form
\begin{itemize}
\item $\bar \M^4_X = \M^4_X \cup \{(1_{\hat l}, 3_{\hat q}), (2_{\hat t}, 3_{\hat q})\}$ for some $\hat q \in [d] \setminus D_2$
($\bar \M^4_X = \M^4_X \cup \{(1_{\hat l}, 2_{\hat t}), (2_{\hat t}, 3_{\hat q})\}$ for some $\hat t \in [d] \setminus D_1$).
Substituting these two partial maps in $\alpha X = \beta$ and using~\eqref{nf9} we obtain $\alpha_{\hat l \hat t}(1,2) = 0$ (resp. $\alpha_{\hat l \hat q}(1,3) = 0$).
More generally it can be checked that
\begin{equation}\label{nf8}
\alpha_{lt}(1,2) = \alpha_{lq}(1,3) = 0, \quad \forall l \in D_3, \forall t \in [d] \setminus D_1, \forall q \in [d] \setminus D_2.
\end{equation}

\item $\bar \M^4_X = \M^4_X \cup \{(1_{\hat l}, j_s), (2_{\hat t}, j_s)\}$ (resp. $\bar \M^4_X = \M^4_X \cup \{(1_{\hat l}, j_s), (3_{\hat q}, j_s)\}$)
for some $4 \leq j \leq n$ and $s \in [d]$. Substituting these two partial maps in $\alpha X = \beta$ and using~\eqref{nf7} we obtain $\alpha_{\hat l s}(1,j) = 0$.
More generally, it can be checked that
\begin{equation}\label{nf10}
\alpha_{ls}(1,j) = 0, \quad \forall 4 \leq j \leq n, \forall l \in D_3, \forall s \in [d].
\end{equation}
\end{itemize}
From~\eqref{nf1}-\eqref{nf10} it follows that the inequality $\alpha X \leq \beta$ can be equivalently written as
$$\frac{\beta}{|D_3|} \Big(\sum_{l \in D_3}{\Big(\sum_{t\in D_1}{X_{lt}(1,2)}+\sum_{q\in D_2}{X_{lq}(1,3)}\Big)}-\sum_{t \in D_1}{\sum_{q \in D_2}{X_{tq}(2,3)}}\Big) \leq \beta.$$
Since $\alpha X  \leq \beta$ is nontrivial and valid, we have $\beta > 0$ and this completes the proof.

\subsection{Size inequalities}
\label{appendix2}
In this section, we assume that an upper bound $\hat m$ on the size of the universe is available and we utilize $\hat m$ to improve our proposed LP relaxation.
As in Section~\ref{sec:lp}, consider a collection of $n$ objects each consisting of $d_i$, $i \in [n]$ elements; that is, we have a total number of $\bar d = \sum_{i \in [n]}{d_i}$ elements. Let $\N$ denote the set consisting of all these elements, \ie $\N= \cup_{i \in [n]}{\cup_{t \in [d_i]}{i_t}}$. Define $d_{\min} = \min_{i \in [n]} d_i$ and $d_{\max} = \max_{i \in [n]} d_i$, where we assume $d_{\max} \geq 2$.
Let $\hat m \in \{d_{\max}, \ldots ,\bar d-1\}$.
Notice that $\hat m= \bar d$ is a trivial upper bound and cannot be exploited to improve the relaxation.
\begin{proposition}
Consider a subset $\N' \subset \N$ of cardinality $\hat m + 1$.
Then the following inequality is valid for the feasible region of Problem~\eqref{ip}:
\begin{equation}\label{sizeineq}
\sum_{1 \leq i < j \leq n}{\sum_{\substack{t \in [d_i] q \in [d_j]: \\ i_t, j_q \in \N'}}{X_{tq}(i,j)}} \geq 1,
\end{equation}
where $i_t$ and $j_q$ denote the $t$-th element of object $\S_i$ and the $q$-th element of object $\S_j$, respectively.
\end{proposition}

\begin{proof}
To see the validity of inequality~\eqref{sizeineq},
suppose that $X_{tq}(i,j) = 0$ for all $i_t, j_q \in \N'$. Since by assumption $|\N'| = \hat m + 1$, it then follows that the size of the universe is at least $\hat m +1$,
which contradicts  with the assumption that $\hat m$ is an upper bound on the size of the universe.
\end{proof}

Henceforth, we refer to inequalities of the form~\eqref{sizeineq} for all
$\N' \subset \N$ as \emph{size inequalities}.
The following example demonstrates that size inequalities can tighten the proposed LP relaxation.

\begin{example}\label{eg3}
Let $n =3$ and $d_1 = d_2 = d_3 =2$; moreover, suppose that $\hat m = 3$.
Then it can be checked that the following is feasible for Problem~\eqref{lp2}:
\begin{equation}\label{ex3}
X(1,2) =
\begin{pmatrix}
0 & 0 \\
0 & 0
\end{pmatrix}, \qquad
X(1,3) =
\begin{pmatrix}
0 & 0 \\
0 & 1
\end{pmatrix}, \qquad
X(2,3) =
\begin{pmatrix}
1 & 0 \\
0 & 0
\end{pmatrix}.
\end{equation}
Now consider a size inequality obtained by letting $\N'=\{1_1, 1_2, 2_1, 2_2\}$ in inequality~\eqref{sizeineq}:
$$
X_{11}(1,2)+X_{12}(1,2)+X_{21}(1,2)+X_{22}(1,2) \geq 1.
$$
Substituting~\eqref{ex3} in the above inequality yields $0+0+0+0 \not\geq 1$.

Next, suppose that $\hat m = 4$  and consider the point:
\begin{equation}\label{ex4}
X(1,2) =
\begin{pmatrix}
0 & 0 \\
0 & 1
\end{pmatrix}, \qquad
X(1,3) =
\begin{pmatrix}
0 & 0 \\
0 & 0
\end{pmatrix}, \qquad
X(2,3) =
\begin{pmatrix}
0 & 0 \\
0 & 0
\end{pmatrix}.
\end{equation}
Again, it can be checked that~\eqref{ex4} is feasible for Problem~\eqref{lp2}. Consider now the size inequality obtained
by letting $\N'=\{1_1, 2_1, 2_2,3_1,3_2\}$:
$$
X_{11}(1,2)+ X_{12}(1,2)+X_{11}(1,3)+ X_{12}(1,3)+X_{11}(2,3)+ X_{12}(2,3)+X_{21}(2,3)+ X_{22}(2,3)\geq 1.
$$
Substituting~\eqref{ex4} in the above inequality yields $0+0+0+0+0+0+0+0 \not\geq 1$.
\end{example}

\begin{remark}\label{rem2}
Let us revisit the graph partitioning problem described in Remark~\ref{rem1}; namely, the problem of partitioning the nodes of a graph
into at most $K$ subsets~\cite{ChoRao93}. Let $\tilde N$ denote a subset of $[n]$ with $|\tilde N| = K +1$. Then the \emph{clique inequality}
associated with $\tilde N$ is defined as
\begin{equation}\label{clique1}
\sum_{i,j \in \tilde N: i < j}{y_{ij}} \geq 1.
\end{equation}
Now, let $t \in [d_{\min}]$ and consider a size inequality~\eqref{sizeineq} with $\N' = \{i_t: i \in I \subseteq [n]\}$, where $|I| = \hat m + 1$; that is, the inequality:
\begin{equation}\label{clique2}
\sum_{i,j \in I: i < j}{X_{tt}(i,j)} \geq 1.
\end{equation}
Comparing~\eqref{clique1} and~\eqref{clique2}, it follows that for a particular type of $\N'$, the corresponding size inequalities have the same form as
clique inequalities. However, we would like to remark that while a clique inequality~\eqref{clique1} associated with a clique
of size $r$ contains $\binom{r+1}{2}$ variables, a size inequality~\eqref{clique2} associated with $\N'$ of cardinality $r$ may contain different number of variables.

For instance, let $n = 5$, $d_i = 3$ for all $i \in \{1,\ldots,5\}$, and suppose that $\hat m = 4$. Then letting $\N'=\{1_1,1_2,1_3, 2_1, 2_2\}$ in~\eqref{sizeineq}, we obtain the size inequality
$$
X_{11}(1,2)+ X_{12}(1,2)+X_{21}(1,2)+ X_{22}(1,2)+X_{31}(1,2)+ X_{32}(1,2) \geq 1,
$$
consisting of six variables, while letting $\N'=\{1_1,2_1, 3_1, 4_1,4_2\}$, we obtain a size inequality
\begin{align*}
&X_{11}(1,2)+ X_{11}(1,3)+X_{11}(1,4)+ X_{12}(1,4)+X_{11}(2,3)+ X_{11}(2,4)+X_{12}(2,4)\\
+&X_{11}(3,4)+X_{12}(3,4) \geq 1,
\end{align*}
consisting of nine variables.
% and finally, letting $\N'=\{1_2,2_2, 3_2,4_2,5_2\}$ we obtain an inequality of the form~\eqref{clique1}
%\begin{align*}
%&X_{22}(1,2)+ X_{22}(1,3)+X_{22}(1,4)+ X_{22}(1,5)+X_{22}(2,3)+ X_{22}(2,4)+X_{22}(2,5)\\
%+&X_{22}(3,4)+X_{22}(3,5)+X_{22}(4,5) \geq 1,
%\end{align*}
%consisting of ten variables.

Hence size inequalities~\eqref{sizeineq} can be considered a generalization of clique inequalities for joint object matching. It is well-known that
the separation problem over clique inequalities~\eqref{clique1} is NP-hard~\cite{Eisen01}. Various heuristics for separating clique inequalities have been proposed in the literature~\cite{Eisen01} and similar ideas could be developed to efficiently separate over size inequalities. However, such a computational study is beyond the scope of this paper.
\end{remark}

\begin{remark}
Consider the following variant of joint object matching: find a collection of consistent partial maps $X(i,j) \in \{0,1\}^{d_i\times d_j}$ for all $1 \leq i < j \leq n$
corresponding to a universe of at most $\hat m$ elements so as to minimize~\eqref{objective}. It can be checked that this problem can be equivalently solved by solving the ILP obtained by adding all size inequalities~\eqref{sizeineq} to Problem~\eqref{ip}.
\end{remark}

\paragraph{Acknowledgements.} 
Antonio De Rosa has been partially supported by the NSF DMS Grant No.~1906451, the NSF DMS Grant No.~2112311, and the NSF DMS CAREER Award No.~2143124.

\bibliographystyle{plain}

\begin{footnotesize}

\end{footnotesize}

\end{document}